\documentclass[11pt,letterpaper]{amsart}
\usepackage{amsmath,amssymb,amsfonts,amsthm,amscd}
\usepackage{enumerate}
\usepackage{mathtools}
\usepackage[dvipsnames]{xcolor}
\usepackage[margin=1.15in]{geometry}
\usepackage{mathrsfs}
\usepackage{euscript}
\usepackage{stix2}

\usepackage[
colorlinks,
linktocpage,
linkcolor=red,
citecolor=cyan,
urlcolor=blue
]{hyperref}
\usepackage{tikz}
\usetikzlibrary{arrows.meta}

\allowdisplaybreaks
\theoremstyle{plain}
\newtheorem{thm}{Theorem}[section]
\newtheorem{lemma}[thm]{Lemma}
\newtheorem{prop}[thm]{Proposition}

\newtheorem{cor}[thm]{Corollary}
\theoremstyle{definition}
\newtheorem{dfn}[thm]{Definition}

\newtheorem*{thm*}{Theorem}

\newtheorem*{dfn*}{Definition}

\numberwithin{equation}{section}

\newcommand{\bR}{\mathbb{R}}
\newcommand{\bZ}{\mathbb{Z}}
\newcommand{\bN}{\mathbb{N}}

\newcommand{\SL}{\mathsf{SL}}
\newcommand{\SO}{\mathsf{SO}}

\newcommand{\sG}{\mathsf{G}}
\newcommand{\sH}{\Gamma_0}
\newcommand{\sK}{\mathsf{K}}
\newcommand{\Hom}{\mathrm{Hom}}

\newcommand{\diag}{\mathrm{diag}}

\newcommand{\cF}{\mathcal F}
\newcommand{\cC}{\mathcal C}
\newcommand{\cFt}{\mathcal F^{(2)}}
\newcommand{\cL}{\mathcal L}
\newcommand{\cD}{\mathcal D}
\newcommand{\cO}{\mathcal O}

\newcommand{\cG}{\mathcal G}
\newcommand{\cM}{\mathcal M}
\newcommand{\rQ}{\mathrm Q}

\newcommand{\dph}{d_\varphi}
\newcommand{\ii}{\mathbb{i}} 
\newcommand{\bd}{\partial_\infty\Gamma}
\newcommand{\bdd}{\partial^{(2)}_\infty\Gamma}
\newcommand{\Lam}{\Lambda(\Gamma)}
\DeclareMathOperator{\supp}{supp}
\DeclareMathOperator{\diam}{diam}

\DeclareMathOperator{\op}{\mathrm{op}}
\newcommand{\fk}{\mathfrak}
\newcommand{\La}{{\Lambda^{(2)}(\Gamma)}}
\newcommand{\qslash}{{\scriptstyle /}}

\title[Twisted Patterson--Sullivan densities for Anosov subgroups]{Twisted Patterson--Sullivan densities on nilpotent covers for Anosov subgroups}

\author{Krishnendu Gongopadhyay}

\address{Indian Institute of Science Education and Research (IISER) Mohali, Knowledge City,  Sector 81, S.A.S. Nagar 140306, Punjab, India}
\email{krishnendu@iisermohali.ac.in}

\author{Neelanjan Mondal}

\address{Indian Institute of Science Education and Research (IISER) Mohali, Knowledge City,  Sector 81, S.A.S. Nagar 140306, Punjab, India}
\email{mondalneelanjan@gmail.com}

\subjclass[2020]{Primary 22E40; Secondary 37A25, 37D40, 20F67, 28D05}

\keywords{Anosov subgroups, Patterson-Sullivan theory, twisted conformal densities, nilpotent covers,  critical exponents, Poincaré series, Hopf–Tsuji–Sullivan theorem.}
\date{19th October, 2026}

\thanks{}

\begin{document}
	
\begin{abstract}
	Let $\Gamma<\mathsf{SL}(d,\mathbb R)$ be a torsion-free,
	Zariski-dense Borel--Anosov subgroup, and let $\varphi$ be a linear functional positive on
	the limit cone of $\Gamma$. For every $\chi \in \Hom(\Gamma, \bR)$, we construct twisted $(\varphi,\chi)$-conformal densities and prove their existence, uniqueness and ergodicity. For every non-trivial normal subgroup $\Gamma_0\lhd\Gamma$ with nilpotent quotient, we classify the ergodic $\varphi$-conformal densities of $\Gamma_0$ in terms of characters of $\Gamma/\Gamma_0$. Up to scaling, they
	are naturally parametrized by
	$\operatorname{Hom}(\Gamma/\Gamma_0,\mathbb R)$, and their
	dimensions are precisely the critical exponents of the
	corresponding twisted Poincar\'e series.  We also prove a shadow lemma for twisted
	densities of normal subgroups, with estimates uniform
	over densities with fixed parameters, and a twisted
	Hopf--Tsuji--Sullivan theorem.
\end{abstract}


	\maketitle
	

	\section{Introduction}

Patterson--Sullivan theory relates the orbital growth of a discrete group
of hyperbolic isometries to conformal measures on its limit set
\cite{Patterson,Sullivan}. 
For a non-elementary convex cocompact subgroup
$\Gamma<\operatorname{Isom}(\mathbb H^n)$, the conformal density  supported on the limit set  is unique up
to scaling \cite{Sullivan}, and it plays a central role in the ergodic theory of the
geodesic flow and the horospherical dynamics on the unit tangent bundle of
$\mathbb{H}^n/\Gamma$. 
For an infinite normal subgroup
$\Gamma_0\lhd\Gamma$ of infinite index, the situation is very different.
In  each dimension, the conformal densities of $\Gamma_0$
form a convex cone which is in general  infinite dimensional, and its extremal
rays, 
which are precisely the \emph{ergodic} densities, encode the measure theory of the infinite cover $\mathbb{H}^n\qslash\Gamma$ \cite{Rob11, PPS}. 
Such densities are related to positive
eigenfunctions of the Laplacian \cite{Sul87} and to horospherically invariant Radon
measures on the infinite cover
$\mathbb H^n/\Gamma_0$
\cite{BL98,Rob03,Bab04}.

	In this rank one setting, when 
	$\Gamma/\Gamma_0$ is nilpotent, a striking
	rigidity appears. Every ergodic conformal density of $\Gamma_0$ supported
	on the limit set is, up to scaling, a conformal density of the ambient
	group $\Gamma$ twisted by a character of $\Gamma/\Gamma_0$, and every
	character arises in this way
	\cite[Theorem~1.10]{PPS}. Thus, the ergodic densities of $\Gamma_0$,
	which a priori may form an infinite dimensional family, are parametrized
	by the finite dimensional vector space
	$\operatorname{Hom}(\Gamma/\Gamma_0,\mathbb{R})$. This picture has its
	origins in the study of horocycle invariant Radon measures on abelian
	covers \cite{BL98,Sar04}, alongside related work on nilpotent and periodic covers \cite{Bab04,LS07}. Paulin, Pollicott and Schapira developed the theory of twisted Patterson
	densities for convex cocompact groups in pinched negative curvature with
	H\"older potentials, providing the framework for this classification.

In higher rank, Patterson-Sullivan theory was developed in general
symmetric spaces by Albuquerque and Quint
\cite{Albuquerque,Quint02a,Quint02b}, with further ergodic developments
due to Link \cite{Link06}.	 
Within this framework, Anosov subgroups
\cite{Lab,GW,KLPanosovcharacterizations,GGKW,BPS19} arise as natural
higher rank analogues of convex cocompact groups, and the Patterson
Sullivan theory in this setting has since been developed extensively
\cite{Sam,DK,BLLO,ELO,LO,LO24,Sam24,CZZ1,BCZZ2}.
	By contrast, much less is known for normal subgroups of infinite index.
	Their growth has recently been studied in connection with coamenability
	\cite{DOT} (see also \cite{CDST25}), but to our knowledge,  no classification of the ergodic conformal densities has been available in  higher rank.

\subsection*{The Higher-Rank Setting}\label{ss:setting}
	
	Let $\sG=\SL(d,\bR)$, with $d\geq 2$, and let $\sK=\SO(d)$.
	The Cartan involution $g\mapsto(g^{-1})^{\mathsf T}$ induces the
	Cartan decomposition
	$\fk{sl}_d(\bR)=\fk k\oplus\fk p$,
	$\fk k=\fk{so}(d)$,
	$\fk p=\operatorname{Sym}_0(d,\bR)$,
	where $\operatorname{Sym}_0(d,\bR)$ denotes the space of symmetric
	traceless matrices.
	We fix the Cartan subspace $\fk a\subset \fk p$ consisting of
	diagonal matrices with trace zero and the closed positive Weyl chamber
	$$\fk a^+=\left\{\diag(a_1,\ldots,a_d)\in\fk a:
	a_1\geq\cdots\geq a_d\right\}.$$
	The associated simple roots are
	$\Delta=\{\alpha_1,\cdots,\alpha_{d-1}\}$,
	where $\alpha_i(\diag(a_1, \cdots,a_d))=a_i-a_{i+1}$. 
	Every $g\in\sG$ admits a Cartan decomposition
	\[
	g=k_1\exp(\kappa(g))k_2,
	\qquad k_1,k_2\in\sK,
	\]
	where $\kappa(g)\in\fk a^+$ is uniquely determined and is called the
	Cartan projection of $g$. More explicitly,
	\[
	\kappa(g)
	=
	\diag\bigl(\log\sigma_1(g),\ldots,\log\sigma_d(g)\bigr),
	\]
	where $\sigma_i(g)$ is the $i$-th eigenvalue of
	$\sqrt{g^{\mathsf T}g}$, ordered so that $	\sigma_1(g)\geq\cdots\geq\sigma_d(g)>0.$
	Finally, let  $\cF(\bR^d)=\sG/\mathsf P$ be the full flag variety, where $\mathsf P$ is the Borel
	subgroup of upper triangular matrices.
	
	Anosov representations were introduced by Labourie~\cite{Lab}
	and subsequently generalized by Guichard and Wienhard~\cite{GW}. We use
	the equivalent characterization in terms of the linear growth of
	singular-value gaps \cite{KLPanosovcharacterizations, GGKW, BPS19}.
	Let $\Gamma$ be finitely generated  and $|\cdot|$ be the word length associated to a fixed finite generating set.  We say  $\Gamma<\sG$ is \emph{Borel--Anosov} if  there are constants $c,C>0$ such that
	\begin{equation}\label{eq:anosovgap}
		\alpha_i(\kappa(\gamma))\ \geq\ c\,\lvert\gamma\rvert-C
		\qquad(1\leq i\leq d-1,\ \gamma\in\Gamma).
	\end{equation}
 A Borel--Anosov subgroup  of $\sG$   is word hyperbolic \cite[Theorem 1.4]{KLPmorse} (see also \cite[Section 3]{BPS19}). A central feature of Borel Anosov subgroups is that they admit a
	$\Gamma$-equivariant embedding
	$	\theta\colon\partial_{\infty}\Gamma\to\mathcal F(\mathbb R^{d})$
	from the Gromov boundary $\partial_{\infty}\Gamma$ of $\Gamma$ to the
	full flag manifold $\mathcal F(\mathbb R^{d})$, which is moreover
	transverse (see \cite{BPS19, GGKW, KLPanosovcharacterizations}).
We denote by $\Lam:=\theta(\partial_\infty\Gamma)$
the \emph{limit set} of $\Gamma$.
	We define the  \emph{limit cone} of $\Gamma$  as
	the asymptotic cone of $\kappa(\Gamma)$
	\[
	\cL_\Gamma
	:=
	\left\{
	\lim_{n\to\infty} t_n\kappa(\gamma_n):
	t_n\to0^+,\ \gamma_n\in\Gamma
	\right\}. 
	\]
Write
\[
\cL_\Gamma^{\ast}
:=
\left\{
\varphi\in\fk a^\ast:
\varphi>0\ \text{on }\cL_\Gamma\smallsetminus\{0\}
\right\}.
\]
Throughout, we fix a torsion-free, Zariski-dense
Borel--Anosov subgroup $\Gamma<\mathsf G$, a finite symmetric
generating set $S$ for $\Gamma$, and
$\varphi\in\cL_\Gamma^{\ast}$.
	
		\subsection*{The Main Result}
		
In this article, we develop a twisted Patterson--Sullivan
theory for $\Gamma$ to prepare the ground for an explicit
description of the ergodic $\varphi$-conformal densities
of its normal subgroups.

For every real character $\chi$ of $\Gamma$, we construct
a twisted $(\varphi,\chi)$-conformal density supported on
$\Lambda(\Gamma)$ (Theorem~\ref{thm:existence}), with dimension
given by the corresponding twisted critical exponent
(Proposition~\ref{prop:rigidity}), and prove its uniqueness
  (Corollary~\ref{cor:twisted-uniqueness})
and ergodicity (Theorem~\ref{thm:HTS}).
This leads to the following natural question.

\medskip \noindent\textbf{Question.}
Does the Patterson--Sullivan measure of $\Gamma$  remain ergodic under the action of a normal
subgroup $\Gamma_0\lhd\Gamma$ with nilpotent quotient?

\medskip Our main theorem answers this question affirmatively
and, more generally, classifies the ergodic
$\varphi$-conformal densities of $\Gamma_0$, up to scaling,
in terms of the real characters of $\Gamma/\Gamma_0$.

	\begin{thm}\label{thm:B}
	Let $\Gamma<\sG$ be a Zariski-dense Borel--Anosov subgroup, and fix $\varphi\in\mathcal L_\Gamma^{\ast}$. Let $\Gamma_0\lhd\Gamma$ with $\Gamma/\Gamma_0$ nilpotent. Then, for every $\sigma>0$,
	\[
	\left\{
	\mu\in\cM_\varphi(\sigma,\Gamma_0):
	\Gamma_0\text{-action on }(\cF,\mu_o)\text{ is ergodic}
	\right\}
	=
	\bigcup_{\substack{\chi\in\Hom(\Gamma,\bR)\\
			\chi|_{\Gamma_0}=0}}
	\cM_{\varphi,\chi}(\sigma,\Gamma).
	\]
	Moreover, there is a bijection
	\[
	\left\{
	\mu:
	\begin{array}{l}
		\mu\in\cM_\varphi(\sigma,\Gamma_0)
		\text{ for some }\sigma>0,\\
		\Gamma_0\text{-action on }(\cF,\mu_o)\text{ is ergodic}
	\end{array}
	\right\}\big/\bR_{>0}
	\overset{1-1}{\xleftrightarrow{\qquad}}
	\Hom(\Gamma/\Gamma_0, \bR).
	\]
	In particular, the set of dimensions of ergodic $\varphi$-conformal densities of
	$\Gamma_0$ is exactly $\{\delta_{\varphi,\chi}(\Gamma):\chi\in \Hom(\Gamma/\Gamma_0, \bR)\}$.
\end{thm}

\subsection*{Beyond the Classification}
Motivated by the strategy in the rank-one work of \cite{PPS},
we develop a theory of twisted conformal densities and Poincar\'e
series in higher rank. Implementing the  strategy  in this setting requires geometric and
measure-theoretic estimates adapted to Cartan projections and
the flag limit set. Indeed, the ambient symmetric space is no
longer negatively curved, and the associated quasi-distance
need not be a symmetric metric, so the rank-one geometric inputs
cannot be applied directly. A key role is played by the
$\Gamma$-equivariant Anosov limit map, 
which identifies the Gromov boundary with the flag limit set.
Through this identification, we pull conformal measures back
to $\partial_\infty\Gamma$ and study them using word shadows.
We establish a Vitali covering lemma for these shadows using
the hyperbolicity of $\mathrm{Cay}(\Gamma,S)$, and combine
coarse additivity of Cartan projections with quantitative
Busemann cocycle estimates to prove a twisted shadow lemma
uniform over all densities with fixed parameters.

Towards the proof of our main result we  establish a twisted shadow lemma for conformal densities
of nonelementary normal subgroups $\Gamma_0\lhd\Gamma$,
with estimates uniform over
$\mu\in\cM_{\varphi,\chi}(\sigma,\Gamma_0)$ for fixed
$\sigma$ and $\chi$ (Theorem~\ref{thm:shadowprinciple}).
We also prove a differentiation theorem along word shadows:
for $f\in L^1(\mu)$, shadow averages converge to $f(\xi)$
for $\mu$-almost every $\xi\in\Lambda(\Gamma)$ along every
radial sequence in $\Gamma$ approaching $\xi$
(Theorem~\ref{thm:differentiation}). Our differentiation argument works directly with word shadows,
whose images in the flag variety $\cF$ need not be metric balls.

Our comparison principle shows that, for
$\mu,\nu\in\cM_{\varphi,\chi}(\sigma,\Gamma_0)$,
mass domination on $\Gamma o$ implies measure domination
at every $x\in\mathbb X$ (Theorem~\ref{thm:fatou}).  In particular, equality of these masses
determines the density uniquely. This principle also yields
rigidity for characters: when $\varphi$ is symmetric,
the equality $\delta_{\varphi,\chi}(\Gamma)=\delta_\varphi(\Gamma)$
forces $\chi=0$ (Corollary~\ref{prop:reversibility}).

We also prove a twisted Hopf--Tsuji--Sullivan theorem
(Theorem~\ref{thm:HTS}).

Finally, we apply this theory to critical exponents of infinite
normal subgroups. For amenable quotients, we obtain equality
of the untwisted critical exponents when $\varphi$ is symmetric
(Corollary~\ref{cor:amenable-cover}), together with a lower bound for arbitrary
infinite normal subgroups (Proposition~\ref{prop:normal-lower-bound}).


	\subsection{Plan of the paper}
	
	Section~\ref{sec:prelim} collects the geometric preliminaries.
	Section~\ref{sec:densities} develops the theory of twisted conformal densities.
	Section~\ref{sec:differentiation} establishes
	differentiation along word shadows and
	the comparison principle.
	Section~\ref{sec:HTS} constructs
	the twisted Bowen--Margulis--Sullivan measure and proves the twisted
	Hopf--Tsuji--Sullivan theorem.
	Section~\ref{sec:nilpotent} proves the classification for nilpotent covers
	and its application to the growth form.
	Appendix~\ref{app:amenable-covers} gives applications to critical exponents of
	normal subgroups and amenable covers.

\section{Preliminaries}\label{sec:prelim}

\subsection{Cartan projections and singular-value estimates}\label{ss:cartan}


Let $\sG=\SL(d,\bR)$ and $\sK=\SO(d)$. We write
\[
\begin{aligned}
	\fk a
	&:=\left\{\operatorname{diag}(a_1,\ldots,a_d):
	a_i\in\bR,\ \sum_{i=1}^d a_i=0\right\},\\
	\fk a^+
	&:=\left\{\operatorname{diag}(a_1,\ldots,a_d)\in\fk a:
	a_1\geq\cdots\geq a_d\right\}
\end{aligned}
\]
for the space of real diagonal matrices of trace zero
and its closed positive Weyl chamber, respectively. The restricted roots are $e_i-e_j$  for $i\neq j$, the simple roots
	are $\alpha_i(\diag(a_1, \cdots, a_d))=a_i-a_{i+1}$ for $1\leq i\leq d-1$, and the \emph{fundamental weights}
	are
	\[
	\omega_i(\diag(a_1, \cdots, a_d))\ :=\ a_1+\dots+a_i\qquad(1\leq i\leq d-1).
	\]
The weights 	$\{\omega_1,\dots,\omega_{d-1}\}$ form a basis of $\fk a^\ast$.  The Weyl group 
	$W=\mathfrak{S}_d$ permutes the diagonal entries, and its longest element $w_0$ reverses the
	order of the entries, and the \emph{opposition involution} is given by
	\[
	\mathbb{i}=-w_0:\ \fk a\to\fk a,\qquad
	\ii\big(\diag(a_1,\dots,a_d)\big)=\diag(-a_d,\dots,-a_1).
	\]
	It preserves $\fk a^+$ and  the Euclidean norm
	$\|\diag(a_1, \cdots,a_d)\|^2=\sum_{j=1}^d a_j^2$.
 A direct
	computation gives
	\begin{equation}\label{eq:weights-ii}
		\omega_i\circ\ii=\omega_{d-i}\qquad(1\leq i\leq d-1).
	\end{equation}
The Cartan projection is $\sK$-bi-invariant, and inversion acts on it
by opposition:
\begin{equation}\label{eq:kappa-inverse}
	\kappa(k_1gk_2)=\kappa(g),
	\qquad
	\kappa(g^{-1})=\ii\bigl(\kappa(g)\bigr)
	\quad
	(g\in\sG,\;k_1,k_2\in\sK).
\end{equation}

Let $\mathbb X=\sG/\sK$ be the associated Riemannian symmetric space
and let $o=e\sK$. We equip $\mathbb X$ with the $\sG$-invariant
Riemannian metric whose inner product at $o$, under the identification
$T_o\mathsf X\simeq\fk p$, is
$\langle X,Y\rangle_o=\operatorname{tr}(XY)$.
Thus, writing $\lVert X\rVert^2=\operatorname{tr}(X^2)$ for $X\in\fk p$, we have
\[
d_{\mathbb X}(\sK,g\sK)
=
\lVert\kappa(g)\rVert
=
\left(\sum_{i=1}^{d}\bigl(\log\sigma_i(g)\bigr)^2\right)^{1/2}.
\]
Equivalently, $\mathbb X$ may be identified with the space of
positive-definite symmetric matrices of determinant one via
$g\sK\mapsto gg^{\mathsf T}$.
By $\sG$-invariance of the metric on $\mathsf X$, it follows that
\begin{equation}
d_{\mathsf X}(g\sK,h\sK)
=
\bigl\|\kappa(g^{-1}h)\bigr\|
\qquad (g,h\in\sG).
\end{equation}
	The Cartan projection induces a \(\sG\)-invariant
	\(\fk a^+\)-valued distance on \(\mathbb X\), defined by
	\[
	\kappa (g\sK,h\sK):=\kappa(g^{-1}h).
	\]
	This is well defined by the \(\sK\)-bi-invariance of \(\kappa\), and it satisfies
	\[
	\kappa(h\sK,g\sK)=\ii\,\kappa(g\sK,h\sK),
	\qquad
	\kappa(o,g\sK)=\kappa(g).
	\]

We equip each $\Lambda^k\bR^d$ with the Euclidean structure induced
from $\bR^d$, and write $\lVert\cdot\rVert$ for the associated norm.
For $1\leq i\leq d-1$, let
$\Lambda^i:\sG\rightarrow
\mathrm{SL}\left({\binom{d}{i}},\bR\right)$ be the $i$th
exterior-power representation with highest weight $\omega_i$, and
write $\|\cdot\|_{\mathrm{op}}$ for the standard operator norm. By
\cite[Lem.~5.33(b)(i)]{BQbook} and
\eqref{eq:kappa-inverse},
\begin{equation}\label{eq:wedge-norm}
	\left\lVert\Lambda^i g\right\rVert_{\mathrm{op}}
	=\sigma_1(g)\cdots\sigma_i(g)
	=e^{\omega_i(\kappa(g))},
	\qquad
	\left\lVert\left(\Lambda^i g\right)^{-1}\right\rVert_{\mathrm{op}}
	=e^{\omega_i(\ii\kappa(g))}.
\end{equation}
For $g \in \sG$, the Jordan projection is the element $\mathsf{Jd}_g\in \fk a^+$, given by
\[
\mathsf{Jd}_g=\diag(\lambda_1(g), \cdots, \lambda_d(g))
\]
where $\lambda_1(g)\ge \cdots\ge \lambda_d(g)$ are the logarithms of the  modulus of the eigenvalues of $g$. We note that for $g \in \sG$ the following identity holds:
\begin{equation}\label{Eq:Jordan}
\lim_{n\to\infty}\frac{1}{n}\kappa(g^n)=\mathsf{Jd}_g.
\end{equation}
Moreover, similarly to the Cartan projection, we have the following relation: 
\begin{equation}\label{Eq:Jd-rel}
	\mathsf{Jd}_{g^{-1}} = \ii\,\mathsf{Jd}_g  
	\quad
	(g\in\sG).
\end{equation}

For \(0\leq k\leq d\), let \(\mathbb{Gr}_k(\bR^d)\) denote the
Grassmannian of \(k\)-dimensional subspaces of \(\bR^d\). For \(E\in\mathbb{Gr}_k(\bR^d)\), let
\(\mathbb V_E\in\Lambda^k\bR^d\) denote the exterior product of an
orthonormal basis of \(E\), and  this vector is uniquely determined by
\(E\) up to sign.
If \(E\in\mathbb{Gr}_i(\bR^d)\) and
\(F\in\mathbb{Gr}_{d-i}(\bR^d)\), then
\begin{equation}\label{eq:wedge-pairing}
	\bigl\lVert\mathbb V_E\wedge\mathbb V_F\bigr\rVert
	=
	\bigl\lvert
	\langle\mathbb V_E,\mathbb V_{F^\perp}\rangle
	\bigr\rvert,
\end{equation}
and the common value in \eqref{eq:wedge-pairing}  is non-zero if and only if
\(E\oplus F=\bR^d\)(see
\cite[1.6.1 and 1.7.8, pp.~24 and~34]{Federer}).


	\begin{lemma}\label{lem:sv}
		Let $g,h\in \sG$ and $1\leq i\leq d-1$. Then the following hold:
		\begin{enumerate}
			\item  $\omega_i(\kappa(gh))\leq\omega_i(\kappa(g))+\omega_i(\kappa(h))$.
			\item 
			$\lVert\kappa(gh)-\kappa(h)\rVert\leq\sqrt d\,\lVert\kappa(g)\rVert$ and
			$\lVert\kappa(gh)-\kappa(g)\rVert\leq\sqrt d\,\lVert\kappa(h)\rVert$.
		\end{enumerate}
	\end{lemma}
	
	\begin{proof}
			By \eqref{eq:wedge-norm} and the submultiplicativity of the
		operator norm,
		\[
		e^{\omega_i(\kappa(gh))}
		= \lVert \Lambda^i(gh)\rVert_{\op}
		\leq \lVert\Lambda^i g\rVert_{\op}
		\lVert\Lambda^i h\rVert_{\op}
		= e^{\omega_i(\kappa(g))+\omega_i(\kappa(h))},
		\]
		which proves~(1).
		
	For~(2), by \cite[Lemma~31.1]{CanaryAnosovNotes}, we have
	\(\sigma_j(g)\sigma_d(h)\leq\sigma_j(gh)
	\leq\sigma_j(g)\sigma_1(h)\) for every \(1\leq j\leq d\).
		Since \(\prod_{r=1}^d\sigma_r(h)=1\), we have
		\(\sigma_1(h)\geq1\geq\sigma_d(h)\), and hence
		\[
		\bigl|\log\sigma_j(gh)-\log\sigma_j(g)\bigr|
		\leq
		\max\{\log\sigma_1(h),-\log\sigma_d(h)\}
		\leq\lVert\kappa(h)\rVert .
		\]
	Summing the squares over
	\(j=1,\ldots,d\) yields
		\[
		\lVert\kappa(gh)-\kappa(g)\rVert
		\leq\sqrt d\,\lVert\kappa(h)\rVert .
		\]
			Finally, using \(\kappa(g^{-1})=\ii\kappa(g)\) and the fact that
		\(\ii\) is an isometry, we obtain
		\[
		\begin{aligned}
			\lVert\kappa(gh)-\kappa(h)\rVert
			&=\lVert\kappa(h^{-1}g^{-1})-\kappa(h^{-1})\rVert\\
			&\leq \sqrt d\,\lVert\kappa(g^{-1})\rVert
			=\sqrt d\,\lVert\kappa(g)\rVert.
		\end{aligned}
		\]
	\end{proof}
	
For \(g\in\sG\), write \(g=k_1\exp(\kappa(g))k_2\), with
\(k_1,k_2\in\sK\), and let \(V_i:=\langle e_1,\ldots,e_i\rangle\).
If \(\alpha_i(\kappa(g))>0\), equivalently
\(\sigma_i(g)>\sigma_{i+1}(g)\), then
\(U_i(g):=k_1V_i\) and \(S_i(g):=k_2^{-1}V_i\) are well defined. Fix
$\widetilde w_0:= \operatorname{antidiag}(1, -1, \cdots,(-1)^{d-1})
\in\mathsf{N}_{\sK}(\mathsf{A})$,
such that 
$\widetilde w_0\,\mathsf{Z}_{\sK}(\mathsf{A})=w_0$.
Since \(\widetilde w_0V_i=V_{d-i}^{\perp}\), the Cartan decomposition
\(g^{-1}=(k_2^{-1}\widetilde w_0)\exp(\ii\kappa(g))
(\widetilde w_0^{-1}k_1^{-1})\) gives
\begin{equation}\label{eq:S-vs-U}
	S_i(g)
	=(k_2^{-1}\widetilde w_0V_{d-i})^\perp
	=U_{d-i}(g^{-1})^\perp.
\end{equation}
If all the gaps $\alpha_i(\kappa(g))$, $1\leq i\leq d-1$, are positive,  then
\(U(g):=(U_1(g),\ldots,U_{d-1}(g))\in\cF\) is a complete flag.


Recall that the $\mathsf P_k$- and $\mathsf P_{d-k}$-Anosov conditions are
equivalent; see \cite[Rem.~3.1 and Props.~4.5, 4.9]{BPS19}. Thus, by
\cite[Cor.~32.4]{CanaryAnosovNotes}, if
$x\in\partial_\infty\Gamma$ and $\gamma_n\to x$ in
$\overline{\Gamma}:=\Gamma\cup\partial_\infty\Gamma$, then
$U(\gamma_n)$ is eventually defined and
\begin{equation}\label{eq:cartan-property}
	U(\gamma_n)\longrightarrow\theta(x)
	\qquad\text{in }\cF.
\end{equation}

\begin{lemma}\label{lem:product}
	Let $g,h\in\sG$, let $1\leq i\leq d-1$, and suppose that
	$\alpha_i(\kappa(h))>0$. Choose the signs of
	$\mathbb V_{S_i(h)}$ and $\mathbb V_{U_i(h)}$ compatibly so that
	$\Lambda^i h\,\mathbb V_{S_i(h)}
	=
	e^{\omega_i(\kappa(h))}\mathbb V_{U_i(h)}$.
	Then as operators on $\Lambda^i\bR^d$, we have
	$\Lambda^i h
	=
	e^{\omega_i(\kappa(h))}
	\big\langle\mathbb V_{S_i(h)},\cdot\big\rangle
	\mathbb V_{U_i(h)}
	+
	E_h$ with 
	$\lVert E_h\rVert_{\mathrm{op}}
	\leq
	e^{\omega_i(\kappa(h))-\alpha_i(\kappa(h))}$.
	Consequently, we have the following product estimate.
	\begin{equation}\label{eq:prod-two-sided}
	\begin{aligned}
		e^{\omega_i(\kappa(h))}
		\left(
		\bigl\lVert\Lambda^i g\mathbb V_{U_i(h)}\bigr\rVert
		-e^{-\alpha_i(\kappa(h))}
		\lVert\Lambda^i g\rVert_{\mathrm{op}}
		\right)
		&\leq\lVert\Lambda^i(gh)\rVert_{\mathrm{op}}\\
		&\leq
		e^{\omega_i(\kappa(h))}
		\left(
		\bigl\lVert\Lambda^i g\mathbb V_{U_i(h)}\bigr\rVert
		+e^{-\alpha_i(\kappa(h))}
		\lVert\Lambda^i g\rVert_{\mathrm{op}}
		\right).
	\end{aligned}
	\end{equation}
\end{lemma}

	\begin{proof}
		Set $N:=\binom{d}{i}$ and $T:=\Lambda^i h$, and write
		$\sigma_1(T)\geq\cdots\geq\sigma_N(T)$ for the singular values of $T$.
		By the singular-value formulas for exterior powers
		\cite[Appendix~A.5, (A.19)--(A.22)]{BPS19},
		\[
		\begin{aligned}
			\sigma_1(T)
			=\sigma_1(h)\cdots\sigma_i(h)
			=e^{\omega_i(\kappa(h))},\quad
			\sigma_2(T)
			=\sigma_1(h)\cdots\sigma_{i-1}(h)\sigma_{i+1}(h)
			=e^{\omega_i(\kappa(h))-\alpha_i(\kappa(h))}.
		\end{aligned}
		\]
		Since $\alpha_i(\kappa(h))
		>0$ and 
		$\sigma_1(T)>\sigma_2(T)$, 
		the compatible choice of signs gives
		\[
	T\mathbb V_{S_i(h)}
	=
	\sigma_1(h)\cdots\sigma_i(h)\mathbb V_{U_i(h)}
	=
	\sigma_1(T)\mathbb V_{U_i(h)}.
	\]
		By the singular-value decomposition
		\cite[Corollary~2.6.7, p.~154]{HJ13}, this pair extends to
		orthonormal singular bases $\{v_r\}_{r=1}^N$ and
		$\{u_r\}_{r=1}^N$ satisfying $v_1=\mathbb V_{S_i(h)}$,
		$u_1=\mathbb V_{U_i(h)}$, and $Tv_r=\sigma_r(T)u_r$.
		Consequently,
		\[
			T
		=
		\sum_{r=1}^N
		\sigma_r(T)\langle v_r,\cdot\rangle u_r
		=
		e^{\omega_i(\kappa(h))}
		\big\langle\mathbb V_{S_i(h)},\cdot\big\rangle
		\mathbb V_{U_i(h)}
		+E_h,\qquad 	E_h:=\sum_{r=2}^N\sigma_r(T)\langle v_r,\cdot\rangle u_r
		\]
		where $	\lVert E_h\rVert_{\mathrm{op}}
		\leq\sigma_2(T)
		=e^{\omega_i(\kappa(h))-\alpha_i(\kappa(h))}$.
		Applying $\Lambda^i g$ yields
		\[
		\Lambda^i(gh)
		=e^{\omega_i(\kappa(h))}
		\big\langle\mathbb V_{S_i(h)},\cdot\big\rangle
		\Lambda^i g\mathbb V_{U_i(h)}
		+\Lambda^i gE_h.
		\]
	Moreover,
	\[
	\begin{split}
		\lVert\Lambda^i gE_h\rVert_{\mathrm{op}}
		\leq
		\lVert\Lambda^i g\rVert_{\mathrm{op}}
		\lVert E_h\rVert_{\mathrm{op}}
		\leq
		e^{\omega_i(\kappa(h))-\alpha_i(\kappa(h))}
		\lVert\Lambda^i g\rVert_{\mathrm{op}}.
	\end{split}
	\]
		The triangle and reverse triangle inequalities therefore give
		\[
		\begin{aligned}
			e^{\omega_i(\kappa(h))}
			\left(
			\bigl\lVert\Lambda^i g\mathbb V_{U_i(h)}\bigr\rVert
			-e^{-\alpha_i(\kappa(h))}
			\lVert\Lambda^i g\rVert_{\mathrm{op}}
			\right)
			&\leq\lVert\Lambda^i(gh)\rVert_{\mathrm{op}}\\
			&\leq
			e^{\omega_i(\kappa(h))}
			\left(
			\bigl\lVert\Lambda^i g\mathbb V_{U_i(h)}\bigr\rVert
			+e^{-\alpha_i(\kappa(h))}
			\lVert\Lambda^i g\rVert_{\mathrm{op}}
			\right).
		\end{aligned}
		\]
	\end{proof}


		\subsection{Iwasawa cocycle and the Busemann function}\label{ss:cocycle}

Let $\cF$ denote the compact $\sG$-homogeneous space of complete flags
$\eta=(\eta_1\subset\cdots\subset\eta_{d-1})$ in $\bR^d$. The open
$\sG$-orbit $\cFt\subset\cF\times\cF$ consists of the transverse pairs,
characterized by $\xi_i\oplus\eta_{d-i}=\bR^d$ for 
$1\leq i\leq d-1.$

\begin{dfn}\label{def:iwasawa}
	For $g\in \sG$ and $\eta\in\cF$, define 
	\[
	\sigma:\sG \times\cF\longrightarrow\fk a,
	\qquad
	\omega_i\bigl(\sigma(g,\eta)\bigr)
	:=
	\log\left\lVert\Lambda^i g\,\mathbb V_{\eta_i}\right\rVert,
	\qquad (1\leq i\leq d-1).
	\]
\end{dfn}
	
	
	For the representation $\rho=\Lambda^i$, the highest weight is $\omega_i$ and the corresponding
	invariant line is $\Lambda^i\eta_i=\bR\mathbb V_{\eta_i}$. Hence
	\cite[Lem.~8.17]{BQbook} identifies the map $\sigma$ defined above with the
	Iwasawa cocycle of \cite[\S8.2]{BQbook} (see also
	\cite[\S6.1 and Lem.~6.4(i)]{Quint02b}). In particular, $\sigma$ is a
	continuous cocycle by \cite[Lem.~8.2]{BQbook}. We record this, together with
	the estimates needed below. 
	

\begin{lemma}\label{lem:sigma-props}
	For all $g,h\in\sG$ and $\eta\in\cF$, the following properties hold.
	\begin{enumerate}
		\item $\sigma(gh,\eta)=\sigma(g,h\eta)+\sigma(h,\eta)$. Moreover,
		$\sigma(k,\eta)=0$ for every $k\in\sK$.
		
		\item The map $\sigma$ is continuous.
		
		\item $\lVert\sigma(g,\eta)\rVert\leq
		\lVert\kappa(g)\rVert$.
		
		\item If all gaps of $\kappa(h)$ are positive, then, for every $i$,
		\[
		\Bigl\lvert
		\omega_i\bigl(\kappa(gh)\bigr)
		-\omega_i\bigl(\kappa(h)\bigr)
		-\omega_i\bigl(\sigma(g,U(h))\bigr)
		\Bigr\rvert
		\leq
		2e^{\,2\sqrt d\,\lVert\kappa(g)\rVert-\alpha_i(\kappa(h))}
		\]
		provided
		$e^{\,2\sqrt d\,\lVert\kappa(g)\rVert-\alpha_i(\kappa(h))}
		\leq\tfrac12$.
	\end{enumerate}
\end{lemma}
	
\begin{proof}
	Assertions~(1) and~(2) follow from the preceding discussion.
	For $k\in\sK$, the $\sK$-invariance gives
	\[
	\omega_i(\sigma(k,\eta))
	=\log\lVert\Lambda^i k\mathbb V_{\eta_i}\rVert=0
	\]
	for every $1\le i\le d-1$, and hence $\sigma(k,\eta)=0$. \textup{(3)} follows from
	\cite[Cor.~8.20(a)]{BQbook}. 
	For~(4), fix $i$ and set
	\[
	A_i:=\bigl\lVert\Lambda^i g\mathbb V_{U_i(h)}\bigr\rVert,
	\qquad
	t_i:=e^{-\alpha_i(\kappa(h))}
	\frac{\lVert\Lambda^i g\rVert_{\mathrm{op}}}{A_i}.
	\]
	Dividing~\eqref{eq:prod-two-sided} by
	$e^{\omega_i(\kappa(h))}A_i$ gives
	\[
	1-t_i\leq
	\frac{\lVert\Lambda^i(gh)\rVert_{\mathrm{op}}}
	{e^{\omega_i(\kappa(h))}A_i}
	\leq1+t_i.
	\]
	By Cauchy--Schwarz inequality,
\[
 \omega_i(\kappa(g)) =\Bigl|\sum_{j=1}^i \log\sigma_j(g)\Bigr| \leq\sqrt{i}\Bigl(\sum_{j=1}^i (\log\sigma_j(g))^2\Bigr)^{1/2} \leq\sqrt d\,\lVert\kappa(g)\rVert. 
 \]
	Hence, by~\eqref{eq:wedge-norm},
	\[
	t_i
	\leq
	e^{-\alpha_i(\kappa(h))
		+\omega_i(\kappa(g))
		+\omega_i(\ii\kappa(g))}
	\leq
	e^{2\sqrt d\,\lVert\kappa(g)\rVert-\alpha_i(\kappa(h))}
	=:\varepsilon_i.
	\]
	Under the stated assumption, $t_i\leq\varepsilon_i\leq\frac12$.
	Since $|\log(1\pm t)|\leq2t$ for $0\leq t\leq\frac12$, taking
	logarithms and using~\eqref{eq:wedge-norm} and Definition~\ref{def:iwasawa} yields
	\[
	\begin{aligned}
		\bigl|\omega_i(\kappa(gh))-\omega_i(\kappa(h))
		-\omega_i(\sigma(g,U(h)))\bigr|
		&\leq 2t_i\\
		&\leq
		2e^{2\sqrt d\,\lVert\kappa(g)\rVert-\alpha_i(\kappa(h))},
	\end{aligned}
	\]
	as required.
	
	\end{proof}

	\begin{dfn}\label{def:busemann}
	For $\eta\in\cF$, define
	\[
	\beta_\eta\colon \mathbb X\times \mathbb X\longrightarrow\fk a,
	\qquad
	\beta_\eta(g\sK,h\sK)
	:=\sigma(g^{-1},\eta)-\sigma(h^{-1},\eta).
	\]
\end{dfn}

By Lemma~\ref{lem:sigma-props}\textup{(1)}, the right-hand side remains
unchanged if $g$ and $h$ are replaced by $gk_1$ and $hk_2$,
respectively, for any $k_1,k_2\in\sK$. Hence, $\beta_\eta$ is well
defined.

\begin{lemma}\label{lem:beta-props}
	For all $g,g_1,g_2,g_3\in\sG$ and $\eta\in\cF$, the following hold.
	\begin{enumerate}
		\item  $\beta_\eta(g_1\sK,g_3\sK)
		=\beta_\eta(g_1\sK,g_2\sK)+\beta_\eta(g_2\sK,g_3\sK),$ and  $	\beta_{g\eta}(gg_1\sK,gg_2\sK)=\beta_\eta(g_1\sK,g_2\sK).$
		\item  $\beta_\eta(g\sK,h\sK)
		=\sigma(g^{-1}h,h^{-1}\eta)$
		\item $\lVert\beta_\eta(g\sK,h\sK)\rVert
		\leq\lVert \kappa(g\sK,h\sK)\rVert$.
		\item  The map $\eta\longmapsto \beta_\eta(g\sK, h\sK)$ is continuous on $\cF$.
	\end{enumerate}
\end{lemma}

The assertions follow immediately from
Definition~\ref{def:busemann} and
Lemma~\ref{lem:sigma-props}\textup{(1)}--\textup{(3)}.

\subsection{The Gromov product}

\begin{dfn}\label{def:gromov}
	Define
	\[
	\cG:\cFt\longrightarrow\fk a,\qquad
	\omega_i\bigl(\cG(\xi,\eta)\bigr)
	:=-\log\bigl\lVert
	\mathbb V_{\xi_i}\wedge\mathbb V_{\eta_{d-i}}
	\bigr\rVert,
	\quad (1\leq i\leq d-1).
	\]
\end{dfn}


This is the vector-valued Gromov product of
\cite[\S8.10]{BPS19}; see also \cite[\S4]{Sam}. The map $\cG$ is
well defined and continuous on $\cF^{(2)}$, and
$\omega_i(\cG(\xi,\eta))\geq0$ for every $1\leq i\leq d-1$.

\begin{lemma}\label{lem:gromov-equiv}
	For every $g\in\sG$ and $(\xi,\eta)\in\cFt$,
	\[
	\cG(g\xi,g\eta)
	=
	\cG(\xi,\eta)+\sigma(g,\xi)+\ii\sigma(g,\eta).
	\]
\end{lemma}

\begin{proof}
	This is \cite[Lem.~4.12]{Sam}, in the above conventions.
\end{proof}


	\subsection{The Finsler quasi-distance and coarse additivity}
	\label{sec:shadows}
	
	Write
	\begin{equation}\label{Eq:d}
	\dph(\gamma_1,\gamma_2)
	:=
	\varphi\bigl(\kappa(\gamma_1^{-1}\gamma_2)\bigr)
	=
	\varphi\bigl(\kappa(\gamma_1\sK,\gamma_2\sK)\bigr),
	\qquad \gamma_1,\gamma_2\in\Gamma.
	\end{equation}
	Thus, $\dph(e,\gamma)=\varphi(\kappa(\gamma))$ and $\dph$ is left
	invariant, that is, $\dph(g\gamma_1,g\gamma_2)=\dph(\gamma_1,\gamma_2)$. Moreover, writing $\varphi^\ii:=\varphi\circ\ii$, \eqref{eq:kappa-inverse} gives
	\begin{equation}\label{eq:dphi-flip}
	d_{\varphi^\ii}(\gamma_1,\gamma_2)
	=d_\varphi(\gamma_2,\gamma_1).
	\end{equation}
	We also set
	\begin{equation}\label{eq:L}
	L(R):=
	\max\bigl\{\lVert\kappa(\gamma)\rVert:
	\gamma\in\Gamma,\ |\gamma|\leq R\bigr\}
	\leq C_\Gamma R+C_\Gamma .
	\end{equation}
	
\begin{lemma}\label{lem:dphi-qi}
	The following assertions hold.
	\begin{enumerate}
		\item There exist $A\geq1$ and $B\geq0$, depending only on
		$(\Gamma,S,\varphi)$, such that, for all $\gamma_1,\gamma_2\in\Gamma$,
		\[
		A^{-1}d_w(\gamma_1,\gamma_2)-B
		\leq \dph(\gamma_1,\gamma_2)
		\leq A\,d_w(\gamma_1,\gamma_2)+B.
		\]
		
		\item For all $\gamma_1,\gamma_2,\gamma_3\in\Gamma$,
		\[
		\begin{aligned}
			\bigl|\dph(\gamma_1,\gamma_3)-\dph(\gamma_2,\gamma_3)\bigr|
			&\leq
			\sqrt d\,\lVert\varphi\rVert
			\lVert\kappa(\gamma_1^{-1}\gamma_2)\rVert,\\
			\bigl|\dph(\gamma_3,\gamma_1)-\dph(\gamma_3,\gamma_2)\bigr|
			&\leq
			\sqrt d\,\lVert\varphi\rVert
			\lVert\kappa(\gamma_1^{-1}\gamma_2)\rVert.
		\end{aligned}
		\]
		In particular, both left-hand sides are at most
		$\sqrt d\,\lVert\varphi\rVert
		L(d_w(\gamma_1,\gamma_2))$.
	\end{enumerate}
\end{lemma}
	
\begin{proof}
	Assertion \textup{(1)} is precisely \cite[Lem.~4.5]{DOT}.
	For \textup{(2)}, left invariance reduces to the case $\gamma_3=e$.
	Since
	$\gamma_1^{-1}=(\gamma_1^{-1}\gamma_2)\gamma_2^{-1}$,
	Lemma~\ref{lem:sv}\textup{(2)} gives
	\[
	\begin{aligned}
		\bigl|\dph(\gamma_1,e)-\dph(\gamma_2,e)\bigr|
		&=
		\bigl|\varphi\bigl(
		\kappa(\gamma_1^{-1})-\kappa(\gamma_2^{-1})
		\bigr)\bigr|\\
		&\leq
		\lVert\varphi\rVert
		\bigl\lVert
		\kappa(\gamma_1^{-1})-\kappa(\gamma_2^{-1})
		\bigr\rVert\\
		&\leq
		\sqrt d\,\lVert\varphi\rVert
		\bigl\lVert\kappa(\gamma_1^{-1}\gamma_2)\bigr\rVert.
	\end{aligned}
	\]
The second inequality follows similarly by applying
Lemma~\ref{lem:sv}\textup{(2)} to
\(\gamma_1=\gamma_2(\gamma_2^{-1}\gamma_1)\), and using
\eqref{eq:kappa-inverse} together with the \(\iota\)-invariance of the norm.
The final bound in both cases follows from the definition of \(L\) in
\eqref{eq:L}.
\end{proof}

	\begin{prop}[{\cite[Cor.~5.15]{LO};
			see also \cite[Prop.~4.6]{DOT}}]
		\label{prop:coarse-add}
		There exists $D_0=D_0(\Gamma,S)\geq0$ such that, whenever
		$\gamma_2$ lies on a word geodesic from $\gamma_1$ to $\gamma_3$,
		\[
		\bigl\lVert
		\kappa(\gamma_1^{-1}\gamma_3)
		-\kappa(\gamma_1^{-1}\gamma_2)
		-\kappa(\gamma_2^{-1}\gamma_3)
		\bigr\rVert
		\leq D_0.
		\]
		In particular, with $D_\varphi:=\lVert\varphi\rVert D_0$,
		\[
		\bigl|
		\dph(\gamma_1,\gamma_3)
		-\dph(\gamma_1,\gamma_2)
		-\dph(\gamma_2,\gamma_3)
		\bigr|
		\leq D_\varphi.
		\]
	\end{prop}
	
	\begin{proof}
		The geodesic assumption gives
		$|\gamma_1^{-1}\gamma_3|
		=|\gamma_1^{-1}\gamma_2|+|\gamma_2^{-1}\gamma_3|$; hence the first
		assertion follows from \cite[Cor.~5.15]{LO}, and the second follows
		by applying $\varphi$.
	\end{proof}

We next realize the cocycles introduced in Subsection~\ref{ss:cocycle}
as asymptotic differences of Cartan projections.

\begin{prop}\label{prop:buslimit}
	Let $x\in\bd$ and let $(\gamma_n)$ be a sequence in $\Gamma$ converging
	to $x$ in $\overline{\Gamma}$. Then, for every $\gamma\in\Gamma$,
	\[
	\kappa(\gamma^{-1}\gamma_n)-\kappa(\gamma_n)
	\longrightarrow
	\sigma\bigl(\gamma^{-1},\theta(x)\bigr)
	\qquad(n\to\infty).
	\]
	Consequently, for every $\gamma_1,\gamma_2\in\Gamma$,
	\[
	\kappa(\gamma_1\sK,\gamma_n\sK)
	-\kappa(\gamma_2\sK,\gamma_n\sK)
	\longrightarrow
	\beta_{\theta(x)}(\gamma_1\sK,\gamma_2\sK)
	\qquad(n\to\infty).
	\]
\end{prop}

\begin{proof}
	Since $\gamma_n\to x$ in $\overline{\Gamma}$, we have
	$|\gamma_n|\to\infty$; hence \eqref{eq:anosovgap} gives
	\[
	\min_{1\leq i\leq d-1}
	\alpha_i\bigl(\kappa(\gamma_n)\bigr)
	\longrightarrow\infty.
	\]
	Fix $\gamma\in\Gamma$. By
	Lemma~\ref{lem:sigma-props}\textup{(4)}, for every $i$ and all
	sufficiently large $n$,
	\[
	\left|
	\omega_i\!\left(
	\kappa(\gamma^{-1}\gamma_n)-\kappa(\gamma_n)
	-\sigma(\gamma^{-1},U(\gamma_n))
	\right)
	\right|
	\leq
	2e^{\,2\sqrt d\,\lVert\kappa(\gamma^{-1})\rVert
		-\alpha_i(\kappa(\gamma_n))}
	\longrightarrow0.
	\]
	By \eqref{eq:cartan-property} and
	Lemma~\ref{lem:sigma-props}\textup{(2)}, we obtain 
	\[
	\sigma(\gamma^{-1},U(\gamma_n))
	\longrightarrow
	\sigma(\gamma^{-1},\theta(x))
	\qquad(n\to\infty).
	\]
	Together with the preceding estimates, this proves the first assertion.
Consequently, applying the first assertion with $\gamma=\gamma_j$, $j=1,2$, we obtain, by
Definition~\ref{def:busemann},
\[
\begin{aligned}
	\kappa(\gamma_1\sK,\gamma_n\sK)-\kappa(\gamma_2\sK,\gamma_n\sK)
	&\longrightarrow
	\sigma(\gamma_1^{-1},\theta(x))
	-\sigma(\gamma_2^{-1},\theta(x)) \qquad(n\to\infty),
\end{aligned}
\]
where the right hand side is equal to $\beta_{\theta(x)}(\gamma_1\sK,\gamma_2\sK).$
\end{proof}

\begin{prop}\label{Prop:busctrl}
	For every $R\geq 0$, set
	\[
	A_R:=D_\varphi+2\sqrt{d}\,\lVert\varphi\rVert\,L(R).
	\]
	Then, for all $\gamma_1,\gamma_2\in\Gamma$ and every
	$x\in\cO_R(\gamma_1,\gamma_2)$,
	\[
	\left|
	\varphi\bigl(\beta_{\theta(x)}
	(\gamma_1\sK,\gamma_2\sK)\bigr)
	-\dph(\gamma_1,\gamma_2)
	\right|
	\leq A_R.
	\]
\end{prop}

\begin{proof}
	Choose a geodesic ray $c$ from $\gamma_1$ to $x$ meeting
	$B_w(\gamma_2,R)$, and let $\gamma:=c(t_0)$ be such that
	$d_w(\gamma,\gamma_2)\leq R$. Set $\gamma_n:=c(n)$ for $n\geq t_0$, so
	that $\gamma_n\to x$. 
	Applying $\varphi$ to the second assertion of Proposition~\ref{prop:buslimit} and using \eqref{Eq:d}, we obtain
	\[
	\dph(\gamma_1,\gamma_n)-\dph(\gamma_2,\gamma_n)
	\longrightarrow
	\varphi\bigl(\beta_{\theta(x)}
	(\gamma_1\sK,\gamma_2\sK)\bigr)
	\qquad(n\to\infty).
	\]
	Fix $n\geq t_0$. Since $\gamma$ lies on the geodesic segment from
	$\gamma_1$ to $\gamma_n$, Proposition~\ref{prop:coarse-add} gives
	\[
	\left|
	\dph(\gamma_1,\gamma_n)-\dph(\gamma_1,\gamma)-\dph(\gamma,\gamma_n)
	\right|
	\leq D_\varphi.
	\]
	Moreover, Lemma~\ref{lem:dphi-qi}\textup{(2)} and
	$d_w(\gamma,\gamma_2)\leq R$ yield
	\[
	\left|
	\dph(\gamma_2,\gamma_n)-\dph(\gamma,\gamma_n)
	\right|
	\leq \sqrt{d}\,\lVert\varphi\rVert\,L(R)
	\]
	and
	\[
	\left|
	\dph(\gamma_1,\gamma)-\dph(\gamma_1,\gamma_2)
	\right|
	\leq \sqrt{d}\,\lVert\varphi\rVert\,L(R).
	\]
	Combining these estimates, we obtain
	\[
	\left|
	\dph(\gamma_1,\gamma_n)-\dph(\gamma_2,\gamma_n)
	-\dph(\gamma_1,\gamma_2)
	\right|
	\leq A_R.
	\]
	By Proposition~\ref{prop:buslimit}, letting $n\to\infty$ in the preceding estimate yields the assertion.
\end{proof}

\subsection{Word shadows}


By enlarging the hyperbolicity constant if necessary, we fix
$\delta_0\geq 1$ such that $\operatorname{Cay}(\Gamma,S)$ is
$\delta_0$-hyperbolic in the sense of
\cite[Definition~III.H.1.20]{BH} and every geodesic triangle in
$\mathrm{Cay}(\Gamma,S)$ is $\delta_0$-slim; such a choice is possible by
\cite[Propositions~III.H.1.17 and III.H.1.22]{BH}.

For $\gamma\in\Gamma$ and $R\geq0$, define the word-metric ball
$B_w(\gamma,R)
:=
\left\{
\gamma'\in\Gamma:
d_w(\gamma',\gamma)\leq R
\right\}.$

\begin{dfn}\label{def:word-shadows}
	For $\gamma_1,\gamma_2\in\Gamma$ and $R\geq0$, define the shadow
	\[
	\cO_R(\gamma_1,\gamma_2)
	:=
	\left\{
	x \in\bd:
	\text{some geodesic ray $[\gamma_1,x)$ meets
		$B_w(\gamma_2,R)$}
	\right\}.
	\]
\end{dfn}


The following properties of word shadows follow from the definition and a standard compactness argument.

\begin{lemma}\label{lem:shadow-basic}
	Let $\gamma_1,\gamma_2\in\Gamma$, $R\geq0$.
	\begin{enumerate}
		\item $\cO_R(\gamma_1,\gamma_2)$ is a closed subset of $\bd$, non-decreasing in $R$, and
		$\gamma\,\cO_R(\gamma_1,\gamma_2)=\cO_R(\gamma\gamma_1,\gamma\gamma_2)$ for every $\gamma\in\Gamma$.
		\item  If $c$ is a geodesic ray from $\gamma_1$ to $x\in\bd$,
		then $x\in\cO_R(\gamma_1,c(t))$ for all $t\geq0$, $R\geq0$.
		\item If $d_w(\gamma_1,\gamma_2)\leq R$ then $\cO_R(\gamma_1,\gamma_2)=\bd$.
	\end{enumerate}
\end{lemma}


	\begin{lemma}
		\label{lem:multiplicity}
		For every $R\geq 0$ there exists $M_R\geq 1$ such that, for all
		$m\in\bR$ and $x\in\bd$,
		\[
		\left|\left\{\gamma\in\Gamma:
		x\in\cO_R(e,\gamma),\
		m\leq\dph(e,\gamma)<m+1
		\right\}\right|
		\leq M_R.
		\]
	\end{lemma}

	\begin{proof}
		Suppose $x\in\cO_R(e,\gamma)\cap\cO_R(e,\gamma')$ with
		$\dph(e,\gamma),\dph(e,\gamma')\in[m,m+1)$. Choose geodesic rays
		$c,c'$ from $e$ to $x$ such that $c$ meets $B_w(\gamma,R)$ at
		$\gamma_1=c(t_1)$ and $c'$ meets $B_w(\gamma',R)$ at
		$\widetilde{\gamma}_2=c'(t_2)$. Since $c$ and $c'$ have the same initial point and the same endpoint at infinity, by \cite[Lem.~III.H.3.3(1)]{BH}, we have
		$d_w\bigl(c(t_2),c'(t_2)\bigr)\leq 2\delta_0$.
		Setting $\gamma_2:=c(t_2)$, we obtain
		\[
		d_w(\gamma_2,\gamma')
		\leq d_w(\gamma_2,\widetilde{\gamma}_2)
		+d_w(\widetilde{\gamma}_2,\gamma')
		\leq 2\delta_0+R.
		\]
		Thus, $c$ meets $B_w(\gamma',2\delta_0+R)$ at $\gamma_2$. By
		Lemma~\ref{lem:dphi-qi}\textup{(2)}, with
		$q(R):=\sqrt d\,\lVert\varphi\rVert\,L(2\delta_0+R)$,
		we have
		\[
		\dph(e,\gamma_1), \dph(e, \gamma_2)\in
		\bigl[m-q(R),m+1+q(R)\bigr].
		\]
		Assume $t_1\leq t_2$, the other case being identical. Then $\gamma_1$
		lies on the geodesic $c([0,t_2])$ from $e$ to $\gamma_2$, so
		Proposition~\ref{prop:coarse-add} gives
		\[
		\dph(\gamma_1,\gamma_2)
		\leq \dph(e,\gamma_2)-\dph(e,\gamma_1)+D_\varphi
		\leq 1+2q(R)+D_\varphi.
		\]
		Hence, by Lemma~\ref{lem:dphi-qi}\textup{(1)},
		\[
		d_w(\gamma_1,\gamma_2)
		\leq A\bigl(1+2q(R)+D_\varphi+B\bigr)
		=:C(R).
		\]
		Consequently, we obtain
		\[
		d_w(\gamma,\gamma')
		\leq R+C(R)+R+2\delta_0.
		\]
		Thus, all such $\gamma'$ lie in the ball of radius
		$2R+2\delta_0+C(R)$ about $\gamma$, and we may take
		\[
		M_R:=\left|B_w\bigl(e,2R+2\delta_0+C(R)\bigr)\right|.
		\]
	\end{proof}

\subsection{Geometric estimates for shadows}

For $\gamma_0,\gamma_1,\gamma_2\in\Gamma$, define
\[
(\gamma_1\mid\gamma_2)_{\gamma_0}
:=
\frac{1}{2}\Bigl(
d_w(\gamma_0,\gamma_1)
+d_w(\gamma_0,\gamma_2)
-d_w(\gamma_1,\gamma_2)
\Bigr);
\]
see \cite[Def.~III.H.1.19]{BH}. Set
$\overline{\Gamma}:=\Gamma\cup\bd$. Following
\cite[Def.~III.H.3.15]{BH}, for $x,y\in\overline{\Gamma}$ and
$\gamma_0\in\Gamma$, define
\begin{equation}\label{Eq:Ext}
(x\mid y)_{\gamma_0}
:=
\sup
\liminf_{n,m\to\infty}
(\gamma_n\mid\gamma_m')_{\gamma_0},
\end{equation}
where the supremum is taken over all sequences $(\gamma_n)$ and
$(\gamma_m')$ in $\Gamma$ converging respectively to $x$ and $y$ in
$\overline{\Gamma}$. If either $x$ or $y$ lies in $\Gamma$, the
corresponding sequence is understood to be constant; see
\cite[Rem.~III.H.3.17(3)]{BH}.

By \cite[Def.~III.H.3.20 and Prop.~III.H.3.21]{BH}, we fix a visual
metric $\varrho$ on $\bd$,  at $e$, such that, for some
$\varepsilon_0>0$ and $C_0\geq1$,
\begin{equation}\label{eq:visual}
	C_0^{-1}e^{-\varepsilon_0(x\mid y)_e}
	\leq \varrho(x,y)
	\leq C_0e^{-\varepsilon_0(x\mid y)_e},
	\qquad x,y\in\bd.
\end{equation}

\begin{lemma}\label{lem:concentration}
	There exists $c_2=c_2(\delta_0)\geq0$ such that, for all
	$\gamma_1,\gamma_2\in\Gamma$, $R\geq0$, and
	$x,y\in\bd\smallsetminus\cO_R(\gamma_1,\gamma_2)$, one has
	$(x\mid y)_{\gamma_2}\geq R-c_2$. Consequently,
	\[
	\sup_{\gamma\in\Gamma}
	\diam_\varrho\bigl(\bd\smallsetminus\cO_R(\gamma,e)\bigr)
	\leq
	C_0e^{-\varepsilon_0(R-c_2)}
	\longrightarrow 0
	\qquad (R\to\infty),
	\]
	with the convention $\diam_\varrho\emptyset=0$.
\end{lemma}

\begin{proof}
If $x,y\notin\cO_R(\gamma_1,\gamma_2)$, choose
$\gamma_n\in[\gamma_1,x)$ such that $\gamma_n\to x$. Then
$d_w(\gamma_2,[\gamma_1,\gamma_n])>R$. By the proof that \textup{(1)} implies \textup{(3)} in
\cite[Prop.~III.H.1.17]{BH} and the observation following
\cite[Def.~III.H.1.19]{BH}, one has
 $(\gamma_n\mid\gamma_1)_{\gamma_2}\geq R-4\delta_0$. Hence, 
 $(x\mid\gamma_1)_{\gamma_2}\geq R-4\delta_0$ by \eqref{Eq:Ext}. The same argument
 with $y$ in place of $x$, we obtain
 $(y\mid\gamma_1)_{\gamma_2}
 \geq R-4\delta_0$. Applying \cite[Rem.~III.H.3.17(4)]{BH}, we obtain the first assertion. The second assertion
 follows from \eqref{eq:visual} by taking $\gamma_1=\gamma$ and
 $\gamma_2=e$.
\end{proof}

	\begin{lemma}\label{lem:shrinking}
		For every $\gamma\in\Gamma$ and $R\geq 0$, if
		$x,y\in\cO_R(e,\gamma)$, then
		$(x\mid y)_e\geq\lvert\gamma\rvert-R-3\delta_0$. Consequently,
		$\diam_\varrho\cO_R(e,\gamma)\leq
		C_0e^{-\varepsilon_0(\lvert\gamma\rvert-R-3\delta_0)}$.
	\end{lemma}

	
	\begin{proof}
		Choose a geodesic ray $[e,x)$ meeting $B_w(\gamma,R)$ at $u$ and a sequence $\gamma_n\in[e,x)$ converging to $x$ such that $u\in[e,\gamma_n]$ for every $n$. Since $d_w(u,\gamma)\leq R$, we have $|u|\geq|\gamma|-R$, while $d_w(\gamma_n,\gamma)\leq d_w(\gamma_n,u)+R=|\gamma_n|-|u|+R$. Consequently, $(\gamma_n\mid\gamma)_e\geq\frac12(|u|+|\gamma|-R)\geq|\gamma|-R$, and hence $(x\mid\gamma)_e\geq|\gamma|-R$ by \eqref{Eq:Ext}. Similarly, $(y\mid\gamma)_e\geq|\gamma|-R$. Thus, \cite[Rem.~III.H.3.17(4)]{BH} yields $(x\mid y)_e\geq\min\{(x\mid\gamma)_e,(\gamma\mid y)_e\}-2\delta_0\geq|\gamma|-R-2\delta_0\geq|\gamma|-R-3\delta_0$. The diameter estimate follows from \eqref{eq:visual}.
	\end{proof}

	\begin{lemma}\label{lem:engulfing}
		There exists $R_1=R_1(\delta_0)\geq1$ such that, for all $R\geq R_1$ and $\gamma_1,\gamma_2\in\Gamma$, if $\cO_R(e,\gamma_1)\cap\cO_R(e,\gamma_2)\neq\emptyset$ and $\lvert\gamma_2\rvert\geq\lvert\gamma_1\rvert$, then $\cO_R(e,\gamma_2)\subseteq\cO_{6R}(e,\gamma_1)$.
	\end{lemma}

\begin{proof}
	Fix $x\in\cO_R(e,\gamma_1)\cap\cO_R(e,\gamma_2)$ and
	$y\in\cO_R(e,\gamma_2)$. By the proof of
	Lemma~\ref{lem:shrinking}, we obtain
	$(x\mid\gamma_1)_e\geq\lvert\gamma_1\rvert-R$ and
	$(x\mid\gamma_2)_e,(y\mid\gamma_2)_e\geq
	\lvert\gamma_2\rvert-R$. Applying
	\cite[Rem.~III.H.3.17(4)]{BH} twice and using
	$\lvert\gamma_2\rvert\geq\lvert\gamma_1\rvert$, we obtain
	$(\gamma_1\mid\gamma_2)_e\geq
	\lvert\gamma_1\rvert-R-2\delta_0$ and hence
	$(y\mid\gamma_1)_e\geq
	\lvert\gamma_1\rvert-R-4\delta_0$. Let $[\smash{e,y})$ be a geodesic ray and choose
	$\gamma_n\in[e,y)$ with $\gamma_n$ converging to $y$. By
	\cite[Rem.~III.H.3.17(5)]{BH},
	$\liminf_{n\to\infty}(\gamma_n\mid\gamma_1)_e\geq
	\lvert\gamma_1\rvert-R-6\delta_0$. Moreover, the proof of
	\cite[Prop.~III.H.1.17, (1)$\Rightarrow$(3)]{BH}, together
	with the observation following
	\cite[Def.~III.H.1.19]{BH}, gives
	$d_w(\gamma_1,[e,\gamma_n])\leq
	(e\mid\gamma_n)_{\gamma_1}+4\delta_0
	=\lvert\gamma_1\rvert-(\gamma_n\mid\gamma_1)_e+4\delta_0$.
	Letting $n\to\infty$, we obtain 
	 $d_w(\gamma_1,[e,y))\leq R+10\delta_0$. Therefore, for
	$R\geq R_1:=2\delta_0$, one has
	$d_w(\gamma_1,[e,y))\leq6R$, and hence
	$y\in\cO_{6R}(e,\gamma_1)$. Since $y$ was arbitrary, the
	assertion follows.
\end{proof}

	The next lemma is a Vitali covering lemma for word shadows.
	
	\begin{lemma}\label{lem:vitali}
		Let $R\geq R_1$ and $I\subseteq\Gamma$. Then there exists
		$J\subseteq I$ such that the shadows
		$\{\cO_R(e,\gamma)\}_{\gamma\in J}$ are pairwise disjoint and
		\[
		\bigcup_{\gamma\in I}\cO_R(e,\gamma)
		\subseteq
		\bigcup_{\gamma\in J}\cO_{6R}(e,\gamma).
		\]
	\end{lemma}
	
	\begin{proof}
	If $I=\varnothing$, take $J=\varnothing$. Otherwise, since
	word balls are finite, we may enumerate
	$I=\{\gamma_1,\gamma_2,\ldots\}$ with
	$$|\gamma_1|\leq|\gamma_2|\leq\cdots,$$ using a finite
	enumeration when $I$ is finite. We then construct $J$
	inductively by including $\gamma_n$ if and only if
	$\cO_R(e,\gamma_n)$ is disjoint from $\cO_R(e,\gamma_m)$
	for every $m<n$ with $\gamma_m\in J$. Thus, the selected
	shadows are pairwise disjoint.
	Now, if $\gamma_n\notin J$, the selection rule gives
	$m<n$ with $\gamma_m\in J$ such that
	$\cO_R(e,\gamma_n)\cap\cO_R(e,\gamma_m)\neq\varnothing$.
	Since $|\gamma_m|\leq|\gamma_n|$, Lemma~\ref{lem:engulfing}
	therefore yields
	$\cO_R(e,\gamma_n)\subseteq\cO_{6R}(e,\gamma_m)$.
	For $\gamma_n\in J$, the same inclusion holds with $m=n$.
	Thus, taking unions, we obtain
	\[
	\bigcup_{\gamma\in I}\cO_R(e,\gamma)
	\subseteq
	\bigcup_{\gamma\in J}\cO_{6R}(e,\gamma),
	\]
	as required.
	\end{proof}
	
	\section{Twisted conformal densities}\label{sec:densities}
	
	Throughout this section, let $\Gamma<\sG$ be a Borel--Anosov subgroup. Let $\sH\lhd\Gamma$ be a
	non-elementary normal subgroup and let $\chi\in\Hom(\sH,\bR)$.
	We assume that $\chi$ is invariant under the conjugation action of
	$\Gamma$ on $\sH$, that is,
	\begin{equation}\label{eq:conj-inv}
		\chi(\gamma \gamma_0\gamma^{-1})=\chi(\gamma_0)
		\qquad
		(\gamma\in\Gamma,\ \gamma_0\in \sH).
	\end{equation}
	This condition is automatic whenever $\chi$ is the restriction to $\sH$
	of a character of $\Gamma$, in particular, it holds when $\chi=0$.

	\begin{dfn}\label{def:twisted}
		Let $\delta>0$. A \emph{twisted $(\varphi,\chi)$-conformal density of dimension
			$\delta$ for $\sH$} is a family $\nu=(\nu_x)_{x\in \mathbb X}$ of finite non-zero Borel measures
		on $\cF$, all supported on $\Lam$, such that for all $x,y\in \mathbb X$, $\gamma_0\in \sH$:
		\begin{align}
			\nu_x&=e^{-\delta\,\varphi(\beta_\bullet(x,y))}\,\nu_y
			&&\text{(conformality)},\label{eq:conf}\\
			(\gamma_0)_\ast\nu_x&=e^{-\chi(\gamma_0)}\,\nu_{\gamma_0x}
			&&\text{(twisted equivariance)}.\label{eq:equiv}
		\end{align}
		We write $\cM_{\varphi,\chi}(\delta,\sH)$ for the set of such families. Denote 
		$\cM_\varphi(\delta,\sH):=\cM_{\varphi,0}(\delta,\sH)$ and call its elements \emph{$\varphi$-conformal densities of dimension
			$\delta$ for $\sH$}. 
	\end{dfn}

	
	For a finite Borel measure $\mu$ on $\cF$, we write
	$\lVert\mu\rVert:=\mu(\cF)$ for its total mass. Since each $\nu_x$ is supported on $\Lam$, we have
	$\lVert\nu_x\rVert=\nu_x(\Lam)$.
	
	Via the limit map $\theta$, we identify $\partial_\infty\Gamma$ with
	$\Lam$, and hence each shadow $\cO_R(\gamma_1,\gamma_2)$
	with its image in $\Lambda_\Gamma$. Thus, for every Borel measure
	$\mu$ on $\cF$ supported on $\Lambda_\Gamma$, we write
	$\mu(\cO_R(\gamma_1,\gamma_2))
:=\mu(\theta(\cO_R(\gamma_1,\gamma_2)))$.

	\begin{lemma}\label{lem:density-basic}
		Let $\nu\in\cM_{\varphi,\chi}(\delta,\sH)$.
		\begin{enumerate}
			\item $\lVert\nu_{\gamma_0x}\rVert=e^{\chi(\gamma_0)}\lVert\nu_x\rVert$ for $\gamma_0\in\sH$, $x\in\mathbb X$.
			\item For every Borel $A\subseteq\cF$ and $x,y\in\mathbb X$,
			$\displaystyle \nu_x(A)=\int_A e^{-\delta\varphi(\beta_\xi(x,y))}\,d\nu_y(\xi)$.
			\item For $\gamma_0\in\sH$ and $x\in\mathbb X$,
			$\displaystyle \frac{d((\gamma_0)_\ast\nu_x)}{d\nu_x}(\xi)
			= e^{-\chi(\gamma_0)}\,e^{-\delta\varphi(\beta_\xi(\gamma_0x,\,x))}$, a positive continuous
			function of $\xi\in\cF$.
			\item $\supp\nu_x=\Lam$ for every $x\in\mathbb X$.
			\item $\lvert\log\lVert\nu_x\rVert-\log\lVert\nu_y\rVert\rvert
			\leq\delta\,\lVert\varphi\rVert\,\lVert\kappa(x,y)\rVert$ for all $x,y\in\mathbb X$.
		\end{enumerate}
	\end{lemma}
	
	\begin{proof}
		
(1) By \eqref{eq:equiv}, we have $\lVert\nu_x\rVert=\lVert (\gamma_0)_\ast\nu_x\rVert
=e^{-\chi(\gamma_0)}\lVert\nu_{\gamma_0x}\rVert$.
		
		(2) This is precisely the integral form of \eqref{eq:conf}.
		
	(3) Combining \eqref{eq:equiv} and \eqref{eq:conf}, we obtain
	$(\gamma_0)_\ast\nu_x=e^{-\chi(\gamma_0)}\nu_{\gamma_0x}
	=e^{-\chi(\gamma_0)}e^{-\delta\varphi(\beta_\bullet(\gamma_0x,x))}\nu_x$,
	which immediately yields the asserted Radon--Nikodym derivative. 
	
(4) By \eqref{eq:conf}, $\supp\nu_x$ is independent of $x$. It is a
non-empty closed subset of $\Lam$ and is $\sH$-invariant by
\eqref{eq:equiv}. Since $\theta\colon\bd\to\Lam$ is a
$\sH$-equivariant homeomorphism, $\theta^{-1}(\supp\nu_x)$ is a
non-empty closed $\sH$-invariant subset of $\bd$. By
\cite[Prop.~12.1(1) and Thm.~12.2(5)]{KB}, we obtain $\bd\subseteq\theta^{-1}(\supp\nu_x)$.
Therefore, $\supp\nu_x=\Lam$.
		
	(5) By (2) and Lemma~\ref{lem:beta-props}(3), we have
	\[
	e^{-\delta\lVert\varphi\rVert\lVert\kappa(x,y)\rVert}
	\leq
	\frac{\lVert\nu_x\rVert}{\lVert\nu_y\rVert}
	\leq
	e^{\delta\lVert\varphi\rVert\lVert\kappa(x,y)\rVert}.
	\]
	Taking logarithms proves (5).
	\end{proof}

\begin{dfn}
		For $\gamma\in\Gamma$ and $\nu\in\cM_{\varphi,\chi}(\delta,\sH)$, define
	\[
	\nu^{[\gamma]}_x
	:=\frac{1}{\lVert\nu_{\gamma o}\rVert}\,(\gamma^{-1})_\ast\nu_{\gamma x}
	\qquad (x\in\mathbb X).
	\]
\end{dfn}

\begin{lemma}\label{lem:translates}
	For $\gamma\in\Gamma$ and $\nu\in\cM_{\varphi,\chi}(\delta,\sH)$, $\nu^{[\gamma]}\in\cM_{\varphi,\chi}(\delta,\sH)$ and
	$\left\lVert\nu^{[\gamma]}_o\right\rVert=1$. Moreover, for every $R\geq0$ and
	$\gamma_1,\gamma_2\in\Gamma$,
	\begin{equation}\label{eq:fullness-translation}
		\frac{\nu_{\gamma_2o}\bigl(\cO_R(\gamma_1,\gamma_2)\bigr)}
		{\lVert\nu_{\gamma_2o}\rVert}
		=
		\nu^{[\gamma_2]}_o
		\bigl(\cO_R(\gamma_2^{-1}\gamma_1,e)\bigr).
	\end{equation}
\end{lemma}

\begin{proof}
	By definition and the $\Gamma$-invariance of $\Lam$, each
	$\nu^{[\gamma]}_x$ is a finite non-zero Borel measure supported on $\Lam$. 
	To verify conformality, let $x,y\in\mathbb X$ and $\xi\in\cF$. By the
	definition of $\nu^{[\gamma]}$, the conformality of $\nu$, we have
	\[
	\begin{aligned}
		\frac{d\nu^{[\gamma]}_x}{d\nu^{[\gamma]}_y}(\xi)
		=
		\frac{
			d\bigl(\lVert\nu_{\gamma o}\rVert^{-1}
			(\gamma^{-1})_\ast\nu_{\gamma x}\bigr)
		}{
			d\bigl(\lVert\nu_{\gamma o}\rVert^{-1}
			(\gamma^{-1})_\ast\nu_{\gamma y}\bigr)
		}(\xi)
		=
		\frac{d\bigl((\gamma^{-1})_\ast\nu_{\gamma x}\bigr)}
		{d\bigl((\gamma^{-1})_\ast\nu_{\gamma y}\bigr)}(\xi)
		=
		\frac{d\nu_{\gamma x}}{d\nu_{\gamma y}}(\gamma\xi)
		=
		e^{-\delta\varphi(\beta_{\gamma\xi}(\gamma x,\gamma y))},
	\end{aligned}
	\]
which is equal to $e^{-\delta\varphi(\beta_\xi(x,y))}$ by Lemma~\ref{lem:beta-props}(1).
For $\gamma_0\in\sH$, normality of $\sH$ and \eqref{eq:equiv} give
	\[
	\begin{aligned}
		(\gamma_0)_\ast\nu^{[g]}_x
		=
		\frac{1}{\lVert\nu_{go}\rVert}
		(g^{-1})_\ast(g\gamma_0g^{-1})_\ast\nu_{gx}
		=
		\frac{e^{-\chi(g\gamma_0g^{-1})}}{\lVert\nu_{go}\rVert}
		(g^{-1})_\ast\nu_{g\gamma_0x}
		=
		e^{-\chi(\gamma_0)}\nu^{[g]}_{\gamma_0x},
	\end{aligned}
	\]
	where the last equality follows from \eqref{eq:conj-inv}. Hence
	$\nu^{[g]}\in\cM_{\varphi,\chi}(\delta,\sH)$. Since push-forward
	preserves total mass,
	\[
	\left\lVert\nu^{[\gamma]}_o\right\rVert
	=
	\frac{\lVert(\gamma^{-1})_\ast\nu_{\gamma o}\rVert}
	{\lVert\nu_{\gamma o}\rVert}
	=1.
	\]
	Finally, by Lemma~\ref{lem:shadow-basic}(1),
	\[
	\begin{aligned}
		\nu^{[\gamma_2]}_o
		\bigl(\cO_R(\gamma_2^{-1}\gamma_1,e)\bigr)
		&=
		\frac{1}{\lVert\nu_{\gamma_2o}\rVert}
		\bigl((\gamma_2^{-1})_\ast\nu_{\gamma_2o}\bigr)
		\bigl(\cO_R(\gamma_2^{-1}\gamma_1,e)\bigr)\\
		&=
		\frac{\nu_{\gamma_2o}
			\bigl(\gamma_2\cO_R(\gamma_2^{-1}\gamma_1,e)\bigr)}
		{\lVert\nu_{\gamma_2o}\rVert}\\
		&=
		\frac{\nu_{\gamma_2o}
			\bigl(\cO_R(\gamma_1,\gamma_2)\bigr)}
		{\lVert\nu_{\gamma_2o}\rVert}.
	\end{aligned}
	\]
	This proves \eqref{eq:fullness-translation}.
\end{proof}

\begin{dfn}\label{def:normalised-class}
	We denote by $\cD_{\varphi,\chi}(\delta,\sH)$ the collection of
	normalised base-point measures arising from twisted
	$(\varphi,\chi)$-conformal densities:
	\[
	\cD_{\varphi,\chi}(\delta,\sH)
	:=
	\left\{
	\nu_o:\ 
	\nu\in\cM_{\varphi,\chi}(\delta,\sH),\
	\lVert\nu_o\rVert=1
	\right\}.
	\]
	For simplicity, we write
	$\cD:=\cD_{\varphi,\chi}(\delta,\sH)$.
\end{dfn}

Let $\mathcal P(\cF)$ denote the space of Borel probability measures
on $\cF$, endowed with the weak-$\ast$ topology.

\begin{lemma}\label{lem:compact-class}
	The set $\cD$ is closed in $\mathcal P(\cF)$, and hence compact.
\end{lemma}
	
	\begin{proof}
		Let
		$\nu_n=(\nu_{n,x})_{x\in\mathbb X}
		\in\cM_{\varphi,\chi}(\delta,\sH)$
		satisfy $\lVert\nu_{n,o}\rVert=1$, and suppose that
		$\nu_{n,o}\to\sigma$ weak-$\ast$ in $\mathcal P(\cF)$.
		Hence $\sigma$ is a probability
		measure with $\supp\sigma\subseteq\Lam$.
	 For $x\in\mathbb X$, define $\sigma_x
		:=
		e^{-\delta\varphi(\beta_\bullet(x,o))}\sigma.$
	It follows directly from Lemma~\ref{lem:beta-props}(4) that
	each $\sigma_x$ is a finite non-zero Borel measure and that $\supp\sigma_x=\supp\sigma\subseteq\Lam$.
		By the cocycle property in Lemma~\ref{lem:beta-props}(1), for all
		$x,y\in\mathbb X$,
		\[
		\sigma_x
		=
		e^{-\delta\varphi(\beta_\bullet(x,y))}
		e^{-\delta\varphi(\beta_\bullet(y,o))}\sigma
		=
		e^{-\delta\varphi(\beta_\bullet(x,y))}\sigma_y.
		\]
	Thus, $(\sigma_x)_{x\in\mathbb X}$
		satisfies \eqref{eq:conf}.
		 Since
		$\nu_n\in\cM_{\varphi,\chi}(\delta,\sH)$, combining
		\eqref{eq:conf} and \eqref{eq:equiv}, gives, for every
		$\gamma_0\in\sH$ and $f\in \mathcal C(\cF)$,
		\[
		\int_{\cF}f(\gamma_0\xi)\,d\nu_{n,o}(\xi)
		=
		\int_{\cF}
		f(\xi)e^{-\chi(\gamma_0)}
		e^{-\delta\varphi(\beta_\xi(\gamma_0o,o))}
		\,d\nu_{n,o}(\xi).
		\]
		Since both integrands are continuous,  passing to the weak-$\ast$
		limit, we obtain
		\[
		\begin{aligned}
			\int_{\cF}f(\xi)\,d((\gamma_0)_\ast\sigma)(\xi)
			&=
			\int_{\cF}f(\gamma_0\xi)\,d\sigma(\xi)\\
			&=
			\int_{\cF}
			f(\xi)e^{-\chi(\gamma_0)}
			e^{-\delta\varphi(\beta_\xi(\gamma_0o,o))}
			\,d\sigma(\xi)\\
			&=
			e^{-\chi(\gamma_0)}
			\int_{\cF}f(\xi)\,d\sigma_{\gamma_0o}(\xi).
		\end{aligned}
		\]
		Since this holds for every $f\in \mathcal C(\cF)$, it implies
			\[
		(\gamma_0)_\ast\sigma=e^{-\chi(\gamma_0)}\sigma_{\gamma_0o}
		\qquad (\gamma_0\in\sH).
		\]
		Using the equivariance of $\beta$ from
		Lemma~\ref{lem:beta-props}(1) and the conformality established above,
		 we obtain
		\[
		\begin{aligned}
			(\gamma_0)_\ast\sigma_x
			&=
			e^{-\delta\varphi(\beta_\bullet(\gamma_0x,\gamma_0o))}\,(\gamma_0)_\ast\sigma\\
			&=
			e^{-\chi(\gamma_0)}
			e^{-\delta\varphi(\beta_\bullet(\gamma_0x,\gamma_0o))}\sigma_{\gamma_0o}\\
			&=
			e^{-\chi(\gamma_0)}\sigma_{\gamma_0x},
		\end{aligned}
		\]
		for every $\gamma_0\in\sH$ and $x\in\mathbb X$.
		Hence $(\sigma_x)_{x\in\mathbb X}\in
		\cM_{\varphi,\chi}(\delta,\sH)$. Since
		$\beta_\bullet(o,o)=0$, we have $\sigma_o=\sigma$, and hence
		$\lVert\sigma_o\rVert=1$. Therefore $\sigma\in\cD$, proving that
		$\cD$ is weak-$\ast$ closed in $\mathcal P(\cF)$. Since $\cF$ is
		compact, $\mathcal P(\cF)$ is weak-$\ast$ compact, and consequently
		so is $\cD$.
	\end{proof}
	
	\subsection{Uniform fullness and the shadow principle}

	\begin{lemma}\label{lem:fullness}
		There exist
		$R_2=R_2(\varphi,\chi,\delta,\sH)\geq R_1$ and
		$m_0\in(0,1]$ such that, for every
		$\nu\in\cM_{\varphi,\chi}(\delta,\sH)$, every $R\geq R_2$,
		and all $\gamma_1,\gamma_2\in\Gamma$,
		\[
		\nu_{\gamma_2o}\bigl(\cO_R(\gamma_1,\gamma_2)\bigr)
		\geq
		m_0\lVert\nu_{\gamma_2o}\rVert.
		\]
	\end{lemma}

	\begin{proof}
		By \eqref{eq:fullness-translation} and
		Lemma~\ref{lem:translates}, it is enough to find
		$R_2\geq R_1$ and $m_0>0$ such that $\sigma\bigl(\cO_R(\gamma,e)\bigr)\geq m_0$
		for every $\sigma\in\cD$, $\gamma\in\Gamma$, and $R\geq R_2$.
		Suppose, to the contrary, that there exist sequences
		$\sigma_n\in\cD$, $\gamma_n\in\Gamma$, and $R_n\to\infty$ such that $\sigma_n\bigl(\cO_{R_n}(\gamma_n,e)\bigr)\rightarrow0.$
		If $\lvert\gamma_n\rvert\leq R_n$ for infinitely many $n$,
		then Lemma~\ref{lem:shadow-basic}(3) gives
		$\cO_{R_n}(\gamma_n,e)=\bd$, contradicting
		$\sigma_n(\cO_{R_n}(\gamma_n,e))\to0$. Therefore, $\lvert\gamma_n\rvert>R_n$ for all sufficiently large $n$.
		Set $K_n:=\bd\smallsetminus\cO_{R_n}(\gamma_n,e)$. 
		Since
		$\theta(K_n)=\Lambda_\Gamma\smallsetminus
		\theta(\cO_{R_n}(\gamma_n,e))$
		and $\sigma_n(\Lambda_\Gamma)=1$, we have
	\[
	\sigma_n\bigl(\theta(K_n)\bigr)
	=
1
	-\sigma_n\bigl(\cO_{R_n}(\gamma_n,e)\bigr)
	\longrightarrow 1
	\qquad (n\to\infty).
	\]
	Thus, we have
		$\sigma_n(\theta(K_n))>\frac12$ for all sufficiently large $n$.
		If $K_n=\varnothing$, then $\theta(K_n)=\varnothing$, and hence
		$\sigma_n(\theta(K_n))=0$, a contradiction. Therefore,
		$K_n\neq\varnothing$ for all sufficiently large $n$, and we may
		choose $q_n\in K_n$.
		By Lemma~\ref{lem:concentration},
		\[
		K_n\subseteq\overline B_\rho(q_n,\varepsilon_n),
		\qquad
		\varepsilon_n
		:=
		C_0e^{-\varepsilon_0(R_n-c_2)}
		\longrightarrow0.
		\]
		After passing to a subsequence, compactness of $\bd$ and
		Lemma~\ref{lem:compact-class} give $q_n\longrightarrow q\in\bd$ and $\sigma_n\longrightarrow\sigma\in\cD$  weak-$\ast$.
		Let $r>0$ and let $f\in \mathcal C(\cF)$ satisfy $0\leq f\leq1$ and
		$f=0$ on $\theta(\overline B_\rho(q,r))$. For all sufficiently
		large $n$, we have 
		$K_n
		\subseteq
		\overline B_\rho(q_n,\varepsilon_n)
		\subseteq
		\overline B_\rho(q,r).$
		Hence $f$ vanishes on $\theta(K_n)$, and therefore
		\[
		0\leq
		\int_{\cF} f\,d\sigma_n
		=
		\int_{\cO_{R_n}(\gamma_n,e)} f\,d\sigma_n
		\leq
		\sigma_n\bigl(\cO_{R_n}(\gamma_n,e)\bigr)
		\longrightarrow 0.
		\]
		The weak-$\ast$ convergence gives
		$\int_{\cF}f\,d\sigma=0$. Since $r$ and $f$ are arbitrary and
		$\supp\sigma\subseteq\Lam$, it follows that
		\[
		\supp\sigma\subseteq
		\bigcap_{r>0}\theta\bigl(\overline B_\rho(q,r)\bigr)
		=
		\{\theta(q)\}.
		\]
		Thus, $\sigma=\mathfrak D_{\theta(q)}$.
	Since $\sigma\in\cD$, Lemma~\ref{lem:density-basic}(3) gives
	$(\gamma_0)_*\sigma\sim\sigma$ for every $\gamma_0\in\sH$.
	Since $\sigma=\fk D_{\theta(q)}$, we have
	$(\gamma_0)_*\sigma=\fk D_{\gamma_0\theta(q)}$. Thus, the mutual absolute continuity
	of $\fk D_{h\theta(q)}$ and $\fk D_{\theta(q)}$ forces
	$\gamma_0\theta(q)=\theta(q)$. The equivariance and injectivity of $\theta$
	then give $\gamma_0q=q$ for every $\gamma_0\in\sH$, contradicting
\cite[Prop.~12.1(1), Thm.~12.2(1)]{KB}. This proves the lemma.
	\end{proof}

\begin{thm}\label{thm:shadowprinciple}
	Let $(\sH,\chi)$ be as above and let $\delta>0$. For every
	$R\geq R_2$, there exists
	$C_R=C_R(\varphi,\chi,\delta,\sH)\geq1$ such that, for every
	$\nu\in\cM_{\varphi,\chi}(\delta,\sH)$ and
	$\gamma_1,\gamma_2\in\Gamma$,
	\[
	C_R^{-1}\lVert\nu_{\gamma_2o}\rVert
	e^{-\delta\dph(\gamma_1,\gamma_2)}
	\leq
	\nu_{\gamma_1o}\bigl(\cO_R(\gamma_1,\gamma_2)\bigr)
	\leq
	C_R\lVert\nu_{\gamma_2o}\rVert
	e^{-\delta\dph(\gamma_1,\gamma_2)}.
	\]
	In particular, if $\sH=\Gamma$, then for every $\gamma\in\Gamma$,
	\begin{equation}\label{eq:shadow-character}
		C_R^{-1}e^{\chi(\gamma)-\delta\dph(e,\gamma)}
		\leq
		\frac{\nu_o\bigl(\cO_R(e,\gamma)\bigr)}
		{\lVert\nu_o\rVert}
		\leq
		C_Re^{\chi(\gamma)-\delta\dph(e,\gamma)}.
	\end{equation}
\end{thm}

\begin{proof}
	By Lemma~\ref{lem:density-basic}(2),
	\[
	\nu_{\gamma_1o}\bigl(\cO_R(\gamma_1,\gamma_2)\bigr)
	=
	\int_{\cO_R(\gamma_1,\gamma_2)}
	e^{-\delta\varphi(
		\beta_\xi(\gamma_1o,\gamma_2o))}
	\,d\nu_{\gamma_2o}(\xi).
	\]
	For every $\xi\in\cO_R(\gamma_1,\gamma_2)$,
	Proposition~\ref{Prop:busctrl} gives
	\[
	\begin{aligned}
		e^{-\delta A_R}e^{-\delta\dph(\gamma_1,\gamma_2)}
		\nu_{\gamma_2o}\bigl(\cO_R(\gamma_1,\gamma_2)\bigr)
		&\leq
		\nu_{\gamma_1o}\bigl(\cO_R(\gamma_1,\gamma_2)\bigr)\\
		&\leq
		e^{\delta A_R}e^{-\delta\dph(\gamma_1,\gamma_2)}
		\nu_{\gamma_2o}\bigl(\cO_R(\gamma_1,\gamma_2)\bigr).
	\end{aligned}
	\]
	By Lemma~\ref{lem:fullness},
	\[
	\nu_{\gamma_2o}\bigl(\cO_R(\gamma_1,\gamma_2)\bigr)
	\geq m_0\lVert\nu_{\gamma_2o}\rVert,
	\]
	while
	\[
	\nu_{\gamma_2o}\bigl(\cO_R(\gamma_1,\gamma_2)\bigr)
	\leq\lVert\nu_{\gamma_2o}\rVert.
	\]
	It follows that
	\[
	\begin{aligned}
		m_0e^{-\delta A_R}\lVert\nu_{\gamma_2o}\rVert
		e^{-\delta\dph(\gamma_1,\gamma_2)}
		&\leq
		\nu_{\gamma_1o}\bigl(\cO_R(\gamma_1,\gamma_2)\bigr)\\
		&\leq
		e^{\delta A_R}\lVert\nu_{\gamma_2o}\rVert
		e^{-\delta\dph(\gamma_1,\gamma_2)}\le m_0^{-1}e^{\delta A_R}\lVert\nu_{\gamma_2o}\rVert
		e^{-\delta\dph(\gamma_1,\gamma_2)}.
	\end{aligned}
	\]
	Thus, the first assertion follows with $	C_R:=m_0^{-1}e^{\delta A_R}.$
	If $\sH=\Gamma$, applying the first assertion with
	$\gamma_1=e$, $\gamma_2=\gamma$ and using
	$\lVert\nu_{\gamma o}\rVert
	=e^{\chi(\gamma)}\lVert\nu_o\rVert$ from
	Lemma~\ref{lem:density-basic}(1) gives
	\eqref{eq:shadow-character}.
\end{proof}

	\begin{cor}\label{cor:shadow-consequences}
		Under the assumptions of
		\textup{Theorem~\ref{thm:shadowprinciple}}, fix
		$\nu\in\cM_{\varphi,\chi}(\delta,\sH)$ and $R\geq R_2$.
		\begin{enumerate}
			\item 
			If $\sH=\Gamma$, then
			\[
			e^{\chi(\gamma)-\delta\dph(e,\gamma)}
			\leq C_R
			\qquad (\gamma\in\Gamma).
			\]
			
			\item
			For every $m\in\bR$,
			\[
			\sum_{\substack{\gamma\in\Gamma\\
					m\leq\dph(e,\gamma)<m+1}}
			\lVert\nu_{\gamma o}\rVert
			e^{-\delta\dph(e,\gamma)}
			\leq
			C_RM_R\lVert\nu_o\rVert,
			\]
			where $M_R$ is given by
			Lemma~\ref{lem:multiplicity}. In particular, if
			$\sH=\Gamma$, then
			\[
			\sum_{\substack{\gamma\in\Gamma\\
					m\leq\dph(e,\gamma)<m+1}}
			e^{\chi(\gamma)-\delta\dph(e,\gamma)}
			\leq C_RM_R.
			\]
			
			\item 
			For every $\gamma\in\Gamma$,
			\[
			C_R^{-1}e^{-\delta d_{\varphi^{\ii}}(e,\gamma)}
			\leq
			\frac{\lVert\nu_{\gamma o}\rVert}
			{\lVert\nu_o\rVert}
			\leq
			C_Re^{\delta\dph(e,\gamma)}.
			\]
			
			\item 
			For every $\gamma_1,\gamma_2\in\Gamma$,
			\[
			\nu_{\gamma_1o}
			\bigl(\cO_{6R}(\gamma_1,\gamma_2)\bigr)
			\leq
			C_{6R}C_R\,
			\nu_{\gamma_1o}
			\bigl(\cO_R(\gamma_1,\gamma_2)\bigr).
			\]
			
			\item 
			\[
			\sum_{\gamma\in\Gamma}
			\lVert\nu_{\gamma o}\rVert
			e^{-\delta\dph(e,\gamma)}
			=+\infty.
			\]
		\end{enumerate}
	\end{cor}
	
	\begin{proof}
		For \textup{(1)}, the lower bound in
		\eqref{eq:shadow-character} and the inequality
		$\nu_o(\cO_R(e,\gamma))\leq\lVert\nu_o\rVert$ give $e^{\chi(\gamma)-\delta\dph(e,\gamma)}
		\leq C_R$.
		
		For \textup{(2)}, Theorem~\ref{thm:shadowprinciple}, applied with
		$\gamma_1=e,\gamma_2=\gamma$  gives $\lVert\nu_{\gamma o}\rVert
		e^{-\delta\dph(e,\gamma)}
		\leq
		C_R\nu_o\bigl(\cO_R(e,\gamma)\bigr).$
		Summing over all $\gamma\in\Gamma$ satisfying
		$m\leq\dph(e,\gamma)<m+1$ and applying
		Lemma~\ref{lem:multiplicity}, we obtain
		\[
		\sum_{\substack{\gamma\in\Gamma\\
				m\leq\dph(e,\gamma)<m+1}}
		\lVert\nu_{\gamma o}\rVert
		e^{-\delta\dph(e,\gamma)}
		\leq
		C_RM_R\lVert\nu_o\rVert.
		\]
		If $\sH=\Gamma$, then
		$\lVert\nu_{\gamma o}\rVert
		=e^{\chi(\gamma)}\lVert\nu_o\rVert$ by
		Lemma~\ref{lem:density-basic}(1), which gives the particular case.
		
		For the lower bound in \textup{(3)}, apply
		Theorem~\ref{thm:shadowprinciple} with
		$\gamma_1=\gamma,\gamma_2=e$. Since
		$\nu_{\gamma o}(\cO_R(\gamma,e))
		\leq\lVert\nu_{\gamma o}\rVert$, we obtain
		\[
		\frac{\lVert\nu_{\gamma o}\rVert}
		{\lVert\nu_o\rVert}
		\geq
		C_R^{-1}e^{-\delta\dph(\gamma,e)}
		=
		C_R^{-1}e^{-\delta d_{\varphi^{\ii}}(e,\gamma)},
		\]
		where the equality follows from \eqref{eq:dphi-flip}.
		Applying the theorem with
		$\gamma_1=e,\gamma_2=\gamma$ and using
		$\nu_o(\cO_R(e,\gamma))\leq\lVert\nu_o\rVert$ similarly gives
		\[
		\frac{\lVert\nu_{\gamma o}\rVert}
		{\lVert\nu_o\rVert}
		\leq
		C_Re^{\delta\dph(e,\gamma)}.
		\]
		
		For \textup{(4)},  by applying  Theorem~\ref{thm:shadowprinciple} at the radii $6R$ and $R$,
		respectively, gives
		\[
		\begin{aligned}
			\nu_{\gamma_1o}\bigl(\cO_{6R}(\gamma_1,\gamma_2)\bigr)
			\leq
			C_{6R}\lVert\nu_{\gamma_2o}\rVert
			e^{-\delta\dph(\gamma_1,\gamma_2)}
			\leq
			C_{6R}C_R
			\nu_{\gamma_1o}\bigl(\cO_R(\gamma_1,\gamma_2)\bigr).
		\end{aligned}
		\]
		
	For (5), let $(\gamma_k)_{k\geq1}$ be an enumeration of $\Gamma$.
	For every $x\in\partial_\infty\Gamma$, choose a geodesic ray $c$ from
	$e$ to $x$. By Lemma~\ref{lem:shadow-basic}(2), we have 
	$x\in\cO_R(e,c(n))$ for every $n\geq0$. 
	Since the elements $c(n)$ are pairwise distinct, our identification
	of $\partial_\infty\Gamma$ with $\Lambda_\Gamma$ gives
	$\Lambda_\Gamma
	\subseteq
	\limsup_{k\to\infty}\cO_R(e,\gamma_k).$ 
	If
$\sum_{\gamma\in\Gamma}
\nu_o\bigl(\cO_R(e,\gamma)\bigr)<\infty,$
	then  the Borel--Cantelli gives $\nu_o\left(
	\limsup_{k\to\infty}\cO_R(e,\gamma_k)
	\right)=0,$
	contradicting
	$\nu_o(\Lambda_\Gamma)=\lVert\nu_o\rVert>0$. Therefore, we have 
	\[
	\sum_{\gamma\in\Gamma}
	\nu_o\bigl(\cO_R(e,\gamma)\bigr)=+\infty.
	\]
	Finally, Theorem~\ref{thm:shadowprinciple} gives
	$\nu_o\bigl(\cO_R(e,\gamma)\bigr)
	\leq
	C_R\lVert\nu_{\gamma o}\rVert
	e^{-\delta\dph(e,\gamma)},$
	and hence $	\sum_{\gamma\in\Gamma}
	\lVert\nu_{\gamma o}\rVert
	e^{-\delta\dph(e,\gamma)}
	=+\infty.$
	\end{proof}

	\subsection{Critical exponents}\label{ss:critical}
\begin{dfn}
	\label{def:twisted-critical-exponent}
	For $\chi\in\Hom(\Gamma,\bR)$, define
	\[
	Q_{\Gamma,\varphi,\chi}(s)
	:=
	\sum_{\gamma\in\Gamma}
	e^{\chi(\gamma)-s\,\dph(e,\gamma)}
\qquad
	\delta_{\varphi,\chi}(\Gamma)
	:=
	\limsup_{m\to\infty}
	\frac{1}{m}
	\log
	\sum_{\substack{\gamma\in\Gamma\\
			m-1<\dph(e,\gamma)\leq m}}
	e^{\chi(\gamma)}.
	\]
	We write
	$\delta_\varphi(\Gamma):=\delta_{\varphi,0}(\Gamma)$. For a
	subgroup $\sH\leq\Gamma$ and
	$\chi\in\Hom(\sH,\bR)$, the quantities
	$Q_{\sH,\varphi,\chi}(s)$,
	$\delta_{\varphi,\chi}(\sH)$, and
	$\delta_\varphi(\sH)$ are defined analogously by summing over
	$\gamma_0\in\sH$.
\end{dfn}

By Lemma~\ref{lem:dphi-qi}(1), there exist $A>0$ and $B\geq0$ such that 
$$\{\gamma\in\Gamma:m-1<\dph(e,\gamma)\leq m\}
\subseteq B_w\bigl(e,A(m+B)\bigr).$$
Since $\Gamma$ is finitely generated, every finite-radius ball in the word metric is finite, and hence so is each of the above annuli.

	
\begin{lemma}\label{lem:abscissa}
	Let $\chi\in\Hom(\Gamma,\bR)$. Then
	$Q_{\Gamma,\varphi,\chi}(s)<\infty$ whenever
	$s>\delta_{\varphi,\chi}(\Gamma)$, and
	$Q_{\Gamma,\varphi,\chi}(s)=+\infty$ whenever
	$s<\delta_{\varphi,\chi}(\Gamma)$. The same assertions hold for
	every subgroup $\sH\leq\Gamma$ and every
	$\chi\in\Hom(\sH,\bR)$. Moreover, for fixed
	$\gamma_1,\gamma_2\in\Gamma$, the value of
	$\delta_{\varphi,\chi}(\Gamma)$ is unchanged if
	$\dph(e,\gamma)$ is replaced by
	$\dph(\gamma_1,\gamma\gamma_2)$.
\end{lemma}

\begin{proof}
	Write
	\[
	S_m
	:=
	\sum_{\substack{\gamma\in\Gamma\\
			m-1<\dph(e,\gamma)\leq m}}
	e^{\chi(\gamma)}
	\qquad\text{and}\qquad
	\delta
	:=
	\limsup_{m\to\infty}\frac{1}{m}\log S_m.
	\]
	If $m-1<\dph(e,\gamma)\leq m$, then
	$	e^{-s\dph(e,\gamma)}
	\in
	\left[
	e^{-|s|}e^{-sm},
	e^{|s|}e^{-sm}
	\right].$
	Suppose first that $s>\delta$. Choose
	$\varepsilon\in(0,s-\delta)$. By the definition of the limsup,
	there exists $M\in\bN$ such that
	$S_m\leq e^{(\delta+\varepsilon)m}$ for every $m\geq M.$
	The preceding estimate gives
	\[
	\begin{aligned}
		\sum_{m\geq M}
		\sum_{\substack{\gamma\in\Gamma\\
				m-1<\dph(e,\gamma)\leq m}}
		e^{\chi(\gamma)-s\dph(e,\gamma)}
		\leq
		e^{|s|}
		\sum_{m\geq M}S_me^{-sm}
		\leq
		e^{|s|}
		\sum_{m\geq M}
		e^{-(s-\delta-\varepsilon)m}
		<\infty.
	\end{aligned}
	\]
	Moreover, Lemma~\ref{lem:dphi-qi}(1) gives
	\[
	\{\gamma\in\Gamma:\dph(e,\gamma)\leq M-1\}
	\subseteq
	B_w\bigl(e,A(M+B-1)\bigr).
	\]
	Since $\Gamma$ is finitely generated, the ball on the right is
	finite. Therefore,
	\[
	\sum_{\substack{\gamma\in\Gamma\\
			\dph(e,\gamma)\leq M-1}}
	e^{\chi(\gamma)-s\dph(e,\gamma)}
	<\infty.
	\]
	Combining the last two estimates yields
	$Q_{\Gamma,\varphi,\chi}(s)<\infty$. Now suppose that $s<\delta$. Choose
	$\varepsilon>0$ such that $\varepsilon \in (0, \delta-s)$. There exists a
	sequence $m_k\to\infty$ for which $S_{m_k}\geq e^{(\delta-\varepsilon)m_k}.$ Hence, we obtain
	\[
	\begin{aligned}
		\sum_{\substack{\gamma\in\Gamma\\
				m_k-1<\dph(e,\gamma)\leq m_k}}
		e^{\chi(\gamma)-s\dph(e,\gamma)}
		\geq
		e^{-|s|}S_{m_k}e^{-sm_k}
		\geq
		e^{-|s|}
		e^{(\delta-\varepsilon-s)m_k}
		\longrightarrow+\infty.
	\end{aligned}
	\]
	Since all terms of $Q_{\Gamma,\varphi,\chi}(s)$ are non-negative,
	\[
	Q_{\Gamma,\varphi,\chi}(s)
	\geq
	\sum_{\substack{\gamma\in\Gamma\\
			m_k-1<\dph(e,\gamma)\leq m_k}}
	e^{\chi(\gamma)-s\dph(e,\gamma)}
	\]
	for every $k$. Letting $k\to\infty$ gives
	$Q_{\Gamma,\varphi,\chi}(s)=+\infty$.
	The same argument applies verbatim to every subgroup
	$\sH\leq\Gamma$, with $\gamma$ replaced by $\gamma_0\in\sH$. 
	
	The last assertion follows from
	Lemma~\ref{lem:dphi-qi}(2), since the two distances differ by a bounded amount.
\end{proof}

\begin{prop}\label{prop:delta-props}
	Let $\chi\in\Hom(\Gamma,\bR)$ and
	$\varphi\in\cL^{\ast}_\Gamma$.
	\begin{enumerate}
		\item We have $\varphi^{\ii}\in\cL^{\ast}_\Gamma$ and
		$Q_{\Gamma,\varphi^{\ii},-\chi}(s)
		=Q_{\Gamma,\varphi,\chi}(s)$ for every $s\in\bR$. Consequently,
		$\delta_{\varphi^{\ii},-\chi}(\Gamma)
		=\delta_{\varphi,\chi}(\Gamma)$, and
		$(\Gamma,\varphi^{\ii},-\chi)$ is of divergence type if and only
		if $(\Gamma,\varphi,\chi)$ is.
		
		\item We have $0<\delta_{\varphi,\chi}(\Gamma)<+\infty$ and
		\[
		\left|
		\delta_{\varphi,\chi}(\Gamma)-\delta_\varphi(\Gamma)
		\right|
		\leq c_\chi A,
		\]
		where $c_\chi:=\max_{s\in S}\lvert\chi(s)\rvert$ and $A$ is the
		constant in Lemma~\ref{lem:dphi-qi}(1).
		
		\item If $\sH\leq\Gamma$, then
		$\delta_{\varphi,\chi|_{\sH}}(\sH)
		\leq\delta_{\varphi,\chi}(\Gamma)$.
		
		\item If $\varphi^{\ii}=\varphi$, then
		$\delta_{\varphi,-\chi}(\Gamma)
		=\delta_{\varphi,\chi}(\Gamma)
		\geq\delta_\varphi(\Gamma)$.
	\end{enumerate}
\end{prop}
	
	\begin{proof}
		(1) By \eqref{eq:kappa-inverse},
		$\ii(\cL_\Gamma)=\cL_\Gamma$, and hence
		$\varphi^{\ii}\in\cL^{\ast}_\Gamma$. Moreover,
		\[
		d_{\varphi^{\ii}}(e,\gamma^{-1})
		=
		\varphi^{\ii}\bigl(\kappa(\gamma^{-1})\bigr)
		=
		\varphi\bigl(\ii^2\kappa(\gamma)\bigr)
		=
		\dph(e,\gamma).
		\]
		Since 
		$\chi(\gamma^{-1})=-\chi(\gamma)$, reindexing the series gives
		\[
		\begin{aligned}
			Q_{\Gamma,\varphi^{\ii},-\chi}(s)
			&=
			\sum_{\gamma\in\Gamma}
			e^{-\chi(\gamma^{-1})
				-sd_{\varphi^{\ii}}(e,\gamma^{-1})}\\
			&=
			\sum_{\gamma\in\Gamma}
			e^{\chi(\gamma)-s\dph(e,\gamma)}
			=
			Q_{\Gamma,\varphi,\chi}(s).
		\end{aligned}
		\]
		The remaining assertions follow immediately.
		
		(2) Lemma~\ref{lem:dphi-qi}(1) gives $\dph(e,\gamma)\geq A^{-1}|\gamma|-B.$
		Therefore, for $s>A\log(|S|)$,
		\[
		\begin{aligned}
			Q_{\Gamma,\varphi,0}(s)
			&\leq
			e^{sB}
			\sum_{n=0}^{\infty}
			\left|\{\gamma\in\Gamma:|\gamma|=n\}\right|
			e^{-sn/A}\\
			&\leq
			e^{sB}
			\sum_{n=0}^{\infty}
			|S|^n e^{-sn/A}
			<\infty.
		\end{aligned}
		\]
		Hence $Q_{\Gamma,\varphi,0}(s)<\infty$, and
		Lemma~\ref{lem:abscissa} gives
		$\delta_\varphi(\Gamma)\leq s<+\infty$.
		Since $\chi$ is a homomorphism, we have 
		$	|\chi(\gamma)|
		\leq c_\chi|\gamma|
		\leq c_\chi A\bigl(\dph(e,\gamma)+B\bigr).$
		It follows that
		\[
		\begin{aligned}
			e^{-c_\chi AB}
			Q_{\Gamma,\varphi,0}(s+c_\chi A)
			\leq
			Q_{\Gamma,\varphi,\chi}(s)
			\leq
			e^{c_\chi AB}
			Q_{\Gamma,\varphi,0}(s-c_\chi A).
		\end{aligned}
		\]
		Lemma~\ref{lem:abscissa} now gives
		\[
		\delta_\varphi(\Gamma)-c_\chi A
		\leq
		\delta_{\varphi,\chi}(\Gamma)
		\leq
		\delta_\varphi(\Gamma)+c_\chi A.
		\]
		In particular,
		$\delta_{\varphi,\chi}(\Gamma)<+\infty$.
		It remains to establish positivity. Let $\mathsf N:=[\Gamma,\Gamma]$.
		Since $\Gamma$ is non-elementary, \cite[Thm.~12.2(1)]{KB} gives a
		free subgroup $\mathsf F_2\leq\Gamma$ of rank two.  Since
		$[\mathsf F_2,\mathsf F_2]\leq\mathsf N$ is infinite,
		$\mathsf N$ is infinite. As $\mathsf N\lhd\Gamma$,
		\cite[Thm.~12.2(5)]{KB} gives
		$\Lambda(\mathsf N)=\Lambda(\Gamma)=\bd$. In particular,
		$\mathsf N$ is non-elementary, so applying
		\cite[Thm.~12.2(1)]{KB} to $\mathsf N\leq\Gamma$ yields a
		free subgroup $\mathsf F:=\langle a,b\rangle\leq\mathsf N$ of rank two.
		Set $L:=\max\{|a|,|b|\}$. For each $n\geq1$, let
		\[
		W_n
		:=
		\left\{
		s_1\cdots s_n:
		s_j\in\{a,b\}
		\right\}.
		\]
		The sets $W_n$ are pairwise disjoint and
		$\lvert W_n\rvert=2^n$. Since $W_n\subseteq\mathsf N\subseteq\ker\chi$, we have
		$\chi(w)=0$ for every $w\in W_n$. Moreover,
		Lemma~\ref{lem:dphi-qi}(1), together with $|w|\leq nL$, gives
		\[
		\dph(e,w)
		\leq A|w|+B
		\leq ALn+B.
		\]
		Consequently, for
		$0<s<	\frac{\log 2}{AL}$,
		\[
		\begin{aligned}
			Q_{\Gamma,\varphi,\chi}(s)
			&\geq
			\sum_{n=1}^{\infty}
			\sum_{w\in W_n}
			e^{-s\dph(e,w)}\\
			&\geq
			e^{-sB}
			\sum_{n=1}^{\infty}
			\left(2e^{-sAL}\right)^n
			=
			+\infty.
		\end{aligned}
		\]
		Therefore, Lemma~\ref{lem:abscissa} yields
		$\delta_{\varphi,\chi}(\Gamma)
		\geq
		\frac{\log 2}{AL}
		>0.$
		
	(3) Since
	$Q_{\sH,\varphi,\chi|_{\sH}}(s)\leq
	Q_{\Gamma,\varphi,\chi}(s)$ for every $s\in\bR$, the assertion
	follows from Lemma~\ref{lem:abscissa}.
		
		(4) If $\varphi^{\ii}=\varphi$, then (1) gives $\delta_{\varphi,-\chi}(\Gamma)
		=
		\delta_{\varphi,\chi}(\Gamma).$
		On the other hand,
		\[
		\begin{aligned}
			Q_{\Gamma,\varphi,\chi}(s)
			+
			Q_{\Gamma,\varphi,-\chi}(s)
			=
			\sum_{\gamma\in\Gamma}
			\left(e^{\chi(\gamma)}+e^{-\chi(\gamma)}\right)
			e^{-s\dph(e,\gamma)}
			\geq
			2Q_{\Gamma,\varphi,0}(s).
		\end{aligned}
		\]
	By Lemma~\ref{lem:abscissa}, this yields
	$\max\left\{
	\delta_{\varphi,\chi}(\Gamma),
	\delta_{\varphi,-\chi}(\Gamma)
	\right\}
	\geq\delta_\varphi(\Gamma).$
	Together with the preceding equality, this proves (4).
	\end{proof}

	\begin{prop}\label{prop:rigidity}
		Let $\chi\in\Hom(\Gamma,\bR)$ and $\delta>0$. If
		$\cM_{\varphi,\chi}(\delta,\Gamma)\neq\varnothing$, then
		$\delta=\delta_{\varphi,\chi}(\Gamma)$ and
		$Q_{\Gamma,\varphi,\chi}
		\bigl(\delta_{\varphi,\chi}(\Gamma)\bigr)=+\infty$.
		In particular, $(\Gamma,\varphi,\chi)$ is of divergence type.
	\end{prop}

	\begin{proof}
		Fix $R=R_2$ and
		$\nu\in\cM_{\varphi,\chi}(\delta,\Gamma)$. By
		Corollary~\ref{cor:shadow-consequences}(2) and
		Lemma~\ref{lem:density-basic}(1), for every $m\in\bZ$,
		\[
		\sum_{\substack{\gamma\in\Gamma\\
				m\leq\dph(e,\gamma)<m+1}}
		e^{\chi(\gamma)-\delta\dph(e,\gamma)}
		\leq C_RM_R.
		\]
		By Lemma~\ref{lem:dphi-qi}(1), choose $m_-\in\bZ$ such that
		$\dph(e,\gamma)\geq m_-$ for every $\gamma\in\Gamma$.
		For $s>\delta$, summing over the annuli gives
		\[
		\begin{aligned}
			Q_{\Gamma,\varphi,\chi}(s)
			&=
			\sum_{m\geq m_-}
			\sum_{\substack{\gamma\in\Gamma\\
					m\leq\dph(e,\gamma)<m+1}}
			e^{\chi(\gamma)-\delta\dph(e,\gamma)}
			e^{-(s-\delta)\dph(e,\gamma)}\\
			&\leq
			C_RM_R\sum_{m\geq m_-}e^{-(s-\delta)m}
			<\infty.
		\end{aligned}
		\]
	Hence 
	Lemma~\ref{lem:abscissa} implies that
	$\delta_{\varphi,\chi}(\Gamma)\leq\delta$. On the other hand,
		Corollary~\ref{cor:shadow-consequences}(5) and
		Lemma~\ref{lem:density-basic}(1) give
		\[
		Q_{\Gamma,\varphi,\chi}(\delta)
		=
		\frac{1}{\lVert\nu_o\rVert}
		\sum_{\gamma\in\Gamma}
		\lVert\nu_{\gamma o}\rVert
		e^{-\delta\dph(e,\gamma)}
		=+\infty.
 		\]
		Lemma~\ref{lem:abscissa} therefore gives
		$\delta_{\varphi,\chi}(\Gamma)\geq\delta$. Together with the opposite
		inequality, this shows that $\delta_{\varphi,\chi}(\Gamma)=\delta$.
	\end{proof}

	\begin{cor}\label{prop:poincare}
		Let $\sH\lhd\Gamma$ be a non-elementary normal subgroup and let
		$\chi\in\Hom(\sH,\bR)$ satisfy \eqref{eq:conj-inv}. If
		$\delta>0$ and $\nu\in\cM_{\varphi,\chi}(\delta,\sH)$, then
		\[
		\sum_{\gamma\in\Gamma}
		\lVert\nu_{\gamma o}\rVert\,
		e^{-s\,\dph(e,\gamma)}
		<\infty
		\qquad\text{for every }s>\delta .
		\]
	\end{cor}
	
	\begin{proof}
		The assertion follows from Corollary~\ref{cor:shadow-consequences}(2), as in the first
		part of the proof of Proposition~\ref{prop:rigidity}.
	\end{proof}
	
	\subsection{Existence: the twisted Patterson construction}

\begin{lemma}
	\label{lem:patterson-h}
	Let $\Gamma_0\leq\Gamma$ and
	$\chi\in\operatorname{Hom}(\Gamma_0,\mathbb R)$. Assume that
	$\delta:=\delta_{\varphi,\chi}(\Gamma_0)\in\mathbb R$.  There exists a
	non-decreasing function $f\colon\bR\longrightarrow(0,\infty)$
	such that for every $\varepsilon>0$, there is
	$N_\varepsilon\in\bR$ with
	$f(t+r)\leq e^{\varepsilon r}f(t)$
for all $t\geq N_\varepsilon,\ r\geq0$.
	Moreover, 
	\[
	Q^f_{\Gamma,\varphi,\chi}(s)
	:=
	\sum_{\gamma_0\in\Gamma}
	e^{\chi(\gamma_0)-s\dph(e,\gamma_0)}
	f\bigl(\dph(e,\gamma_0)\bigr)
	\]
	converges for $s>\delta$ and diverges for $s\leq\delta$.
\end{lemma}

	
\begin{proof}
	We give the proof for $\Gamma_0=\Gamma$.
	The same argument applies to any subgroup $\Gamma_0\leq\Gamma$
	with $\chi\in\operatorname{Hom}(\Gamma_0,\mathbb R)$ and
	 $\delta_{\varphi,\chi}(\Gamma_0)\in \bR$.
	
If $Q_{\Gamma,\varphi,\chi}
	(\delta_{\varphi,\chi}(\Gamma))=+\infty$, take
	$f\equiv1$. All the assertions then follow from
	Lemma~\ref{lem:abscissa}, with $N_\varepsilon=0$.
	Suppose that
	$Q_{\Gamma,\varphi,\chi}(\delta_{\varphi,\chi}(\Gamma))<\infty$,
	and set $t_0=0$ and $f(t)=1$ for $t\leq0$.
	For $n\geq1$, assume inductively that $f$ has been defined
	on $(-\infty,t_{n-1}]$. By Lemma~\ref{lem:abscissa}, we have
	$Q_{\Gamma,\varphi,\chi}
	\bigl(\delta_{\varphi,\chi}(\Gamma)-1/n\bigr)
	=+\infty.$
	By Lemma~\ref{lem:dphi-qi}(1), the set
	$\{\gamma\in\Gamma:\dph(e,\gamma)\leq t_{n-1}\}$ is finite.
Since
$Q_{\Gamma,\varphi,\chi}(\delta_{\varphi,\chi}(\Gamma)-1/n)=+\infty$
by Lemma~\ref{lem:abscissa} and removing finitely many terms
preserves divergence, the positive partial sums over
$t_{n-1}<\dph(e,\gamma)\leq T$ tend to $+\infty$ as
$T\to+\infty$.
	Since $f(t_{n-1})e^{-t_{n-1}/n}$ is a fixed positive number,
	we may choose $t_n\geq t_{n-1}+1$ such that
	\[
	f(t_{n-1})e^{-t_{n-1}/n}
	\sum_{\substack{\gamma\in\Gamma\\
			t_{n-1}<\dph(e,\gamma)\leq t_n}}
	e^{\chi(\gamma)
		-(\delta_{\varphi,\chi}(\Gamma)-1/n)\dph(e,\gamma)}
	\geq1.
	\]
	Define
	\[
	f(t):=f(t_{n-1})e^{(t-t_{n-1})/n}
	\qquad(t_{n-1}\leq t\leq t_n).
	\]
	As $t_n\to+\infty$, these formulas define a continuous non-decreasing positive function  on $\bR$.
	By the choice of $t_n$ and the definition of $f$,
	\[
	\sum_{\substack{\gamma\in\Gamma\\
			t_{n-1}<\dph(e,\gamma)\leq t_n}}
	e^{\chi(\gamma)
		-\delta_{\varphi,\chi}(\Gamma)\dph(e,\gamma)}
	f\bigl(\dph(e,\gamma)\bigr)
	\geq1.
	\]
	Summing over these disjoint intervals gives $Q^f_{\Gamma,\varphi,\chi}
	\bigl(\delta_{\varphi,\chi}(\Gamma)\bigr)=+\infty.$
	Since $f\geq1$, Lemma~\ref{lem:abscissa} also gives
	$Q^f_{\Gamma,\varphi,\chi}(s)
	\geq Q_{\Gamma,\varphi,\chi}(s)=+\infty$
	for $s<\delta_{\varphi,\chi}(\Gamma)$.
	
	Fix $\varepsilon>0$, choose $n_\varepsilon\in\mathbb N$ with
	$1/n_\varepsilon\leq\varepsilon$, and set
	$N_\varepsilon:=t_{n_\varepsilon-1}$.
	Since $\log f$ is piecewise affine with slopes at most
	$\varepsilon$ on $[N_\varepsilon,\infty)$, we have
	$\log f(t+r)-\log f(t)\leq\varepsilon r$, hence
	$f(t+r)\leq e^{\varepsilon r}f(t)$ whenever
	$t\geq N_\varepsilon$ and $r\geq0$.
	
	Finally, fix $s>\delta_{\varphi,\chi}(\Gamma)$ and
	choose $\varepsilon\in (0,s-\delta_{\varphi,\chi}(\Gamma))$.
	The preceding estimate gives
	\[
	f(t)\leq
	f(N_\varepsilon)e^{\varepsilon(t-N_\varepsilon)}
	=C_\varepsilon e^{\varepsilon t}
	\qquad(t\geq N_\varepsilon),
	\]
	where
	$C_\varepsilon:=f(N_\varepsilon)
	e^{-\varepsilon N_\varepsilon}$.
	Consequently, since $s-\varepsilon>\delta_{\varphi,\chi}(\Gamma)$
	\[
	\begin{aligned}
		\sum_{\substack{\gamma\in\Gamma\\
				\dph(e,\gamma)\geq N_\varepsilon}}
		e^{\chi(\gamma)-s\dph(e,\gamma)}
		f\bigl(\dph(e,\gamma)\bigr)
		\leq
		C_\varepsilon
		\sum_{\substack{\gamma\in\Gamma\\
				\dph(e,\gamma)\geq N_\varepsilon}}
		e^{\chi(\gamma)-(s-\varepsilon)\dph(e,\gamma)}
		\leq
		C_\varepsilon Q_{\Gamma,\varphi,\chi}(s-\varepsilon)
		<\infty.
	\end{aligned}
	\]
	The remaining terms form a finite sum, proving that
	$Q^f_{\Gamma,\varphi,\chi}(s)<\infty$ for
	$s>\delta_{\varphi,\chi}(\Gamma)$.
\end{proof}

\begin{thm}\label{thm:existence}
	Let $\sH\lhd\Gamma$ be an infinite normal subgroup and
	$\chi\in\Hom(\sH,\bR)$. Suppose that
	$\delta:=\delta_{\varphi,\chi}(\sH)$ is finite.
	Then $\delta>0$, and there exists a family
	$\mu=(\mu_x)_{x\in\mathbb X}$ of finite non-zero Borel
	measures on $\cF$, with $\supp\mu_x=\Lam$ for every
	$x\in\mathbb X$, satisfying \eqref{eq:conf} with
	dimension $\delta$ and \eqref{eq:equiv} for every
	element of $\sH$. If $\chi$ satisfies \eqref{eq:conj-inv},
	then $\mu\in\cM_{\varphi,\chi}(\delta,\sH)$.
	In particular,
	$\cM_{\varphi,\chi}
	(\delta_{\varphi,\chi}(\Gamma),\Gamma)\neq\emptyset$
	for every $\chi\in\Hom(\Gamma,\bR)$.
\end{thm}

	\begin{proof}
		Since $\sH$ is non-elementary by \cite[Thm.~12.2(5)]{KB},
		the free-subgroup argument in the proof of
		Proposition~\ref{prop:delta-props}(2) gives $\delta>0$.

		Let $f$ and $Q^f_{\sH,\varphi,\chi}$ be as in
		Lemma~\ref{lem:patterson-h}. For $s>\delta$, define a probability
		measure on $\overline\Gamma$ by
		\[
		\widetilde\mu_s
		:=
		\frac{1}{Q^f_{\sH,\varphi,\chi}(s)}
		\sum_{\gamma_0\in\sH}
		e^{\chi(\gamma_0)-s\dph(e,\gamma_0)}
		f\bigl(\dph(e,\gamma_0)\bigr)\mathfrak D_{\gamma_0}.
		\]
		Since $\overline\Gamma$ is compact, we may choose
		$s_k\downarrow\delta$ such that
		$\widetilde\mu_{s_k}\to\widetilde\mu$ weak-$\ast$ for some
		$\widetilde\mu\in\mathcal P(\overline\Gamma)$.
			By Lemma~\ref{lem:patterson-h} and Fatou's lemma, we have 
		$+\infty
		=
		Q^f_{\sH,\varphi,\chi}(\delta)
		\leq
		\liminf_{k\to\infty}
		Q^f_{\sH,\varphi,\chi}(s_k).$
		Consequently,
		$\widetilde\mu_{s_k}(\{\gamma\})\to0$ for every
		$\gamma\in\Gamma$. Each point of $\Gamma$ is isolated in
		$\overline\Gamma$, so weak-$\ast$ convergence gives
		$\widetilde\mu(\{\gamma\})=0$. As $\Gamma$ is countable,
		$\widetilde\mu(\Gamma)=0$. 
		Moreover, weak-$\ast$ convergence gives $\widetilde\mu(\overline\Gamma)=1$,
		and since $\widetilde\mu(\Gamma)=0$, we have
		$\widetilde\mu(\bd)=1$.
		
		We define 
		\[
		\mu_o:=\theta_\ast\widetilde\mu,
		\qquad
		\mu_x:=e^{-\delta\varphi(\beta_\bullet(x,o))}\mu_o
		\quad(x\in\mathbb X).
		\]
		By Lemma~\ref{lem:beta-props}(4), these are finite non-zero
		Borel measures supported on $\Lam$, and the cocycle identity in
		Lemma~\ref{lem:beta-props}(1) gives
		$\mu_x=e^{-\delta\varphi(\beta_\bullet(x,y))}\mu_y$.
		Hence, $\mu$ satisfies \eqref{eq:conf}.
		
		It remains to prove twisted equivariance. Fix $\gamma_0\in\sH$ and define
		$V_{\gamma_0}\colon\overline\Gamma\longrightarrow\bR$ by
		\[
		V_{\gamma_0}(x):=
		\begin{cases}
			\dph(\gamma_0,\gamma)-\dph(e,\gamma),
			& \gamma\in\Gamma,\\[2mm]
			\varphi\bigl(\beta_{\theta(x)}(\gamma_0o,o)\bigr),
			& x\in\bd.
		\end{cases}
		\]
		Note that $V_{\gamma _0}\in \mathcal{C}(\overline{\Gamma})$. Set $a:=\dph(\gamma_0,\gamma), b:= \dph(e,\gamma)$, and  $L:=\sqrt d\,\lVert\varphi\rVert\,
		\lVert\kappa(\gamma_0)\rVert$.
		By Lemma~\ref{lem:dphi-qi}(2), we have $|a-b|\le L$.
			For $\varepsilon>0$, let $N_\varepsilon$ be as in
		Lemma~\ref{lem:patterson-h}. If
		$b \geq N_\varepsilon+L$, then both
		$a,b$ are at least
		$N_\varepsilon$. Therefore, we have $\frac{f(a)}{f(b)}\in[e^{-\varepsilon L},e^{\varepsilon L}]$.
		Since $\varepsilon>0$ is arbitrary, Lemma~\ref{lem:dphi-qi}(1)
	and the preceding bounds yield $\frac{f(a)}{f(b)}\longrightarrow1$ as
	$|\gamma|\to+\infty$. Since $(\gamma_0)_\ast\mathfrak D_{\gamma_0'}
	=\mathfrak D_{\gamma_0\gamma_0'}$,
	reindexing by $\gamma=\gamma_0\gamma_0'$ gives
\[
e^{\chi(\gamma_0)}(\gamma_0)_\ast\widetilde\mu_{s_k}
=
R_{s_k}\widetilde\mu_{s_k}.
\tag{*}\label{sta}
\]
	where 
	\[
	R_{s_k}(\gamma)
	:
	=
	e^{-s_k V_{\gamma_0}(\gamma)}\frac{f(a)}{f(b)}.
	\]
Setting $R_{s_k}(x):=e^{-s_kV_{\gamma_0}(x)}$ for $x\in\bd$
	defines 
	$R_{s_k}\in\mathcal C(\overline\Gamma)$.
		Note that 
		$	R_{s_k}
		=
		e^{-(s_k-\delta)V_{\gamma_0}}R_\delta$
		and hence
		\[
		\lVert R_{s_k}-R_\delta\rVert_\infty
		\leq
		\lVert R_\delta\rVert_\infty
		\left(
		e^{(s_k-\delta)L}-1
		\right)
		\longrightarrow0
		\qquad
		(s_k\to \delta).
		\]
		For every $F\in\mathcal C(\overline\Gamma)$,  we obtain 
		\[
		\begin{aligned}
			\left|
			\int FR_{s_k}\,d\widetilde\mu_{s_k}
			-
			\int FR_\delta\,d\widetilde\mu
			\right|
			\leq
			\lVert F\rVert_\infty
			\lVert R_{s_k}-R_\delta\rVert_\infty
			+
			\left|
			\int FR_\delta\,d\widetilde\mu_{s_k}
			-
			\int FR_\delta\,d\widetilde\mu
			\right|,
		\end{aligned}
		\]
		where the right-hand side tends to $0$ as $k\to\infty$.
		Therefore, for  $F\in\mathcal C(\overline\Gamma)$, the preceding
		identity~\eqref{sta} and the integration formula for push-forward
		measures give
			\[
		\begin{aligned}
			e^{\chi(\gamma_0)}
			\int F\circ \gamma_0\,d\widetilde\mu_{s_k}
			&= \int FR_{s_k}\,d\widetilde\mu_{s_k}.
		\end{aligned}
		\]
		Taking  $k\to\infty$ and using weak-$\ast$ convergence, we obtain
	$	e^{\chi(\gamma_0)}(\gamma_0)_\ast\widetilde\mu
	=
	R_\delta\widetilde\mu.$
Using the equivariance
of $\theta$, we obtain
\[
\begin{aligned}
	e^{\chi(\gamma_0)}(\gamma_0)_\ast\mu_o
	&=\theta_\ast(R_\delta\widetilde\mu)\\
	&=
	e^{-\delta\varphi(\beta_\bullet(\gamma_0o,o))}\mu_o\\
	&=\mu_{\gamma_0o}.
\end{aligned}
\]
By the push-forward and Lemma~\ref{lem:beta-props}(1), we obtain 
\[
\begin{aligned}
	(\gamma_0)_\ast\mu_x
	&=
	e^{-\delta\varphi(\beta_{\gamma_0^{-1}\bullet}(x,o))}
	(\gamma_0)_\ast\mu_o\\
	&=
	e^{-\chi(\gamma_0)}
	e^{-\delta\varphi(\beta_\bullet(\gamma_0x,\gamma_0o))}
	\mu_{\gamma_0o}\\
	&=
	e^{-\chi(\gamma_0)}\mu_{\gamma_0x},
\end{aligned}
\]
which  proves \eqref{eq:equiv}.

		By \eqref{eq:conf}, $\supp\mu_x$ is independent of $x$,
		and \eqref{eq:equiv} makes this support $\sH$-invariant.
		Hence, $\theta^{-1}(\supp\mu_o)$ is a non-empty closed
		$\sH$-invariant subset of $\bd$. By
		\cite[Prop.~12.1(1)]{KB}, it equals $\bd$, so
		$\supp\mu_x=\Lam$ for every $x\in\mathbb X$.
		If $\chi$ satisfies \eqref{eq:conj-inv}, it follows that
		$\mu\in\cM_{\varphi,\chi}(\delta,\sH)$.
	
	Finally, when $\sH=\Gamma$,
	Proposition~\ref{prop:delta-props}(2) gives
	$\delta_{\varphi,\chi}(\Gamma)<\infty$, and
	\eqref{eq:conj-inv} holds automatically.
	This proves the final assertion.
	\end{proof}
	
		\begin{cor}\label{prop:jordan-constraint}
			Let $\chi\in\Hom(\Gamma,\bR)$ and set
			$\delta:=\delta_{\varphi,\chi}(\Gamma)$. Then
			\[
			-\delta\,\varphi\bigl(\ii\,\mathsf{Jd}_\gamma\bigr)
			\leq\chi(\gamma)
			\leq\delta\,\varphi\bigl(\mathsf{Jd}_\gamma\bigr)
			\qquad(\gamma\in\Gamma).
			\]
			In particular, if $\varphi$ is symmetric, then
			$\lvert\chi(\gamma)\rvert
			\leq\delta\,\varphi(\mathsf{Jd}_\gamma)$
			for every $\gamma\in\Gamma$.
		\end{cor}
	
\begin{proof}
	By Theorem~\ref{thm:existence} and
	Corollary~\ref{cor:shadow-consequences}(1), there exists
	$C\geq1$ such that
	\[
	n\chi(\gamma)
	\leq\delta\,\varphi(\kappa(\gamma^n))+\log C
	\qquad(\gamma\in\Gamma,\ n\geq1).
	\]
	Dividing by $n$ and letting $n\to\infty$ gives
	$\chi(\gamma)\leq\delta\,\varphi(\mathsf{Jd}_\gamma)$
	by \eqref{Eq:Jordan}. Applying this to $\gamma^{-1}$
	and using $\chi(\gamma^{-1})=-\chi(\gamma)$ and
	\eqref{Eq:Jd-rel} gives the lower bound.
	
	The final assertion follows by symmetry.
\end{proof}

	\section{Differentiation along shadows}\label{sec:differentiation}

Fix $R\geq\max\{R_1,R_2\}$ and let $\mu$ be a finite Borel
measure on $\Lam$ such that, for some $D\geq1$,
\begin{equation}\label{eq:doubling}
	0<\mu(\cO_R(e,\gamma)),
	\qquad
	\mu(\cO_{6R}(e,\gamma))
	\leq D\,\mu(\cO_R(e,\gamma))
	\qquad(\gamma\in\Gamma).
\end{equation}
For $\nu\in\cM_{\varphi,\chi}(\delta,\sH)$, these conditions
hold for $\mu=\nu_o$ with $D=C_RC_{6R}$ by
Theorem~\ref{thm:shadowprinciple} and
Corollary~\ref{cor:shadow-consequences}(4), respectively.

\begin{dfn}\label{def:radial-convergence}
	A sequence $(\gamma_n)$ in $\Gamma$ converges \emph{radially}
	to $x\in\bd$ if $d_w(e,\gamma_n)\longrightarrow+\infty$ and
	$x\in\cO_R(e,\gamma_n)$ for every $n$.
\end{dfn}


\begin{dfn}\label{Def:Conl}
	The \emph{conical limit set} of $\Gamma$ is defined by 
	\[
	\Lambda_{\mathrm{con}}(\Gamma)
	:=\bigl\{\theta(x):x\in\partial_\infty\Gamma
	\text{ is a radial limit of a sequence in }\Gamma\bigr\}.
	\]
\end{dfn}

\begin{lemma}\label{lem:maximal}
	Let $\nu$ be a finite Borel measure on $\bd$ and $\lambda>0$. Set
	\[
	I:=\{\gamma\in\Gamma:
	\nu(\cO_R(e,\gamma))>\lambda\,\mu(\cO_{6R}(e,\gamma))\}.
	\]
	Then
	\[
	\mu\left(\bigcup_{\gamma\in I}\cO_R(e,\gamma)\right)
	\leq\frac{\lVert\nu\rVert}{\lambda}.
	\]
\end{lemma}

\begin{proof}
	The set $E:=\bigcup_{\gamma\in I}\cO_R(e,\gamma)$ is Borel,
	since $\Gamma$ is countable and shadows are closed.
	By Lemma~\ref{lem:vitali}, there exists $J\subseteq I$ such that
	the shadows $\{\cO_R(e,\gamma)\}_{\gamma\in J}$ are pairwise disjoint
	and $E\subseteq\bigcup_{\gamma\in J}\cO_{6R}(e,\gamma)$. Hence
	\[
	\mu(E)
	\leq\sum_{\gamma\in J}\mu(\cO_{6R}(e,\gamma))
	\leq\frac{1}{\lambda}
	\sum_{\gamma\in J}\nu(\cO_R(e,\gamma))
	\leq\frac{\lVert\nu\rVert}{\lambda},
	\]
	where the last inequality follows from disjointness.
\end{proof}
	
	We write $L^1(\mu)$ for the space of measurable functions $f$ on $\bd$
	with $\lVert f\rVert_{L^1(\mu)}:=\int_{\bd}|f|\,d\mu<\infty$,
	identifying functions that agree $\mu$-almost everywhere.

\begin{thm}\label{thm:differentiation}
	Let $\mu$ be as above and $f\in L^1(\mu)$. Then, for $\mu$-almost
	every $x\in\bd$ and every sequence $(\gamma_i)$ converging radially
	to $x$,
	\[
	\lim_{i\to\infty}
	\frac{1}{\mu(\cO_R(e,\gamma_i))}
	\int_{\cO_R(e,\gamma_i)}f\,d\mu=f(x).
	\]
	In particular, for every Borel set $A\subseteq\bd$, the ratios $\frac{\mu(A\cap\cO_R(e,\gamma_i))}{\mu(\cO_R(e,\gamma_i))}$
converge to
	$\mathbb{1}_A(x)$ for $\mu$-almost every $x$ and every radial sequence
	converging to $x$.
\end{thm}

\begin{proof}
	Choose an everywhere finite Borel representative of $f\in L^1(\mu)$ and set
	\[
	A_\gamma h
	:=
	\frac{1}{\mu(\cO_R(e,\gamma))}
	\int_{\cO_R(e,\gamma)}h\,d\mu
	\qquad(h\in L^1(\mu),\ \gamma\in\Gamma).
	\]
	For every $g\in\mathcal C(\bd)$, Lemma~\ref{lem:shrinking}
	and uniform continuity give
	\[
	\sup_{x\in\cO_R(e,\gamma)}|A_\gamma g-g(x)|
	\leq
	\sup_{x,y\in\cO_R(e,\gamma)}|g(y)-g(x)|
	\longrightarrow0
	\qquad(|\gamma|\to\infty).
	\tag{*}\label{e:co}
	\]
	
	For $t>0$, define
	\[
	E_t
	:=
	\left\{
	x\in\bd:
	\lim_{N\to\infty}
	\sup_{\substack{\gamma\in\Gamma,\ |\gamma|\geq N\\
			x\in\cO_R(e,\gamma)}}
	|A_\gamma f-f(x)|>t
	\right\}.
	\]
	The set $E_t$ is Borel, since $f$ is Borel measurable,
	the shadows are Borel, and $\Gamma$ is countable.
	Fix $g\in\mathcal C(\bd)$ and put $h=f-g$. For
	$x\in\cO_R(e,\gamma)$,
	\[
	|A_\gamma f-f(x)|
	\leq A_\gamma|h|+|A_\gamma g-g(x)|+|h(x)|.
	\]
	Together with~\eqref{e:co} this implies
	\[
	E_t
	\subseteq
	\{x\in\bd:|h(x)|>t/2\}
	\cup
	\bigcup_{\substack{\gamma\in\Gamma\\ A_\gamma|h|>t/2}}
	\cO_R(e,\gamma).
	\]
	Whenever $A_\gamma|h|>t/2$, \eqref{eq:doubling} gives
	\[
	\int_{\cO_R(e,\gamma)}|h|\,d\mu
	>
	\frac{t}{2}\mu(\cO_R(e,\gamma))
	\geq
	\frac{t}{2D}\mu(\cO_{6R}(e,\gamma)).
	\]
Define $\nu:=|h|\mu$ by $\nu(B):=\int_B|h|\,d\mu$ for
Borel $B\subseteq\bd$, so that
$\lVert\nu\rVert=\lVert h\rVert_{L^1(\mu)}$. 
	Applying  Lemma~\ref{lem:maximal} with
	$\nu=|h|\mu$ and $\lambda=t/(2D)$, we have
	\[
	\mu\left(
	\bigcup_{\substack{\gamma\in\Gamma\\A_\gamma|h|>t/2}}
	\cO_R(e,\gamma)
	\right)
	\leq\frac{2D}{t}\,\lVert h\rVert_{L^1(\mu)}.
	\tag{**}\label{eq:diff-maximal}
	\]
	On the other hand,
	$\frac{t}{2}\mu(\{x\in\bd:|h(x)|>t/2\})
	\leq\int_{\bd}|h|\,d\mu$, so
	\[
	\mu\bigl(\{x\in\bd:|h(x)|>t/2\}\bigr)
	\leq\frac{2}{t}\,\lVert h\rVert_{L^1(\mu)}.
	\tag{***}\label{eq:diff-markov}
	\]
	Combining the preceding inclusion with
	\eqref{eq:diff-maximal} and \eqref{eq:diff-markov}, we obtain
	\[
	\mu(E_t)
	\leq
	\frac{2}{t}\lVert h\rVert_{L^1(\mu)}
	+
	\frac{2D}{t}\lVert h\rVert_{L^1(\mu)}
	=
	\frac{2(D+1)}{t}\lVert f-g\rVert_{L^1(\mu)}.
	\]
	Since $\mathcal C(\bd)$ is dense in $L^1(\mu)$, it follows
	that $\mu(E_t)=0$ for every $t>0$.
Thus, 
	the set $E:=\bigcup_{n\geq1}E_{1/n}$ is $\mu$-null. For $x\notin E$, the definition of $E_t$ gives
	$A_{\gamma_i}f\to f(x)$ along every radial sequence
	$(\gamma_i)$ converging to $x$, which proves the first assertion.

	The final assertion follows by taking $f=\mathbb{1}_A$.
\end{proof}

\subsection{Mass domination and mass rigidity}\label{ss:fatou}

In this subsection, we prove that mass comparison implies
measure equivalence.
	
\begin{prop}\label{prop:masscomparison}
	Fix $\delta>0$ and $R\geq R_2$, and let
	$\mu,\nu\in\cM_{\varphi,\chi}(\delta,\sH)$. Suppose that
	$\Theta\geq1$ satisfies
	\[
	\Theta^{-1}\lVert\mu_{\gamma o}\rVert
	\leq\lVert\nu_{\gamma o}\rVert
	\leq\Theta\lVert\mu_{\gamma o}\rVert
	\qquad(\gamma\in\Gamma).
	\]
	Then $\mu_o$ and $\nu_o$ are mutually absolutely continuous, with
	\[
	(\Theta C_RC_{6R})^{-1}
	\leq\frac{d\nu_o}{d\mu_o}
	\leq\Theta C_RC_{6R}
	\qquad\mu_o\text{-a.e.}
	\]
	Moreover, $\lVert \nu_{\gamma o}\rVert \le \Theta\lVert\mu_{\gamma o}\rVert$ implies $\nu_o\le \Theta C_RC_{6R}\mu_0$.
\end{prop}

\begin{proof}
	It suffices to prove the one-sided assertion. Let $B\subseteq\bd$
	be Borel and $\varepsilon>0$. By outer regularity, choose an open
	$V\supseteq B$ with $\mu_o(V)\leq\mu_o(B)+\varepsilon$, and set
	\[
	I:=\{\gamma\in\Gamma:\cO_{6R}(e,\gamma)\subseteq V\}.
	\]
	Since $V$ is open, Lemmas~\ref{lem:shadow-basic}(2)
	and~\ref{lem:shrinking} imply that every $x\in B$ lies in
	some $\cO_R(e,\gamma)$ with $\cO_{6R}(e,\gamma)\subseteq V$.
	Hence $B\subseteq\bigcup_{\gamma\in I}\cO_R(e,\gamma)$.
	Since $R\geq R_2\geq R_1$, Lemma~\ref{lem:vitali} provides
	$J\subseteq I$ such that the shadows
	$\{\cO_R(e,\gamma)\}_{\gamma\in J}$ are pairwise disjoint and
	$B\subseteq\bigcup_{\gamma\in J}\cO_{6R}(e,\gamma)\subseteq V.$
Applying Theorem~\ref{thm:shadowprinciple} and the mass bound 
hypothesis, we obtain
\[
\begin{aligned}
	\nu_o(B)
	&\leq C_{6R}\sum_{\gamma\in J}
	\lVert\nu_{\gamma o}\rVert e^{-\delta\dph(e,\gamma)}\\
	&\leq \Theta C_RC_{6R}\sum_{\gamma\in J}
	\mu_o\bigl(\cO_R(e,\gamma)\bigr)\\
	&\leq \Theta C_RC_{6R}\bigl(\mu_o(B)+\varepsilon\bigr).
\end{aligned}
\]
Letting $\varepsilon\downarrow0$ gives
	$\nu_o\leq\Theta C_RC_{6R}\mu_o$.
	Interchanging $\mu$ and $\nu$ proves the reverse inequality and
	the  Radon--Nikodym bounds.
\end{proof}

\begin{thm}\label{thm:fatou}
	Let $\delta>0$,
	and let $\mu,\nu\in\cM_{\varphi,\chi}(\delta,\sH)$.
	\begin{enumerate}
		\item If
		$\lVert\mu_{\gamma o}\rVert\geq\lVert\nu_{\gamma o}\rVert$
		for every $\gamma\in\Gamma$, then $\mu_x\geq\nu_x$ as measures
		for every $x\in\mathbb X$.
		\item If
		$\lVert\mu_{\gamma o}\rVert=\lVert\nu_{\gamma o}\rVert$
		for every $\gamma\in\Gamma$, then $\mu=\nu$.
	\end{enumerate}
\end{thm}

\begin{proof}
	Assume the hypothesis of \textup{(1)}. Fix $R\geq R_2$ and put
	$D:=C_RC_{6R}$. The one-sided estimate in the proof of
	Proposition~\ref{prop:masscomparison} gives
	$\nu_o\leq D\mu_o$.
Set $	\tau:=\sup\{t\geq0:\mu_o\geq t\nu_o\}.$
	Then
	\[
	D^{-1}\leq\tau
	\leq\frac{\lVert\mu_o\rVert}{\lVert\nu_o\rVert}<\infty.
	\]
	Choose $t_n\uparrow\tau$ with $\mu_o\geq t_n\nu_o$.
	Then $\mu_o(B)\geq\lim_{n\to\infty}t_n\nu_o(B)
	=\tau\nu_o(B)$ for every Borel set $B\subseteq\cF$,
	so $\mu_o\geq\tau\nu_o$.
	To prove $\mu_o\geq\nu_o$, it suffices to show that
	$\tau\geq1$, so suppose for a contradiction that $\tau<1$.
	Define
	\[
	\mu'_x:=\frac{\mu_x-\tau\nu_x}{1-\tau}
	\qquad(x\in\mathbb X).
	\]
By \eqref{eq:conf}, $\mu_o\geq\tau\nu_o$ and $\tau<1$ imply that each $\mu'_x$ is a positive measure.
	Moreover, for every $\gamma\in\Gamma$,
	\[
	\begin{aligned}
		\lVert\mu'_{\gamma o}\rVert
		&=\frac{\lVert\mu_{\gamma o}\rVert
			-\tau\lVert\nu_{\gamma o}\rVert}{1-\tau}\\
		&\geq
		\frac{(1-\tau)\lVert\nu_{\gamma o}\rVert}{1-\tau}
		=\lVert\nu_{\gamma o}\rVert>0.
	\end{aligned}
	\tag{*}\label{e:ine}
	\]
In particular, $\mu'_o\neq0$, so \eqref{eq:conf} gives
$\mu'_x\neq0$ for every $x\in\mathbb X$.
Moreover, each $\mu'_x$ is finite and supported on $\Lam$. For $x,y\in\mathbb X$ and $\gamma_0\in\Gamma_0$,
\eqref{eq:conf} and \eqref{eq:equiv} give
\[
\begin{aligned}
	\mu'_x
	&=
	e^{-\delta\varphi(\beta_\bullet(x,y))}
	\frac{\mu_y-\tau\nu_y}{1-\tau}
	=
	e^{-\delta\varphi(\beta_\bullet(x,y))}\mu'_y,\\
	(\gamma_0)_\ast\mu'_x
	&=
	\frac{(\gamma_0)_\ast\mu_x-\tau(\gamma_0)_\ast\nu_x}{1-\tau}
	=
	e^{-\chi(\gamma_0)}
	\frac{\mu_{\gamma_0x}-\tau\nu_{\gamma_0x}}{1-\tau}
	=
	e^{-\chi(\gamma_0)}\mu'_{\gamma_0x}.
\end{aligned}
\]
Hence $\mu'\in\cM_{\varphi,\chi}(\delta,\Gamma_0)$.
	Applying Proposition~\ref{prop:masscomparison} to $(\mu',\nu)$
	therefore gives $\nu_o\leq D\mu'_o$. Consequently,
	\[
	\mu_o
	=\tau\nu_o+(1-\tau)\mu'_o
	\geq\left(\tau+\frac{1-\tau}{D}\right)\nu_o.
	\]
	Since $\tau+\frac{1-\tau}{D}>\tau$, this contradicts the
	definition of $\tau$. Thus $\tau\geq1$, so $\mu_o\geq\nu_o$.
By \eqref{eq:conf}, $\mu_o\geq\nu_o$ implies
$\mu_x\geq\nu_x$ for every $x\in\mathbb X$, proving (1).
 Applying \textup{(1)} in both
	directions proves \textup{(2)}.
\end{proof}

\begin{cor}\label{cor:twisted-uniqueness}
	Let $\chi\in\Hom(\Gamma,\mathbb R)$ and
	$\mu^i\in\mathcal M_{\varphi,\chi}(\delta_i,\Gamma)$
	with $\delta_i>0$, $i=1,2$. Then
	\[
	\delta_1=\delta_2=\delta_{\varphi,\chi}(\Gamma),
	\qquad
	\mu^1=\frac{\|\mu_o^1\|}{\|\mu_o^2\|}\,\mu^2.
	\]
\end{cor}

\begin{proof}
	By Proposition~\ref{prop:rigidity}, we have 
	$\delta_1=\delta_2$.
	Set
	$c:=\frac{\lVert\mu^1_o\rVert}{\lVert\mu^2_o\rVert}>0.$
	Then $c\mu^2\in\cM_{\varphi,\chi}(\delta,\Gamma)$.
	By Lemma~\ref{lem:density-basic}\textup{(1)}, for every
	$\gamma\in\Gamma$, we obtain 
	\[
	\begin{aligned}
		\lVert\mu^1_{\gamma o}\rVert
		=e^{\chi(\gamma)}\lVert\mu^1_o\rVert
		=\lVert(c\mu^2)_{\gamma o}\rVert.
	\end{aligned}
	\]
	Applying Theorem~\ref{thm:fatou}\textup{(2)} to
	$\mu^1$ and $c\mu^2$, with $\sH=\Gamma$, gives
	$\mu^1=c\mu^2$.
\end{proof}

The preceding theorem yields the following rigidity result for characters.
\begin{cor}\label{prop:reversibility}
	Assume $\varphi$ is symmetric and let $\chi\in\Hom(\Gamma,\bR)$ be such that
	$\delta_{\varphi,\chi}(\Gamma)=\delta_\varphi(\Gamma)$. Then $\chi=0$.
\end{cor}

\begin{proof}
	Set $\delta:=\delta_\varphi(\Gamma)$. 
By Proposition~\ref{prop:delta-props}\textup{(4)},
$\delta_{\varphi,-\chi}(\Gamma)
=\delta_{\varphi,\chi}(\Gamma)=\delta$.
Moreover, Theorem~\ref{thm:existence} allows us to choose
$\mu^\pm\in\cM_{\varphi,\pm\chi}(\delta,\Gamma)$ and
$\mu^0\in\cM_\varphi(\delta,\Gamma)$ with
$\lVert\mu^\pm_o\rVert=\lVert\mu^0_o\rVert=1$.
	Put $\sH:=\ker\chi$. Since $\sH$ contains $[\Gamma,\Gamma]$,
	it is a non-elementary normal subgroup by the argument in
	Proposition~\ref{prop:delta-props}\textup{(2)}.
	We define $$\nu_x:=\tfrac12(\mu^+_x+\mu^-_x).$$
	By \eqref{eq:conf}, we have
	\[
	\nu_x
	=\tfrac12 e^{-\delta\varphi(\beta_\bullet(x,y))}
	(\mu_y^++\mu_y^-)
	=e^{-\delta\varphi(\beta_\bullet(x,y))}\nu_y.
	\]
	 Moreover, since
	$\chi|_{\Gamma_0}=0$, \eqref{eq:equiv} gives
	\[
	(\gamma_0)_\ast\nu_x
	=\tfrac12\bigl(
	e^{-\chi(\gamma_0)}\mu^+_{\gamma_0x}
	+e^{\chi(\gamma_0)}\mu^-_{\gamma_0x}
	\bigr)
	=\nu_{\gamma_0x}
	\qquad(\gamma_0\in\Gamma_0,\ x\in\mathbb X).
	\]
	Thus $\nu,\mu^0\in\cM_\varphi(\delta,\sH)$.
	By Lemma~\ref{lem:density-basic}\textup{(1)},
	\[
	\lVert\nu_{\gamma o}\rVert
	=\frac{e^{\chi(\gamma)}+e^{-\chi(\gamma)}}{2}
	=\cosh\chi(\gamma)
	\geq1=\lVert\mu^0_{\gamma o}\rVert
	\qquad(\gamma\in\Gamma).
	\]
	Theorem~\ref{thm:fatou}\textup{(1)} therefore yields
	$\nu_o\geq\mu^0_o$.  Moreover, 
	$	\lVert\nu_o\rVert
	=\tfrac12\bigl(\lVert\mu_o^+\rVert+\lVert\mu_o^-\rVert\bigr)
	=1=\lVert\mu_o^0\rVert.$
	Since both measures have total mass $1$,
	$\nu_o=\mu^0_o$, and \eqref{eq:conf} gives
	$\nu_x=\mu_x^0$ for every $x\in\mathbb X$. 
Consequently, $\cosh\chi(\gamma)=\lVert\nu_{\gamma o}\rVert
=\lVert\mu^0_{\gamma o}\rVert=1$ for every $\gamma\in\Gamma$,
so $\chi=0$.
\end{proof}

	\section{The twisted Bowen--Margulis--Sullivan measure and ergodicity}
	\label{sec:HTS}

	Throughout this section, fix $\chi\in\Hom(\Gamma,\bR)$ and set
	$\delta:=\delta_{\varphi,\chi}(\Gamma)\in(0,\infty)$.
	By Proposition~\ref{prop:delta-props}(1),
	$\varphi^\ii\in\cL^{\ast}_\Gamma$ and
	$\delta_{\varphi^\ii,-\chi}(\Gamma)=\delta$.
	Applying Theorem~\ref{thm:existence} to $(\varphi,\chi)$ and
	$(\varphi^\ii,-\chi)$, we choose and fix densities
	\[
	\mu\in\cM_{\varphi,\chi}(\delta,\Gamma),
	\qquad
	\mu^\ii\in\cM_{\varphi^\ii,-\chi}(\delta,\Gamma).
	\]
	
We define $\Lambda^{(2)}(\Gamma):=(\Lam\times\Lam)\cap\cFt.$ 
The limit map $\theta $ induces a homeomorphism  from $\bdd$ to
$\Lambda^{(2)}_\Gamma$.

\begin{dfn}\label{def:bms-measure}
	We define a Borel measure $\widetilde m$ on $\Lambda^{(2)}(\Gamma)$ by
	\[
	d\widetilde m(\xi,\eta)
	:=e^{\delta\varphi^{\ii}(\cG(\xi,\eta))}
	\,d\mu^{\ii}_o(\xi)\,d\mu_o(\eta).
	\]
\end{dfn}

\begin{dfn}\label{def:radon}
	A positive Borel measure $\lambda$ on $\Lambda^{(2)}(\Gamma)$
	is called \emph{Radon} if it is locally finite.
\end{dfn}

\begin{lemma}\label{lem:bms-invariant}
	The measure $\widetilde m$ is a nonzero $\Gamma$-invariant Radon measure,
	and
	\[
	\widetilde m\sim
	(\mu^{\ii}_o\otimes\mu_o)|_{\Lambda(\Gamma)}.
	\]
\end{lemma}

\begin{proof}
Write $\lambda:=(\mu_o^{\ii}\otimes\mu_o)|_\La$.
For every compact $K\subseteq\La$, continuity of $\cG$ gives
\[
\widetilde m(K)
\leq
\left(\sup_{K}
e^{\delta\varphi^{\ii}(\cG(\xi,\eta))}\right)
\lVert\mu_o^{\ii}\rVert\,\lVert\mu_o\rVert<\infty.
\]
Thus, $\widetilde m$ is Radon.
	Moreover,  we have $\widetilde m\sim\lambda$.
	Applying Lemma~\ref{lem:density-basic}(4) with $\Gamma_0=\Gamma$
	 gives $\supp\mu_o^\ii=\supp\mu_o=\Lam$.
	Choose disjoint nonempty relatively open sets $U,V\subseteq\Lam$.
	By transversality, $U\times V\subseteq\La$, and
	$\lambda(U\times V)=\mu_o^\ii(U)\mu_o(V)>0$.
	Since $\widetilde m\sim\lambda$, it follows that $\widetilde m\neq0$.
	For every $\gamma\in\Gamma$, \eqref{eq:equiv} and \eqref{eq:conf} give
	\[
	\begin{aligned}
		d(\gamma_\ast\mu_o^\ii)(\xi)
		&=e^{\chi(\gamma)-\delta\varphi^\ii(\beta_\xi(\gamma o,o))}
		\,d\mu_o^\ii(\xi),\\
		d(\gamma_\ast\mu_o)(\eta)
		&=e^{-\chi(\gamma)-\delta\varphi(\beta_\eta(\gamma o,o))}
		\,d\mu_o(\eta).
	\end{aligned}
	\]
	By Lemma~\ref{lem:gromov-equiv} and Lemma~\ref{lem:beta-props}(2),
	\[
	\cG(\gamma^{-1}\xi,\gamma^{-1}\eta)
	=\cG(\xi,\eta)+\beta_\xi(\gamma o,o)
	+\ii\beta_\eta(\gamma o,o).
	\]
	Since $\varphi^\ii\circ\ii=\varphi$, the preceding identities give
	\[
	\begin{aligned}
		d(\gamma_\ast\widetilde m)(\xi,\eta)
		=e^{\delta\varphi^\ii(
			\cG(\gamma^{-1}\xi,\gamma^{-1}\eta))}
		\,d(\gamma_\ast\mu_o^\ii)(\xi)\,
		d(\gamma_\ast\mu_o)(\eta)
		=d\widetilde m(\xi,\eta), 
	\end{aligned}
	\]
	proving $\gamma_\ast\widetilde m=\widetilde m$ for every $\gamma\in\Gamma$.
\end{proof}

	\subsection{The twisted Hopf--Tsuji--Sullivan theorem}\label{ss:Hopf--Tsuji--Sullivan}
	
	
Let a countable group $\Gamma$ act measurably on a standard
Borel space $X$, and let $\mu$ be a $\sigma$-finite Borel measure.

\begin{dfn}
	The measure $\mu$ is \emph{quasi-invariant}, and the action
	\emph{non-singular}, if $\gamma_\ast\mu\sim\mu$ for every $\gamma\in\Gamma$.
\end{dfn}

\begin{dfn}
	The action is \emph{essentially free} (or \emph{free modulo $\mu$})
	if $\mu(\{x\in X:\gamma x=x\})=0$ for every $\gamma\ne e$.
\end{dfn}

\begin{dfn}
	A non-singular action is \emph{ergodic} if every
	$\Gamma$-invariant Borel set $A\subseteq X$ satisfies
	$\mu(A)=0$ or $\mu(X\smallsetminus A)=0$.
\end{dfn}

\begin{dfn}
	A non-singular action is \emph{conservative} if, for every
	Borel set $A\subseteq X$, we have
	$\left|\{\gamma\in\Gamma:\gamma x\in A\}\right|=\infty$
	for $\mu$-almost every $x\in A$.
\end{dfn}

\begin{dfn}
	The measure $\mu$ is \emph{non-atomic} (or \emph{atomless})
	if $\mu(\{x\})=0$ for every $x\in X$.
\end{dfn}

	\begin{thm}\label{thm:HTS}
		Let $\chi\in\Hom(\Gamma,\bR)$, $\delta=\delta_{\varphi,\chi}(\Gamma)$, and let
		$\mu\in\cM_{\varphi,\chi}(\delta,\Gamma)$,
		$\mu^\ii\in\cM_{\varphi^\ii,-\chi}(\delta,\Gamma)$. Then:
		\begin{enumerate}
			\item the $\Gamma$-action on $(\La,\widetilde m)$ is ergodic;
			\item the $\Gamma$-action on $(\cF,\mu_o)$ is ergodic, and likewise for
			$\mu^\ii_o$;
			\item $\mu_o$ and $\mu^\ii_o$ are atomless;
			\item the $\Gamma$-action on $(\cF,\mu_o)$ is conservative;
		\end{enumerate}
	\end{thm}
	
	\begin{proof}
		(1) 
			By Lemma~\ref{lem:bms-invariant}, $\widetilde m$ is a nonzero
			$\Gamma$-invariant Radon measure on $\Lambda_\Gamma^{(2)}$ satisfying
		$	\widetilde m
		\sim  (\mu_o^\ii\otimes\mu_o)|_{\Lambda_\Gamma^{(2)}}.$
				The claim follows from \cite[Theorem~3.1]{CantrellTanaka}
			via the $\Gamma$-equivariant homeomorphism $\theta\times\theta$.
		
		(2) By Lemma~\ref{lem:density-basic}(4),
		$\supp\mu_o=\supp\mu_o^\ii=\Lam$.
		It therefore suffices to prove ergodicity on $\Lam$.
		Let $A\subseteq\Lam$ be a $\Gamma$-invariant Borel set with
		$\mu_o(A)>0$. Then
		$\widetilde A:=(\Lam\times A)\cap\La$
		is $\Gamma$-invariant.
		Since $\Gamma$ is non-elementary,
		$\Lam\smallsetminus\{\eta\}$ is a nonempty relatively open
		subset of $\Lam$ for every $\eta\in\Lam$. Thus
		$\mu_o^\ii(\Lam\smallsetminus\{\eta\})>0$, and consequently
		\[
		\widetilde m(\widetilde A)
		=
		\int_A
		\left(
		\int_{\Lam\smallsetminus\{\eta\}}
		e^{\delta\varphi^\ii(\cG(\xi,\eta))}
		\,d\mu_o^\ii(\xi)
		\right)d\mu_o(\eta)>0.
		\]
		By (1), the complement of $\widetilde A$ is
		$\widetilde m$-null. Hence
		\[
		0
		=
		\widetilde m(\Lambda_\Gamma^{(2)}\smallsetminus\widetilde A)
		=
		\int_{\Lam\smallsetminus A}
		\left(
		\int_{\Lam\smallsetminus\{\eta\}}
		e^{\delta\varphi^\ii(\cG(\xi,\eta))}
		\,d\mu_o^\ii(\xi)
		\right)d\mu_o(\eta).
		\]
		Since the inner integral is strictly positive for every
		$\eta\in\Lam$, we obtain $\mu_o(\Lam\smallsetminus A)=0$.
		This proves ergodicity with respect to $\mu_o$.
		Ergodicity with respect to $\mu_o^\ii$ follows similarly.

(3) Suppose that
$a:=\mu_o(\{\xi_0\})>0$, where $\xi_0=\theta(x_0)$.
Let $(\gamma_i)_{i\geq0}$ be the successive vertices of a
geodesic ray from $e$ to $x_0$, with $\gamma_0=e$ and
$|\gamma_i|=i$.
Choose $R$ sufficiently large for the shadow estimates of both
densities, with a common constant $C\geq1$.
By Lemma~\ref{lem:shadow-basic}(2) and
\eqref{eq:shadow-character},
\[
a
\leq \mu_o\bigl(\cO_R(e,\gamma_i)\bigr)
\leq C e^{\chi(\gamma_i)-\delta\dph(e,\gamma_i)}
\lVert\mu_o\rVert.
\]
Moreover,  applying 
 \eqref{eq:shadow-character} to $\mu^\ii$ gives
\[
\begin{aligned}
	\mu_o^\ii\bigl(\cO_R(e,\gamma_i^{-1})\bigr)
	\geq C^{-2}
	\frac{\lVert\mu_o^\ii\rVert}{\lVert\mu_o\rVert}\,a
	=:b>0.
\end{aligned}
\]

Set $S_i:=\cO_R(e,\gamma_i^{-1})$.
Since $|\gamma_i^{-1}|=i$, Lemma~\ref{lem:shrinking} gives
$\operatorname{diam}_{\varrho}S_i\to0$.
Choose $\eta_i\in S_i$. By compactness, the sequence $(\eta_i)$ has a limit point
$\eta_0\in\Lam$. 
For every $\varepsilon>0$, we then have
$S_i\subseteq\overline B_{\varrho}(\eta_0,\varepsilon)$
for infinitely many $i$, and hence
$\mu_o^\ii(\overline B_{\varrho}(\eta_0,\varepsilon))\ge b$.
Since $\mu_o^\ii$ is finite,
\[
\mu_o^\ii(\{\eta_0\})
=\lim_{n\to\infty}
\mu_o^\ii\bigl(\overline B_{\varrho}(\eta_0,1/n)\bigr)
\ge b>0.
\]
By Lemma~\ref{lem:density-basic}(3), for every
$\gamma\in\Gamma$ and $\xi\in\Lam$,
\[
\begin{aligned}
	\mu_o(\{\gamma\xi\})
	=((\gamma^{-1})_\ast\mu_o)(\{\xi\})
	=e^{\chi(\gamma)
		-\delta\varphi(\beta_\xi(\gamma^{-1}o,o))}
	\mu_o(\{\xi\}).
\end{aligned}
\]
Since $\Lam$ is infinite and every $\Gamma$-orbit is dense
\cite[Theorem~2.28 and Proposition~4.2(2)]{KB},
choose $\gamma_0\in\Gamma$ such that
$\xi'':=\gamma_0\xi_0\neq\xi':=\eta_0$.
Thus
$\mu_o^\ii(\{\xi'\})\mu_o(\{\xi''\})>0$.
By Definition~\ref{def:bms-measure}, we have
\[
\widetilde m\bigl(\{(\xi',\xi'')\}\bigr)
=
e^{\delta\varphi^\ii(\cG(\xi',\xi''))}
\mu_o^\ii(\{\xi'\})\mu_o(\{\xi''\})
>0.
\]
Set $O:=\Gamma\cdot(\xi',\xi'')$. Then $O$ is countable,
Borel and $\Gamma$-invariant. Hence (1) gives $\widetilde m(\Lambda_\Gamma^{(2)}\smallsetminus O)=0.$
Set
\[
\Pi':=\{\gamma\in\Gamma:\gamma\xi'=\xi'\},
\qquad
E:=\overline{\Pi'\xi''}\cup\{\xi'\}.
\]
Since $\Gamma$ is torsion-free, $\Pi'$ is trivial or infinite cyclic
(see \cite[Thm.~12.2(1)]{KB} and \cite[Ex.~2.4.3]{Calegari}).
If $\Pi'=\{e\}$, then $E$ is finite. Otherwise, write
$\Pi'=\langle h\rangle$. By north--south dynamics
\cite[Thm.~4.3]{KB},
\[
\overline{\Pi'\xi''}
\subseteq \{h^n\xi'':n\in\mathbb Z\}\cup\{h^+,h^-\}.
\]
Thus, $E$ is countable in either case and closed. Since $\Lam$ is uncountable
and $\Gamma\xi''$ is dense, choose $\gamma_1\in\Gamma$ with
$\gamma_1\xi''\notin E$. In particular, $\gamma_1\xi''\neq \xi'$ and  $\mu_o(\{\gamma_1\xi''\})>0$, and 
$O\cap(\{\xi'\}\times\Lam)
=\{\xi'\}\times(\Pi'\xi''),$ so $(\xi',\gamma_1\xi'')\notin O$.
Consequently,
\[
\begin{aligned}
	0
	&=\widetilde m(\Lambda_\Gamma^{(2)}\smallsetminus O)\\
	&\geq
	\widetilde m\bigl(\{(\xi',\gamma_1\xi'')\}\bigr)\\
	&=
	e^{\delta\varphi^\ii(\cG(\xi',\gamma_1\xi''))}
	\mu_o^\ii(\{\xi'\})\mu_o(\{\gamma_1\xi''\})
	>0,
\end{aligned}
\]
a contradiction.
Therefore, $\mu_o(\{\xi\})=0$ for every $\xi\in\Lam$.
Interchanging the densities shows that $\mu_o^\ii$ is also atomless.

(4) Conservativity follows from (2), (3), and
\cite[Prop.~1.6.6]{Aaronson}.
	\end{proof}


\subsection{Existence, Uniqueness, and Ergodicity}
	\begin{thm}\label{thm:A}
	Let $\Gamma<\sG$ be a Zariski-dense Borel--Anosov subgroup.
	Fix $\varphi\in\cL^{\ast}_\Gamma$ and  $\chi\in\Hom(\Gamma,\bR)$. Then
	$\delta_{\varphi,\chi}(\Gamma)\in (0, \infty)$,
	the triple $(\Gamma,\varphi,\chi)$ is of divergence
	type, and
	\begin{enumerate}
		\item There exists
		$\mu^{\varphi,\chi}\in
		\cM_{\varphi,\chi}
		\bigl(\delta_{\varphi,\chi}(\Gamma),\Gamma\bigr)$ with 
		$\supp\mu_x^{\varphi,\chi}=\Lam$ for $x\in\mathbb X.$
		
		\item If
		$\nu\in\cM_{\varphi,\chi}(\sigma,\Gamma)$ for some
		$\sigma>0$, then
		$\sigma=\delta_{\varphi,\chi}(\Gamma)$ and
		$\nu=c\,\mu^{\varphi,\chi}$ for some $c>0$.
		
		\item For every $x\in\mathbb X$,
		$\mu_x^{\varphi,\chi}$ is atomless and the
		$\Gamma$-action on $(\cF,\mu_x^{\varphi,\chi})$
		is ergodic and conservative. Moreover,  
		$\Lambda_{\mathrm{con}}(\Gamma)=\Lam,$ and
		$\mu_x^{\varphi,\chi}
		(\cF\smallsetminus\Lambda_{\mathrm{con}}(\Gamma))=0.$
	\end{enumerate}
\end{thm}

\begin{proof}

(1) Set $\delta:=\delta_{\varphi,\chi}(\Gamma)$.
By Proposition~\ref{prop:delta-props}(2),  we have $\delta\in (0, \infty)$.
Applying  Theorem~\ref{thm:existence}, we obtain
$\mu^{\varphi,\chi}\in\mathcal M_{\varphi,\chi}(\delta,\Gamma)$
with $\supp\mu_x^{\varphi,\chi}=\Lam$ for every $x\in\mathbb X$.
Applying Proposition~\ref{prop:rigidity}  yields
$Q_{\Gamma,\varphi,\chi}(\delta)
=\infty,$
so $(\Gamma,\varphi,\chi)$ is of divergence type.

(2) Let $\nu\in\mathcal M_{\varphi,\chi}(\sigma,\Gamma)$
for some $\sigma>0$. Applying
Corollary~\ref{cor:twisted-uniqueness} to $\nu$ and the
density $\mu^{\varphi,\chi}$ obtained in (1), we get $\sigma=\delta$,  $\nu=c\,\mu^{\varphi,\chi},$ with  $c:=\frac{\|\nu_o\|}{\|\mu_o^{\varphi,\chi}\|}>0.$

(3) Proposition~\ref{prop:delta-props}(1) gives
$\delta_{\varphi^\ii,-\chi}(\Gamma)=\delta$.
Applying Theorem~\ref{thm:existence} to
$(\varphi^\ii,-\chi)$, we obtain 
$\mu^\ii\in\mathcal M_{\varphi^\ii,-\chi}(\delta,\Gamma)$.
 Theorem~\ref{thm:HTS}\textup{(2)--(4)}
shows that $\mu_o^{\varphi,\chi}$ is atomless and that the
$\Gamma$-action on $(\cF,\mu_o^{\varphi,\chi})$ is ergodic
and conservative.
For every $x\in\mathbb X$,  \eqref{eq:conf} gives 
$\mu_x^{\varphi,\chi}\sim\mu_o^{\varphi,\chi}$.
Therefore, 
atomlessness, ergodicity, and conservativity  hold
for every $\mu_x^{\varphi,\chi}$.

Finally, let $\xi\in\Lam$ and set
$x:=\theta^{-1}(\xi)\in\partial_\infty\Gamma$.
Choose a geodesic ray $[e,x)$ with successive vertices
$(\gamma_i)_{i\geq0}$, where $\gamma_0=e$.
Then $|\gamma_i|=i\to\infty$.
Moreover, for any fixed $R>0$, we have
$[e,x)\cap B_w(\gamma_i,R)\ne\varnothing$ and hence
$x\in\cO_R(e,\gamma_i)$ for every $i\geq0$.
It follows that $(\gamma_i)$ converges radially to $x$,
so $\xi\in\Lambda_{\mathrm{con}}(\Gamma)$.
Since $\xi\in\Lam$ was arbitrary, we obtain
$\Lam\subseteq\Lambda_{\mathrm{con}}(\Gamma)\subseteq\Lam$,
and therefore $\Lambda_{\mathrm{con}}(\Gamma)=\Lam$.
Consequently, by (1), for every $y\in\mathbb X$, we have
$\mu_y^{\varphi,\chi}
(\mathcal F\smallsetminus \Lambda_{\mathrm{con}}(\Gamma))
=\mu_y^{\varphi,\chi}(\mathcal F\smallsetminus\Lam)=0.$
\end{proof}

	\section{Nilpotent covers: proof of Theorem \ref{thm:B}}
	\label{sec:nilpotent}
	
Throughout this section, let $\Gamma_0\lhd\Gamma$ be a nontrivial
normal subgroup with nilpotent quotient
\[
\rQ:=\Gamma/\Gamma_0.
\]
	Let $\pi:\Gamma\to \rQ$ denote the quotient map. Since \(\rQ\) is nilpotent, its lower central series
	\[
	\rQ=\rQ_0\trianglerighteq \rQ_1\trianglerighteq\cdots
	\trianglerighteq \rQ_n=\{e\},
	\quad
	\rQ_{k+1}=[\rQ,\rQ_k],
	\]
	for some finite $n$.
	  Write $\Gamma^{(k)}:=\pi^{-1}(\rQ_k).$
	Then
	\begin{equation}\label{eq:tower}
	\begin{gathered}
		\Gamma=\Gamma^{(0)}\unrhd\cdots
		\unrhd\Gamma^{(n)}=\Gamma_0,\quad
		[\Gamma^{(k-1)},\Gamma]\subseteq\Gamma^{(k)}
		\qquad(1\leq k\leq n),
	\end{gathered}
\end{equation}
	and \(\Gamma^{(k)}\lhd\Gamma\) for every \(k\). Moreover,
\[
\Gamma^{(k)}/\Gamma^{(k+1)}
\cong \rQ_k/\rQ_{k+1}
\qquad(0\leq k<n),
\]
so each successive quotient is abelian.


For $\chi\in\Hom(\Gamma,\bR)$, each restriction
$\chi|_{\Gamma^{(k)}}$ is $\Gamma$-conjugation invariant, since
\[
\chi(\gamma h\gamma^{-1})=\chi(h)
\qquad
(\gamma\in\Gamma,\ h\in\Gamma^{(k)}).
\]
Hence Sections~\ref{sec:densities}--\ref{sec:differentiation}
apply to $(\Gamma^{(k)},\chi|_{\Gamma^{(k)}})$.

	\subsection{Quasi-invariance along the tower}
	
	\begin{prop}\label{prop:quasiinv}
		Let $1\leq k\leq n$, $\sigma>0$, $\chi\in\Hom(\Gamma,\bR)$,
		and
		$\nu\in\cM_{\varphi,\chi|_{\Gamma^{(k)}}}
		(\sigma,\Gamma^{(k)})$.
		For every $\gamma\in\Gamma^{(k-1)}$, there is
		$C_\gamma\geq1$, depending only on
		$(\gamma,\sigma,\varphi,\chi,\Gamma^{(k)})$, such that
		\[
		C_\gamma^{-1}\nu_o
		\leq(\gamma^{-1})_\ast\nu_{\gamma o}
		\leq C_\gamma\nu_o.
		\]
		In particular, $\gamma_\ast\nu_o\sim\nu_o$  for every
		$\gamma\in\Gamma^{(k-1)}$.
	\end{prop}
	
	\begin{proof}
		Fix $\gamma\in\Gamma^{(k-1)}$ and set
		$\nu^\gamma_x:=(\gamma^{-1})_\ast\nu_{\gamma x}$
		for $x\in\mathbb X$. By Lemma~\ref{lem:translates}
		and positive rescaling, we have
	$\nu^\gamma\in
	\cM_{\varphi,\chi|_{\Gamma^{(k)}}}(\sigma,\Gamma^{(k)}).$
	For $\gamma'\in\Gamma$, put
	$c:=[\gamma,\gamma']=\gamma\gamma'\gamma^{-1}(\gamma')^{-1}$.
	Since $\pi(\gamma)\in\rQ_{k-1}$, we have
	$\pi(c)=[\pi(\gamma),\pi(\gamma')]\in[\rQ_{k-1},\rQ]=\rQ_k$,
	and hence $c\in\Gamma^{(k)}$.
	Moreover, since 
	$\chi(c)=0$, we have $c\in\Gamma^{(k)}\cap\ker\chi$, with
	$\gamma\gamma'=c\gamma'\gamma$.
	Thus
		Lemma~\ref{lem:density-basic}\textup{(1)} gives
		\[
		\begin{aligned}
			\left\lVert\nu^\gamma_{\gamma'o}\right\rVert
			=\left\lVert\nu_{\gamma\gamma'o}\right\rVert
			=e^{\chi(c)}\left\lVert\nu_{\gamma'\gamma o}\right\rVert
			=\left\lVert\nu_{\gamma'\gamma o}\right\rVert.
		\end{aligned}
		\]
		Therefore, by Lemma~\ref{lem:density-basic}\textup{(5)}, we obtain
		\[
		\left|
		\log\left\lVert\nu^\gamma_{\gamma'o}\right\rVert-\log \left\lVert\nu_{\gamma'o}\right\rVert
		\right|
		\leq
		\sigma\lVert\varphi\rVert
		\left\lVert\kappa(\gamma'o,\gamma'\gamma o)\right\rVert
		=
		\sigma\lVert\varphi\rVert\lVert\kappa(\gamma)\rVert.
		\]
		Set
		$\Theta_\gamma
		:=e^{\sigma\lVert\varphi\rVert\lVert\kappa(\gamma)\rVert}$, using  the above inequality 
		we have
		\[
		\Theta_\gamma^{-1}\lVert\nu_{\gamma'o}\rVert
		\leq\lVert\nu^\gamma_{\gamma'o}\rVert
		\leq\Theta_\gamma\lVert\nu_{\gamma'o}\rVert
		\qquad(\gamma'\in\Gamma).
		\]
		Fix $R\geq\max\{R_1,R_2\}$. Proposition~\ref{prop:masscomparison}
		therefore yields
		\[
		C_\gamma^{-1}\nu_o\leq\nu^\gamma_o\leq C_\gamma\nu_o,
		\qquad
		C_\gamma:=\Theta_\gamma C_RC_{6R}.
		\]
		Pushing forward by $\gamma$ and using \eqref{eq:conf}, we obtain
		proving 
		$\gamma_\ast\nu_o\sim\nu_o$.
	\end{proof}
	
	\subsection{The emerging character}
	\begin{prop}\label{prop:emerging}
		Let $\Gamma_0\lhd\Gamma$,
		$\Gamma_0\leq\mathsf H\leq\Gamma$, $\sigma>0$, and
		$\mu\in\cM_\varphi(\sigma,\Gamma_0)$.
		Suppose that the $\Gamma_0$-action on $(\cF,\mu_o)$ is ergodic
		and the measure class of $\mu_o$ is $\mathsf H$-invariant.
		Then there is a character $\chi'\in\Hom(\mathsf H,\bR)$,
		vanishing on $\Gamma_0$, such that
		\[
		h_\ast\mu_x=e^{-\chi'(h)}\mu_{hx}
		\qquad(h\in\mathsf H,\ x\in\mathbb X).
		\]
		Moreover,
		$$\chi'(h)=
		\log\left(\frac{\lVert\mu_{ho}\rVert}{\lVert\mu_o\rVert}\right).$$
	\end{prop}
	
\begin{proof}
	Fix $h\in\mathsf H$ and set $\nu_x^h:=(h^{-1})_\ast\mu_{hx}$ for $x\in\mathbb X$. For every $\gamma_0\in\Gamma_0$, normality gives $h\gamma_0h^{-1}\in\Gamma_0$, so the $\Gamma_0$-equivariance of $\mu$ yields $$(\gamma_0)_\ast\nu_x^h=(h^{-1})_\ast(h\gamma_0h^{-1})_\ast\mu_{hx}=(h^{-1})_\ast\mu_{h\gamma_0x}=\nu_{\gamma_0x}^h.$$
	Moreover, by \eqref{eq:conf} and Lemma~\ref{lem:beta-props}(1), we have $\frac{d\nu_x^h}{d\nu_o^h}=e^{-\sigma\varphi(\beta_\bullet(x,o))}$, proving that $\nu^h\in\cM_\varphi(\sigma,\Gamma_0)$.
	Since $\mu_{ho}\sim\mu_o$, by \eqref{eq:conf} and the measure class of $\mu_o$ is $\mathsf H$-invariant, we have $\nu_o^h=(h^{-1})_\ast\mu_{ho}\sim(h^{-1})_\ast\mu_o\sim\mu_o$. Hence,  we   write 
		\[
		\nu_o^h=D_h\mu_o,
		\qquad
	D_h:=\frac{d\nu^h_o}{d\mu_o}>0
	\quad\mu_o\text{-a.e.}
	\]
	For $\gamma_0\in\Gamma_0$, put $J_{\gamma_0}:=e^{-\sigma\varphi(\beta_\bullet(\gamma_0o,o))}$. By Lemma~\ref{lem:density-basic}\textup{(3)},
	\[
	(D_h\circ\gamma_0^{-1})J_{\gamma_0}\mu_o
	=(\gamma_0)_\ast(D_h\mu_o)
	=(\gamma_0)_\ast\nu_o^h
	=J_{\gamma_0}\nu_o^h
	=J_{\gamma_0}D_h\mu_o.
	\]
		Since $J_{\gamma_0}>0$, we have
		$D_h\circ\gamma_0^{-1}=D_h$ $\mu_o$-a.e.
		By the ergodicity of the $\Gamma_0$-action, $D_h=f(h)$
		$\mu_o$-a.e., for some $f(h)>0$. By \eqref{eq:conf}, we have
		\begin{equation*}
			(h^{-1})_\ast\mu_{hx}
			=\nu_x^h=f(h)\mu_x
			\qquad(x\in\mathbb X).
			\tag{*}\label{eq:translate-const}
		\end{equation*}
		For $h,h'\in\mathsf H$, this yields
		\[
		\begin{aligned}
			f(hh')\mu_x
			&=((hh')^{-1})_\ast\mu_{hh'x}\\
			&=((h')^{-1})_\ast
			\bigl((h^{-1})_\ast\mu_{h(h'x)}\bigr)\\
			&=f(h)f(h')\mu_x.
		\end{aligned}
		\]
		Since $\mu_x\ne0$, we obtain $f(hh')=f(h)f(h')$. Moreover, \eqref{eq:equiv} and \eqref{eq:translate-const} give $f(\gamma_0)=1$ for $\gamma_0\in\Gamma_0$. Thus $\chi':=\log f$ is a character vanishing on $\Gamma_0$. Finally, pushing \eqref{eq:translate-const} forward by $h$ and taking total masses at $x=o$, we obtain
		\[
		h_\ast\mu_x=e^{-\chi'(h)}\mu_{hx},
		\qquad
		\chi'(h)=\log\frac{\lVert\mu_{ho}\rVert}{\lVert\mu_o\rVert},
		\]
		which proves both assertions.
\end{proof}

Since $\rQ$ is nilpotent and hence amenable (see Appendix~\ref{app:amenable-covers}), we may fix a right-invariant mean
\[
\mathtt{M}\colon l^\infty(\Gamma_0\backslash\Gamma)\longrightarrow\bR,
\]
that is, a positive linear functional satisfying $\mathtt{M}(\mathbf1)=1$ and $\mathtt{M}(f\circ R_\gamma)=\mathtt{M}(f)$ for every $f\in l^\infty(\Gamma_0\backslash\Gamma)$ and $\gamma\in\Gamma$, where $[\gamma']:=\Gamma_0\gamma'$ and $R_\gamma([\gamma'])=[\gamma'\gamma]$.


\begin{lemma}\label{lem:kill}
	Let $2\leq k\leq n$, $\sigma>0$, and
	$\mu\in\cM_\varphi(\sigma,\Gamma^{(k)})$.
	Suppose that a character
	$\chi_{k-1}\in\Hom(\Gamma^{(k-1)},\bR)$ satisfies
	\[
	\gamma_\ast\mu_x
	=e^{-\chi_{k-1}(\gamma)}\mu_{\gamma x}
	\qquad(\gamma\in\Gamma^{(k-1)},\ x\in\mathbb X).
	\]
	Then $\chi_{k-1}=0$, and hence
	$\mu\in\cM_\varphi(\sigma,\Gamma^{(k-1)})$.
\end{lemma}

\begin{proof}
Since $\Gamma_0\subseteq\Gamma^{(k)}$, we have $\mu\in\cM_\varphi(\sigma,\Gamma_0)$. Using the right-invariant mean $\mathtt{M}$ fixed above, for $\gamma,\gamma'\in\Gamma$ we write
\[
f_\gamma([\gamma'])
:=\log\lVert\mu_{\gamma'\gamma o}\rVert
-\log\lVert\mu_{\gamma'o}\rVert,
\qquad
\chi(\gamma):=\mathtt{M}(f_\gamma).
\]
	To check that $f_\gamma$ is well defined, note that for $\gamma_0\in\Gamma_0$ and $x\in\mathbb X$, \eqref{eq:equiv} gives
	\[
	\lVert\mu_{\gamma_0x}\rVert
	=\lVert(\gamma_0)_\ast\mu_x\rVert
	=\lVert\mu_x\rVert.
	\]
	Thus, $f_\gamma$ is well defined. Moreover, Lemma~\ref{lem:density-basic}\textup{(5)} gives
	\[
	|f_\gamma([\gamma'])|
	\leq\sigma\lVert\varphi\rVert\,
	\lVert\kappa(\gamma'o,\gamma'\gamma o)\rVert
	=\sigma\lVert\varphi\rVert\,
	\lVert\kappa(\gamma)\rVert,
	\]
	so $f_\gamma\in l^\infty(\Gamma_0\backslash\Gamma)$. Now, since $f_{\gamma\gamma'}=f_\gamma+f_{\gamma'}\circ R_\gamma$, the linearity and right-invariance of $\mathtt{M}$ yield
	\[
	\chi(\gamma\gamma')
	=\mathtt{M}(f_\gamma)
	+\mathtt{M}(f_{\gamma'}\circ R_\gamma)
	=\chi(\gamma)+\chi(\gamma').
	\]
Finally, for $\gamma_0\in\Gamma_0$ and $\gamma\in\Gamma$, normality gives $\gamma\gamma_0\gamma^{-1}\in\Gamma_0$. Hence, by \eqref{eq:equiv},
\[
f_{\gamma_0}([\gamma])
=\log\frac{
	\lVert\mu_{(\gamma\gamma_0\gamma^{-1})\gamma o}\rVert}
{\lVert\mu_{\gamma o}\rVert}
=0.
\]
Consequently, $\chi(\gamma_0)=\mathtt{M}(f_{\gamma_0})=0$. Thus $\chi\in\Hom(\Gamma,\bR)$ with $\chi|_{\Gamma_0}=0.$
Fix $\gamma\in\Gamma^{(k-1)}$ and $\gamma'\in\Gamma$. Since $\gamma\gamma'=[\gamma,\gamma']\gamma'\gamma$ with $[\gamma,\gamma']\in\Gamma^{(k)}$, the $\Gamma^{(k)}$-equivariance of $\mu$ gives $\lVert\mu_{\gamma\gamma'o}\rVert=\lVert\mu_{\gamma'\gamma o}\rVert$. By hypothesis, we obtain
\[
\begin{aligned}
	f_\gamma([\gamma'])
	&=\log\frac{\lVert\mu_{\gamma\gamma'o}\rVert}
	{\lVert\mu_{\gamma'o}\rVert}\\
	&=\log\frac{
		e^{\chi_{k-1}(\gamma)}\lVert\mu_{\gamma'o}\rVert}
	{\lVert\mu_{\gamma'o}\rVert}\\
	&=\chi_{k-1}(\gamma).
\end{aligned}
\]
Thus $f_\gamma$ is constant, and $\mathtt{M}(\mathbf1)=1$ gives
\[
\chi(\gamma)=\mathtt{M}(f_\gamma)
=\chi_{k-1}(\gamma)
\qquad(\gamma\in\Gamma^{(k-1)}).
\]
Since $k\geq2$ and $\chi$ vanishes on $\Gamma_0$ and $[\Gamma,\Gamma]$, we have $\Gamma^{(k-1)}\subseteq\Gamma^{(1)}=\Gamma_0[\Gamma,\Gamma]\subseteq\ker\chi$. Hence, $\chi_{k-1}=\chi|_{\Gamma^{(k-1)}}=0$, so $\mu\in\cM_\varphi(\sigma,\Gamma^{(k-1)})$.
\end{proof}
	
	\subsection{The two directions}
	
\begin{prop}\label{prop:descent}
	Let $\sigma>0$ and $\mu\in\cM_\varphi(\sigma,\Gamma_0)$. If the $\Gamma_0$-action on $(\cF,\mu_o)$ is ergodic, then there is a character $\chi\in\Hom(\Gamma,\bR)$ with $\chi|_{\Gamma_0}=0$ such that $\mu\in\cM_{\varphi,\chi}(\sigma,\Gamma)$ and $\sigma=\delta_{\varphi,\chi}(\Gamma)$.
\end{prop}

\begin{proof}
	For the tower \eqref{eq:tower} with $n\geq1$, we prove $\mu\in\cM_\varphi(\sigma,\Gamma^{(k)})$ for $k=n,\ldots,1$ by downward induction. Let the case $k=n$ holds by assumption.
	Suppose that  $\mu\in\cM_\varphi(\sigma,\Gamma^{(k)})$ with $k\ge 2$. Since $\Gamma_0\subseteq\Gamma^{(k)}$, the $\Gamma^{(k)}$-action on $(\cF,\mu_o)$ is ergodic. Proposition~\ref{prop:quasiinv} with $\chi=0$ makes the
	measure class of $\mu_o$ invariant under $\Gamma^{(k-1)}$. 
	Hence, applying  Proposition~\ref{prop:emerging} to $(\Gamma^{(k)},\Gamma^{(k-1)})$, gives a character $\chi_{k-1}\in\Hom(\Gamma^{(k-1)},\bR)$ vanishing on $\Gamma^{(k)}$ such that
	\[
	\gamma_\ast\mu_x
	=e^{-\chi_{k-1}(\gamma)}\mu_{\gamma x}
	\qquad(\gamma\in\Gamma^{(k-1)},\ x\in\mathbb X).
	\]
	Then Lemma~\ref{lem:kill} gives $\chi_{k-1}=0$, and hence $\mu\in\cM_\varphi(\sigma,\Gamma^{(k-1)})$. Iterating for $k=n,\ldots,2$, we obtain $\mu\in\cM_\varphi(\sigma,\Gamma^{(1)})$.
	 Proposition~\ref{prop:quasiinv} at $k=1$ makes the measure
	class of $\mu_o$ $\Gamma$-invariant. Since the
	$\Gamma^{(1)}$-action is ergodic,
	Proposition~\ref{prop:emerging}, applied to
	$(\Gamma^{(1)},\Gamma)$, gives a character $\chi\in\Hom(\Gamma,\bR)$ with $\chi|_{\Gamma^{(1)}}=0$ such that $\mu\in\cM_{\varphi,\chi}(\sigma,\Gamma)$.
	Since $\Gamma_0\subseteq\Gamma^{(1)}$, we have
	$\chi|_{\Gamma_0}=0$.
	Finally, Proposition~\ref{prop:rigidity} gives
	$\sigma=\delta_{\varphi,\chi}(\Gamma)$.
\end{proof}

\begin{prop}\label{prop:ascent}
	Let $\sigma>0$, $\chi\in\Hom(\Gamma,\bR)$ with
	$\chi|_{\Gamma_0}=0$, and
	$\mu\in\cM_{\varphi,\chi}(\sigma,\Gamma)$.
	Then $\sigma=\delta_{\varphi,\chi}(\Gamma)$,
	$\mu\in\cM_\varphi(\sigma,\Gamma_0)$, and the
	$\Gamma_0$-action on $(\cF,\mu_o)$ is ergodic.
\end{prop}

\begin{proof}
	The dimension equality follows from
	Proposition~\ref{prop:rigidity}, and since $\chi$ vanishes on $\Gamma_0$, by \eqref{eq:equiv},   $\mu\in\cM_\varphi(\sigma,\Gamma_0)$.
	Let $A\subseteq\cF$ be a $\Gamma_0$-invariant Borel set
	with $\mu_o(A)>0$, and set
	$\nu_x:=\mathbb1_A\mu_x$ for $x\in\mathbb X$.
	By \eqref{eq:conf}, we have $0<\lVert\nu_x\rVert\leq\lVert\mu_x\rVert<\infty$.
	Also, $\supp\nu_x\subseteq\Lam$.
	We prove by downward induction that
	\[
	\mu_o\left(\gamma A\triangle A\right)=0
	\qquad(\gamma\in\Gamma^{(k)},\ k=n,\ldots,0).
	\]
	The case $k=n$ holds by assumption. Now suppose the claim holds for some $k\geq1$.
	For $\gamma\in\Gamma^{(k)}$, \eqref{eq:conf} gives
	$\mu_{\gamma x}(\gamma A\triangle A)=0$.
	Hence, by \eqref{eq:equiv},
	\[
	\begin{aligned}
		\gamma_\ast\nu_x
		=\mathbb1_{\gamma A}\,\gamma_\ast\mu_x
		=e^{-\chi(\gamma)}\mathbb1_{\gamma A}\mu_{\gamma x}
		=e^{-\chi(\gamma)}\mathbb1_A\mu_{\gamma x}
		=e^{-\chi(\gamma)}\nu_{\gamma x}.
	\end{aligned}
	\]
	Thus,
	$\nu\in\cM_{\varphi,\chi|_{\Gamma^{(k)}}}
	(\sigma,\Gamma^{(k)})$.
	For $\gamma\in\Gamma^{(k-1)}$, equation~\eqref{eq:equiv}
	also gives
	\[
	\begin{aligned}
		(\gamma^{-1})_\ast\nu_{\gamma o}
		=(\gamma^{-1})_\ast(\mathbb1_A\mu_{\gamma o})
		=\mathbb1_{\gamma^{-1}A}
		(\gamma^{-1})_\ast\mu_{\gamma o}
		=e^{\chi(\gamma)}\mathbb1_{\gamma^{-1}A}\mu_o.
	\end{aligned}
	\]
	Proposition~\ref{prop:quasiinv} therefore yields
	\[
	C_\gamma^{-1}\mathbb1_A\mu_o
	\leq e^{\chi(\gamma)}\mathbb1_{\gamma^{-1}A}\mu_o
	\leq C_\gamma\mathbb1_A\mu_o
	\]
	for some $C_\gamma\geq1$.
	Evaluating the upper bound on $A^c$ and the lower bound on $(\gamma^{-1}A)^c$ gives $\mu_o(\gamma^{-1}A\smallsetminus A)=\mu_o(A\smallsetminus\gamma^{-1}A)=0$. 
	Thus $\mu_o\left(\gamma^{-1}A\triangle A\right)=0$, and replacing $\gamma$ by $\gamma^{-1}$ yields $\mu_o\left(\gamma A\triangle A\right)=0$ for every $\gamma\in\Gamma^{(k-1)}$.
	Iterating for $k=n,\ldots,1$, we obtain $\mu_o(\gamma A\triangle A)=0$ for every $\gamma\in\Gamma^{(0)}=\Gamma$. At $k=0$, the same equivariance calculation gives $\nu\in\cM_{\varphi,\chi}(\sigma,\Gamma)$.
	Finally, Corollary~\ref{cor:twisted-uniqueness} gives $\nu=c\mu$, where $c=\frac{\mu_o(A)}{\lVert\mu_o\rVert}>0$. Since $\nu_o=\mathbb1_A\mu_o$, we have $0=\nu_o(\cF\smallsetminus A)=c\mu_o(\cF\smallsetminus A)$. Hence, $\mu_o(\cF\smallsetminus A)=0$, proving that the $\Gamma_0$-action on $(\cF,\mu_o)$ is ergodic.
\end{proof}

	\subsection{Proof of Theorem \ref{thm:B}}

	For every $\sigma>0$, Propositions~\ref{prop:descent} and~\ref{prop:ascent} give
	\[
	\left\{
	\mu\in\cM_\varphi(\sigma,\Gamma_0):
	\Gamma_0\text{-action on }(\cF,\mu_o)\text{ is ergodic}
	\right\}
	=
	\bigcup_{\substack{\chi\in\Hom(\Gamma,\bR)\\
			\chi|_{\Gamma_0}=0}}
	\cM_{\varphi,\chi}(\sigma,\Gamma).
	\]
	By Proposition~\ref{prop:rigidity}, $\cM_{\varphi,\chi}(\sigma,\Gamma)\ne\varnothing$ implies $\sigma=\delta_{\varphi,\chi}(\Gamma)$.
Let $\mu$ be an ergodic $\Gamma_0$-density of positive dimension,
and denote by $\chi_\mu$ the character obtained from
Proposition~\ref{prop:descent}.
Then Lemma~\ref{lem:density-basic}\textup{(1)} gives
\[
\chi_\mu(\gamma)
=\log\frac{\lVert\mu_{\gamma o}\rVert}{\lVert\mu_o\rVert},
\qquad
\chi_{t\mu}(\gamma)
=\log\frac{t\lVert\mu_{\gamma o}\rVert}{t\lVert\mu_o\rVert}
=\chi_\mu(\gamma)
\quad(t>0).
\]
	Thus $\mu\mapsto\chi_\mu$ descends to the quotient by positive scaling.
	If $\chi_\mu=\chi_\nu=\chi$, then Proposition~\ref{prop:descent} gives $\mu,\nu\in\cM_{\varphi,\chi}\bigl(\delta_{\varphi,\chi}(\Gamma),\Gamma\bigr)$. Corollary~\ref{cor:twisted-uniqueness} therefore yields
	$	\mu=\frac{\lVert\mu_o\rVert}{\lVert\nu_o\rVert}\,\nu,$
	proving injectivity.
	Conversely, let $\chi\in\Hom(\Gamma,\bR)$ with $\chi|_{\Gamma_0}=0$. By Proposition~\ref{prop:delta-props}\textup{(2)} and Theorem~\ref{thm:existence}, there exists $\mu^{\varphi,\chi}\in\cM_{\varphi,\chi}\bigl(\delta_{\varphi,\chi}(\Gamma),\Gamma\bigr)$. Proposition~\ref{prop:ascent} shows that $\mu^{\varphi,\chi}$ is an ergodic $\Gamma_0$-density, and Lemma~\ref{lem:density-basic}\textup{(1)} gives
	\[
	\chi_{\mu^{\varphi,\chi}}(\gamma)
	=\log\frac{\lVert\mu_{\gamma o}^{\varphi,\chi}\rVert}
	{\lVert\mu_o^{\varphi,\chi}\rVert}
	=\log\frac{e^{\chi(\gamma)}\lVert\mu_o^{\varphi,\chi}\rVert}
	{\lVert\mu_o^{\varphi,\chi}\rVert}
	=\chi(\gamma).
	\]
	This proves surjectivity.

	
	Finally, by Proposition~\ref{prop:descent}, the dimension of every
	ergodic $\Gamma_0$-density $\mu$ is
	$\delta_{\varphi,\chi_\mu}(\Gamma)$, where $\chi_\mu$ vanishes on $\Gamma_0$. Conversely,
	Theorem~\ref{thm:existence} and Proposition~\ref{prop:ascent}
	give an ergodic $\Gamma_0$-density of dimension
	$\delta_{\varphi,\chi}(\Gamma)$ for every
	$\chi\in\Hom(\Gamma,\bR)$ vanishing on $\Gamma_0$.

	\subsection{The growth form and the norm counting}\label{ss:growthform}

	We apply our classification to Quint's growth form \cite[\S3.1.2]{Quint02a}.

	For $v\in\cL_\Gamma\smallsetminus\{0\}$, 
	\[
	\psi_\Gamma(v)
	=
	\lVert v\rVert
	\inf_{\text{open cones }\cC\ni v}
	\limsup_{T\to\infty}
	\frac{1}{T}
	\log\left|\left\{
	\gamma\in\Gamma:
	\kappa(\gamma)\in\cC,\
	\lVert\kappa(\gamma)\rVert\leq T
	\right\}\right|,
	\]
	where the cones are open in $\fk a$. Write 
	$\psi_\Gamma(0):=0$ and $\psi_\Gamma(v):=-\infty$ for
	$v\notin\cL_\Gamma$. By \cite[Thm.~4.2.2]{Quint02a},
	$\psi_\Gamma$ is $1$-homogeneous, concave, upper
	semicontinuous, finite and non-negative on $\cL_\Gamma$,
	and strictly positive on its interior. 
	By \cite[Lem.~3.1.3]{Quint02a} and
	Lemma~\ref{lem:abscissa}, for every
	$\varphi\in\cL_\Gamma^{\ast}$,
	\begin{equation}\label{eq:quint-sup}
		\delta_\varphi(\Gamma)
		=
		\sup_{v\in\cL_\Gamma\smallsetminus\{0\}}
		\frac{\psi_\Gamma(v)}{\varphi(v)}.
	\end{equation}
	Since inversion maps the elements counted in $\cC$ bijectively
	onto those counted in $\ii\cC$ and preserves $\lVert\kappa\rVert$
	by \eqref{eq:kappa-inverse}, the definition of $\psi_\Gamma$
	gives $\psi_\Gamma\circ\ii=\psi_\Gamma$.
	By \cite[Thm.~7.14]{Sam14}, we have
	$\Theta_\Gamma\in\mathcal L_\Gamma^{\ast}$ and
	$\delta_{\Theta_\Gamma}(\Gamma)=1$.
	Moreover, since $\psi_\Gamma\circ\ii=\psi_\Gamma$ and
	$\ii$ is an isometry, $\Theta_\Gamma\circ\ii$ is also a
	minimum-norm element of $D_\Gamma$. Uniqueness of the minimum-norm element therefore gives
	$\Theta_\Gamma\circ\ii=\Theta_\Gamma$.

	Equip $\fk a^\ast$ with the dual Euclidean norm and set
	$D_\Gamma:=\{\ell\in\fk a^\ast:\ell\geq\psi_\Gamma\}$.
	By \cite[Cor.~3.3.5]{Quint02a}, this non-empty closed convex
	set has a unique element of minimal norm, which we denote
	by $\Theta_\Gamma$.
	
	\begin{cor}\label{cor:D}
		Let $\Gamma<\sG$ be a Zariski-dense Borel--Anosov subgroup,
		and let $\Gamma_0\lhd\Gamma$ with
		$\Gamma/\Gamma_0$ nilpotent. Then:
		\begin{enumerate}
			\item There is a unique, up to a positive scalar, 
			$\Theta_\Gamma$-conformal density
			$\mu^\Theta=(\mu_x^\Theta)_{x\in\mathbb X}$
			of dimension $1$ for $\Gamma$.
			For every $x\in\mathbb X$, the measure $\mu_x^\Theta$
			is atomless, the $\Gamma$-action on $(\cF,\mu_x^\Theta)$
			is ergodic and conservative, and  they give full measure to
			$\Lambda_{\mathrm{con}}(\Gamma)$.
			
			\item The ergodic $\Theta_\Gamma$-conformal densities of
			$\Gamma_0$ of dimension $1$ are precisely the twisted
			$(\Theta_\Gamma,\chi)$-conformal densities of $\Gamma$ for
			$\chi\in\Hom(\Gamma/\Gamma_0,\bR)$ satisfying  $\delta_{\Theta_\Gamma,\chi}(\Gamma)=1$, and 
			every such character satisfies
			$|\chi(\gamma)|\leq\Theta_\Gamma(\mathsf{Jd}_\gamma)$
			for all $\gamma\in\Gamma$.
			
			\item The density $\mu^\Theta$ is, up to a positive scalar,
			the unique ergodic $\Theta_\Gamma$-conformal density
			of dimension $1$ for $\Gamma_0$.
			In particular, the $\Gamma_0$-action on
			$(\cF,\mu_o^\Theta)$ is ergodic.
		\end{enumerate}
	\end{cor}

	\begin{proof}
	
				By the preceding discussion, we have $\Theta_\Gamma\in\cL_\Gamma^{\ast}$,
				$\Theta_\Gamma\circ\ii=\Theta_\Gamma$, and
				$\delta_{\Theta_\Gamma}(\Gamma)=1$.
				
				\textup{(1)} Theorem~\ref{thm:A} with
				$\varphi = \Theta_\Gamma, \chi =0$ gives
				$\mu^\Theta\in\cM_{\Theta_\Gamma}(1,\Gamma)$,
				together with all the stated properties.
				
				\textup{(2)} Theorem~\ref{thm:B}  with
				$\varphi=\Theta_\Gamma$ and $\sigma=1$ gives the asserted
				classification. Moreover, Proposition~\ref{prop:rigidity}
				gives $\delta_{\Theta_\Gamma,\chi}(\Gamma)=1$ whenever
				$\cM_{\Theta_\Gamma,\chi}(1,\Gamma)\neq\emptyset$.
				Since $\Theta_\Gamma\circ\ii=\Theta_\Gamma$,
				Corollary~\ref{prop:jordan-constraint} yields
				 $|\chi(\gamma)|\leq\Theta_\Gamma(\mathsf{Jd}_\gamma)$.
				
			\textup{(3)} Finally, let $\nu\in\cM_{\Theta_\Gamma}(1,\Gamma_0)$ be ergodic. By \textup{(2)}, there is a character $\chi\in\Hom(\Gamma,\bR)$ vanishing on $\Gamma_0$ such that $\nu\in\cM_{\Theta_\Gamma,\chi}(1,\Gamma)$ and $\delta_{\Theta_\Gamma,\chi}(\Gamma)=1=\delta_{\Theta_\Gamma}(\Gamma)$. Since $\Theta_\Gamma$ is symmetric, Corollary~\ref{prop:reversibility} gives $\chi=0$. Thus,  $\nu\in\cM_{\Theta_\Gamma}(1,\Gamma)$, and \textup{(1)} yields $\nu=c\mu^\Theta$ for some $c>0$. Moreover, applying \textup{(2)} with $\chi=0$ shows that the $\Gamma_0$-action on $(\cF,\mu_o^\Theta)$ is ergodic.
	
		\end{proof}
		
			\appendix
			
		\section{Applications to normal subgroups and amenable covers}\label{app:amenable-covers}

Under the standing assumptions on $\Gamma$ and $\varphi$, we use
our conformal densities to adapt the arguments of \cite{PPS}
to higher rank, comparing critical exponents for amenable normal
covers and obtaining a lower bound for arbitrary infinite normal
subgroups.

Recall that a countable group $\rQ$ is \emph{amenable} if it admits
a \emph{right-invariant mean}, that is, a positive linear functional
$\mathtt M\colon l^\infty(\rQ)\longrightarrow\bR$ with
$\mathtt M(\mathbf1)=1$ and
$\mathtt M(f\circ R_q)=\mathtt M(f)$ for every
$f\in l^\infty(\rQ)$ and $q\in \rQ$, where $R_q(x):=xq$.
In particular, nilpotent groups are amenable \cite{Green}.


Suppose first that $\Gamma_0\lhd\Gamma$ is infinite with amenable
quotient. By Proposition~\ref{prop:delta-props}\textup{(2),(3)}
and Theorem~\ref{thm:existence}, choose
$\mu^0\in\cM_\varphi(\delta_\varphi(\Gamma_0),\Gamma_0)$,
and fix a right-invariant mean $\mathtt M$ on
$l^\infty(\Gamma_0\backslash\Gamma)$.
	
	For $\gamma\in\Gamma$, set
	\[
	f_\gamma([\gamma'])
	:=
	\log\frac{\lVert\mu^0_{\gamma'\gamma o}\rVert}
	{\lVert\mu^0_{\gamma'o}\rVert}.
	\]
	As in the first part of the proof of Lemma~\ref{lem:kill},
	$f_\gamma$ is well defined and bounded, and
	$\chi_0(\gamma):=\mathtt M(f_\gamma)$ defines a character
	of $\Gamma$ vanishing on $\Gamma_0$.

		\begin{prop}\label{prop:averaged-density}
		There exists $\nu\in\cM_\varphi(\delta_\varphi(\Gamma_0),\Gamma_0)$ such that
		\[
		\lVert\nu_o\rVert=1,
		\qquad
		\lVert\nu_{\gamma o}\rVert\geq e^{\chi_0(\gamma)}
		\qquad(\gamma\in\Gamma).
		\]
	\end{prop}

	\begin{proof}
		For $x\in\mathbb X$ and $\phi\in\mathcal C(\cF)$, set
		\[
		f_{x,\phi}([\gamma])
		:=
		\frac{1}{\lVert\mu^0_{\gamma o}\rVert}
		\int_{\cF}
		\phi\,d\bigl((\gamma^{-1})_\ast\mu^0_{\gamma x}\bigr).
		\]
		By \eqref{eq:equiv}, $f_{x,\phi}$ is well defined.
		Moreover, Lemma~\ref{lem:density-basic}\textup{(5)},
		with $b_x:=\delta_0\lVert\varphi\rVert\lVert\kappa(x,o)\rVert$,
		gives
		\[
		\lVert f_{x,\phi}\rVert_\infty
		\leq e^{b_x}\lVert\phi\rVert_\infty,
		\qquad
		f_{x,\mathbf1}\geq e^{-b_x}\mathbf1.
		\]
		Thus $\phi\mapsto\mathtt M(f_{x,\phi})$ is a bounded
		positive linear functional on $\mathcal C(\cF)$.
		By the Riesz representation theorem, there is a finite
		Borel measure $\nu_x$ supported on $\Lam$, such that
		\[
		\int_{\cF}\phi\,d\nu_x=\mathtt M(f_{x,\phi})
		\qquad(\phi\in\mathcal C(\cF)).
		\]
			Thus, 
		$\lVert\nu_x\rVert\geq e^{-b_x}>0$.
	Also, $f_{o,\mathbf1}=\mathbf1$, and hence
		$\lVert\nu_o\rVert=\mathtt M(\mathbf1)=1$.
			For $x,y\in\mathbb X$, by~\eqref{eq:conf} and Lemma~\ref{lem:beta-props}(1), 
		\[
		(\gamma^{-1})_\ast\mu^0_{\gamma x}
		=
		q_{x,y}\,(\gamma^{-1})_\ast\mu^0_{\gamma y}
		\qquad(\gamma\in\Gamma),
		\]
		where  $q_{x,y}:=e^{-\delta_0\varphi(\beta_\bullet(x,y))}.$
		Consequently, $f_{x,\phi}=f_{y,\phi q_{x,y}}$, and hence applying $\mathtt{M}$ 
		yields $\nu_x=q_{x,y}\nu_y$.
		Moreover, for $\gamma_0\in\Gamma_0$, normality and
		\eqref{eq:equiv} give
		\[
		\begin{aligned}
			(\gamma^{-1})_\ast\mu^0_{\gamma\gamma_0x}
			=
			(\gamma^{-1})_\ast
			(\gamma\gamma_0\gamma^{-1})_\ast\mu^0_{\gamma x}
			=
			(\gamma_0)_\ast
			\bigl((\gamma^{-1})_\ast\mu^0_{\gamma x}\bigr).
		\end{aligned}
		\]
		Thus, $f_{\gamma_0x,\phi}=f_{x,\phi\circ\gamma_0}$,
		and applying $\mathtt M$ yields
		$\nu_{\gamma_0x}=(\gamma_0)_\ast\nu_x$.
		Therefore, $\nu\in\cM_\varphi(\delta_\varphi(\Gamma_0),\Gamma_0)$.
		Finally, fix $\gamma\in\Gamma$ and put
		$c:=\mathtt M(f_\gamma)=\chi_0(\gamma)$.
		Since $f_{\gamma o,\mathbf1}=e^{f_\gamma}$ and
		$e^t\geq e^c(1+t-c)$, positivity and normalization
		of $\mathtt M$ give
		\[
		\begin{aligned}
			\lVert\nu_{\gamma o}\rVert
			&=\mathtt M(f_{\gamma o,\mathbf1})
			=\mathtt M(e^{f_\gamma})\\
			&\geq e^c\mathtt M(\mathbf1+f_\gamma-c\mathbf1)\\
			&=e^c\bigl(1+\mathtt M(f_\gamma)-c\bigr)
			=e^{\chi_0(\gamma)}.
		\end{aligned}
		\]
	\end{proof}
	The symmetric case of the following corollary was obtained
	in \cite[Thm.~1.3]{DOT}.
In our setting, we give an alternative proof.
	
	\begin{cor}\label{cor:amenable-cover}
		Let $\Gamma_0\lhd\Gamma$ be an infinite normal subgroup
		with amenable quotient. Then there exists
		$\chi_0\in\Hom(\Gamma,\bR)$ vanishing on $\Gamma_0$ such that
	$	\delta_{\varphi,\chi_0}(\Gamma)
	=
	\delta_\varphi(\Gamma_0).$
		If $\varphi$ is symmetric, then
		$\delta_\varphi(\Gamma_0)=\delta_\varphi(\Gamma)$.
	\end{cor}
	\begin{proof}
		Let $\chi_0$ be the character constructed above and choose
		$\nu\in\cM_\varphi(\delta_\varphi(\Gamma_0),\Gamma_0)$ as in
		Proposition~\ref{prop:averaged-density}.
		The mass bound for $\nu$ and Corollary~\ref{prop:poincare}
		give, for every $s>\delta_\varphi(\Gamma_0)$,
		\[
		\begin{aligned}
			Q_{\Gamma,\varphi,\chi_0}(s)
			=
			\sum_{\gamma\in\Gamma}
			e^{\chi_0(\gamma)-s\dph(e,\gamma)}
			\leq
			\sum_{\gamma\in\Gamma}
			\lVert\nu_{\gamma o}\rVert
			e^{-s\dph(e,\gamma)}
			<\infty.
		\end{aligned}
		\]
		Hence, Lemma~\ref{lem:abscissa} yields
		$\delta_{\varphi,\chi_0}(\Gamma)\leq\delta_\varphi(\Gamma_0)$.
		Since $\chi_0$ vanishes on $\Gamma_0$,
		Proposition~\ref{prop:delta-props}(3) gives
	$\delta_\varphi(\Gamma_0)
		\leq\delta_{\varphi,\chi_0}(\Gamma),$
		proving the first assertion.
		
		If $\varphi^\ii=\varphi$, then
		Proposition~\ref{prop:delta-props}(3),(4) gives
		\[
		\delta_\varphi(\Gamma)
		\leq
		\delta_{\varphi,\chi_0}(\Gamma)
		=
		\delta_\varphi(\Gamma_0)
		\leq
		\delta_\varphi(\Gamma),
		\]
		proving the final assertion.
		
	If $\varphi^\ii=\varphi$,
		Proposition~\ref{prop:delta-props}(4) gives
		$\delta_\varphi(\Gamma)
		\leq\delta_{\varphi,\chi_0}(\Gamma)
		=\delta_\varphi(\Gamma_0)$.
		Proposition~\ref{prop:delta-props}(3) gives
		$\delta_\varphi(\Gamma_0)\leq\delta_\varphi(\Gamma)$,
		proving the final assertion.
	\end{proof}

The following lower bound can also be deduced from
\cite[Lem.~4.4 and Rem.~6.10]{DOT}.
In our setting, we give an alternative proof.
	\begin{prop}\label{prop:normal-lower-bound}
	Let $\Gamma_0\lhd\Gamma$ be an infinite normal subgroup
	and set $\overline{\varphi}:=\tfrac12(\varphi+\varphi^\ii)$.
	Then
	\[
	\delta_\varphi(\Gamma_0)
	\geq
	\delta_{\varphi+\varphi^\ii}(\Gamma)
	=
	\frac12\,\delta_{\bar\varphi}(\Gamma).
	\]
	In particular, if $\varphi$ is symmetric, then
	$\delta_\varphi(\Gamma_0)\geq
	\tfrac12\delta_\varphi(\Gamma)$.
\end{prop}

\begin{proof}
	By Theorem~\ref{thm:existence}, choose
	$\mu^0\in\cM_\varphi(\delta_\varphi(\Gamma_0),\Gamma_0)$
	with $\lVert\mu_o^0\rVert=1$.
	Corollary~\ref{cor:shadow-consequences}(3) gives a constant $C\geq1$ such that
\[
e^{-\delta_\varphi(\Gamma_0)d_{\varphi^\ii}(e,\gamma)}
\le C\lVert\mu_{\gamma o}^0\rVert
\qquad(\gamma\in\Gamma).
\]
Moreover, Lemma~\ref{lem:dphi-qi}(1), applied to $\varphi^\ii$,
gives $B\geq0$ such that $d_{\varphi^\ii}(e,\gamma)\geq-B$
for every $\gamma\in\Gamma$.
Thus, by Corollary~\ref{prop:poincare}, for every
$s>\delta_\varphi(\Gamma_0)$,
	\[
	\begin{aligned}
		Q_{\Gamma,\varphi+\varphi^\ii,0}(s)
		&=
		\sum_{\gamma\in\Gamma}
		e^{-s\dph(e,\gamma)-s d_{\varphi^\ii}(e,\gamma)}\\
		&\leq
		e^{(s-\delta_\varphi(\Gamma_0))B}
		\sum_{\gamma\in\Gamma}
		e^{-s\dph(e,\gamma)
			-\delta_\varphi(\Gamma_0)d_{\varphi^\ii}(e,\gamma)}\\
		&\leq
		C e^{(s-\delta_\varphi(\Gamma_0))B}
		\sum_{\gamma\in\Gamma}
		\lVert\mu_{\gamma o}^0\rVert
		e^{-s\dph(e,\gamma)}
		<\infty.
	\end{aligned}
	\]
	Lemma~\ref{lem:abscissa} therefore yields
	$\delta_{\varphi+\varphi^\ii}(\Gamma)
	\leq\delta_\varphi(\Gamma_0)$.
	
Since $\frac{\varphi+\varphi^\ii}{2}=\overline{\varphi}$, we have
$Q_{\Gamma,\varphi+\varphi^\ii,0}(s)
=Q_{\Gamma,\overline{\varphi},0}(2s)$, and hence
$\delta_{\varphi+\varphi^\ii}(\Gamma)
=\tfrac12\delta_{\overline{\varphi}}(\Gamma)$.
If $\varphi^\ii=\varphi$, then $\overline{\varphi}=\varphi$,
proving the final assertion.
\end{proof}

\section*{Declarations} 

{\bf Ethical Approval}. Not applicable.

{\bf Competing Interest}. Not applicable.

{\bf Funding}. Gongopadhyay acknowledges support from ANRF project: \\ ANRF/ARGM/2025/000122/MTR.

{\bf Authors' Contributions.} All authors have contributed equally. 

{\bf Availability of data and materials}. Not applicable.

\medskip {\bf AI declaration.}
During the preparation of this manuscript, the authors used LLM tools (eg. Chat GPT, Claude)  for language editing, stylistic improvement, and proofreading, including assistance in identifying possible inconsistencies in the exposition. The mathematical direction, results and proofs were developed and independently verified by the authors. The authors reviewed the manuscript and assume full responsibility for its content.


\end{document}